\documentclass[10pt]{amsart}  

\usepackage{amsmath,amstext,amscd, amssymb,txfonts, epsfig, color, epstopdf}
\usepackage{pinlabel}
\usepackage[normalem]{ulem}
\usepackage[colorlinks=true, pdfstartview=FitV, linkcolor=blue, citecolor=blue, urlcolor=blue]{hyperref}
 \usepackage[abs]{overpic}
\usepackage{xcolor,varwidth}
\usepackage{enumerate}
\usepackage{tabularx, multirow}
\usepackage{tikz}
\usepackage{graphicx, scalerel}
\usepackage{enumitem}
 
\makeatletter
\def\@tocline#1#2#3#4#5#6#7{\relax
\ifnum #1>\c@tocdepth 
  \else 
    \par \addpenalty\@secpenalty\addvspace{#2}%
\begingroup \hyphenpenalty\@M
    \@ifempty{#4}{%
      \@tempdima\csname r@tocindent\number#1\endcsname\relax
 }{%
   \@tempdima#4\relax
 }%
 \parindent\z@ \leftskip#3\relax \advance\leftskip\@tempdima\relax
 \rightskip\@pnumwidth plus4em \parfillskip-\@pnumwidth
 #5\leavevmode\hskip-\@tempdima #6\nobreak\relax
 \ifnum#1<0\hfill\else\dotfill\fi\hbox to\@pnumwidth{\@tocpagenum{#7}}\par
 \nobreak
 \endgroup
  \fi}
\makeatother
\let\oldtocsection=\tocsection
\let\oldtocsubsection=\tocsubsection
\let\oldtocsubsubsection=\tocsubsubsection
\renewcommand{\tocsection}[2]{\hspace{0em}\oldtocsection{#1}{#2}}
\renewcommand{\tocsubsection}[2]{\hspace{1em}\oldtocsubsection{#1}{#2}}
\renewcommand{\tocsubsubsection}[2]{\hspace{2em}\oldtocsubsubsection{#1}{#2}}

\definecolor{cerulean}{rgb}{0,.48,.65} 
\definecolor{magenta}{rgb}{.5,0,.5} 
\definecolor{dred}{rgb}{.5,0,0} 
\definecolor{green}{rgb}{0,.5,0} 
\definecolor{blue}{rgb}{0,0,1} 
\definecolor{black}{rgb}{0,0,0} 
\definecolor{dgreen}{rgb}{0,.3,0} 
\definecolor{vdred}{rgb}{.3,0,0} 
\definecolor{red}{rgb}{1,0,0} 
\definecolor{salmon}{rgb}{0.98,0.50,0.45} 
\definecolor{gray}{rgb}{.5,.5,.5} 
\definecolor{seagreen}{rgb}{0.13,0.70,0.67} 
\definecolor{chartreuse}{rgb}{0.40,0.80,0.00}
\definecolor{cornflower}{rgb}{0.39,0.58,0.93} 
\definecolor{gold}{rgb}{0.80,0.68,0.00}
 
\theoremstyle{plain}

\newtheorem{thm}{Theorem}[section]
\newtheorem{lemma}[thm]{Lemma}

\newtheorem{cor}[thm]{Corollary}
\newtheorem{prop}[thm]{Proposition}

\newtheorem*{conjproblem*}{The conjugacy problem}
\newtheorem*{0twistedproblem*}{The 0-twisted-conjugacy problem}
\newtheorem*{Htwistedproblem*}{The H-twisted conjugacy problem}
\newtheorem*{Itwistedproblem*}{The I-twisted conjugacy problem}

\theoremstyle{definition}
\newtheorem{rem}[thm]{Remark}
\newtheorem{defn}[thm]{Definition}

\newtheorem{Open questions}[thm]{Open questions}
\newtheorem{Open question}[thm]{Open question}
\newtheorem{Open problems}[thm]{Open problems}
\newtheorem{Open problem}[thm]{Open problem}

\def\Bbb{\mathbb}
\def\bar{\overline}

\def\Z{\Bbb{Z}}
\def\N{\Bbb{N}}
\def\Naturals{\Bbb{N}}

\def\ni{\noindent}

\def\Area{\hbox{\rm Area}}

\def\CL{\hbox{\rm CL}}

\def\F+L{\hbox{$\textup{F}\!_+\textup{L}$}}

\def\Ann{\hbox{\rm Ann}}

\newcommand{\BS}{\mathrm{BS}}

\def\onto{{\kern3pt\to\kern-8pt\to\kern3pt}}

\def\<{\langle}
\def\>{\rangle}
\def\|{{\ |\ }}

\def\g{\gamma}

\newcommand{\set}[1]{\left\{#1\right\}}

\newcommand{\abs}[1]{\left|#1\right|}
\renewcommand{\ni}{\noindent}

\def\*{^{\star}}

\newcommand{\NN}{\mathbb{N}}

\newcommand{\ZZ}{\mathbb{Z}}

\newcommand{\del}{\partial}

\title[Distinguishing filling invariants associated to conjugacy in groups]{Distinguishing filling invariants associated \\ to conjugacy in groups}

\author{Conan Gillis and Timothy Riley}
\date \today

\begin{document}

\begin{abstract} Brick and Corson introduced  annular Dehn functions in 1998 to quantify the conjugacy problem for finitely generated groups and gave the fundamental relationships between it,  the Dehn function, and the conjugator length function.  We furnish the theory with diverse examples.  In particular, we show that these three invariants are independent---no two of the three functions determine the other.

\ni    

  \smallskip
\ni \footnotesize{\textbf{2020 Mathematics Subject Classification:  20F65, 20F10, 20F06}}  \\ 
\ni \footnotesize{\emph{Key words and phrases:} conjugacy problem,  Conjugator length, Dehn function}
\end{abstract}

\thanks{The authors gratefully acknowledge the support of NSF Grants DGE–2139899 (Gillis) and OIA-2428489 (Riley).}

\maketitle

\section{Introduction}
When a word $u$ represents $1$ in  a finitely presented group $G$, it admits a disk-diagram (a \emph{van~Kampen diagram}) witnessing to how this follows from the defining relations.  When words $u$ and $v$ represent conjugate elements, the corresponding witness is an \emph{annular diagram}.   

The \emph{annular Dehn function} $\Ann: \N \to \N$ of a finitely presented   group $G$ was introduced by Brick and Corson in 1998  \cite{brick_corson_1998}.  They defined $\Ann(n)$ to be the minimal $N$ such that whenever $u$ and $v$ are words that represent conjugate elements of $G$ and whose lengths sum to at most $n$, there is an annular diagram with at most $N$ faces such the inner boundary is labelled by $u$ and the outer boundary by $v$.

As Brick and Corson observed, the annular Dehn function is closely related to two other filling functions.   One is the \emph{Dehn function} $\Area: \N \to \N$ of $G$, which is defined so that $\Area(n)$ is the minimal $N$ such that whenever a word $u$ of length at most $n$ represents $1$ in $G$, it admits a van~Kampen diagram with at most $N$ faces.  The other is the \emph{conjugator length function} $\CL: \N \to \N$ of G, which  is defined so that $\CL(n)$ is the minimal $N$ such that whenever $u$ and $v$ are words that represent conjugate elements of $G$ and whose lengths sum to at most $n$,  there is a word $w$ of length at most $N$ such that $uw=wv$ in $G$.  Equivalently, there is an annular diagram for the pair $u$ and $v$ for which there is a path in the $1$-skeleton of length at most $N$ from the start of $u$ on one of the two boundary components to the start of $v$ on the other.  Section~\ref{prelimimaries} contains more detailed definitions.

Dehn functions have been studied extensively---\cite{Bridson6, brady2007geometry, Gersten, Sapir} are surveys. Conjugator length functions have also received considerable attention---\cite{BRS} is a recent account. Much less is known about annular Dehn functions.  

Brick and Corson \cite{brick_corson_1998} observed the fundamental relationships:

\begin{thm}[Brick and Corson's inequalities]\label{BCineq}
    For any finitely presented group $G$,
\begin{equation} \label{ann by area and cl} 
\Area(n) \ \leq \   \Ann(n) \ \leq \  \Area(2\CL(n)+n)
\end{equation}
and 
\begin{equation} \label{CL by Ann}
\CL(n) \ \leq  \ \frac{n}{2}+M\cdot \Ann(n),
\end{equation}
where $M >0$ is a constant depending on the presentation for $G$.
\end{thm}

In brief,  the second bound in \eqref{ann by area and cl} comes from the observation that for a pair of words $u$ and $v$ that represent conjugate elements of $G$, and for a minimal length conjugator $w$, we can form an annular diagram by identifying the two $w$-labelled sides in a minimal-area van~Kampen diagram for $w^{-1}uw v^{-1}$. The bound in \eqref{CL by Ann} comes from taking $M$ to be the length of the longest defining relation, since then the right-hand side is an upper bound on the number of edges in the interior of an annular diagram.

The purpose of this article is to furnish the theory of annular Dehn functions with examples  which show that it harbors some richness and subtlety: specifically, we show that $\Area(n)$, $\CL(n)$, and $\Ann(n)$ are independent invariants and   \eqref{ann by area and cl} and  \eqref{CL by Ann} need not be sharp. These conclusions will be corollaries of the estimates that are summarized in the following theorem.  Our commutator convention is  $[a,b] = aba^{-1}b^{-1}$.  For functions $f,g:\NN\to \NN$, $f \lesssim g$ if there exists a constant $C>0$ such that $f(n)  \leq   Cg(Cn+n)+Cn+C$ for all $n$.  We write $f \simeq g$ when $f\lesssim g$ and $g\lesssim f$.  Up to $\simeq$,  each of $\Ann(n)$, $\CL(n)$, and $\Area(n)$ do not depend on the choice of finite presentation for $G$.

\begin{thm}\label{MainTheorem}
Up to $\simeq$, the following finitely presented groups have the following Dehn functions, conjugator length functions, and annular Dehn functions (for all $d,m \geq 1$):

\begin{center}
\begin{tabular}{lccc}  
              & $\Area(n)$ &   $\Ann(n)$   &   $\CL(n)$   \\
\hline
$G_1 = \mathcal{H}_3(\ZZ)$ &$n^3$&$n^4$&$n^2$\\
$G_2 = \BS(1,2)$ &$2^n$&$2^n$&$n$\\
$G_3 = G_1\times G_2$ &$2^n$&$2^n$&$n^2$\\
$G_4$ & $2^n$&$2^{n^2}$&$n^2$\\
$G_{5,d} = \Z^d \rtimes \Z$  filiform    & $n^{d+1}$   &  &  $n^d$ \\
$G_{6,m}$      &  $n^{3}$  &  & $n^{m+1}$ \\
$G_{7} = G_{5,3} \ast  G_{6,20}$ & $n^4$ & \multirow{2}{*}{ $ \left.   \rule{0mm}{4mm}   \right\rbrace \simeq$}    & $n^{21}$ \\ 
$G_{8} = G_{5,4} \ast  G_{6,20}$ & $n^5$ &  & $n^{21}$  \\
   \end{tabular}
\end{center}
where  
$$\begin{array}{ll}
    G_1 &  \!\!\! = \ \langle a ,b,c\mid [a,c], [b,c], [a,b]c^{-1} \rangle \\
    G_2 &  \!\!\! = \ \langle a,s\mid  asa^{-1} s^{-2} \rangle \\  
    G_4 &  \!\!\! = \ \langle a,b,c,d,s\mid [a,b]c^{-1}, [a,c], [b,c], [b,d],  asa^{-1}s^{-2},  dsd^{-1}   s^{-2} \rangle \\ 
    G_{5,d} & \!\!\! = \ \langle a_1,\ldots, a_d,t\mid  ta_it^{-1}(a_ia_{i-1})^{-1}\ \ \forall i>1, \ [a_1, t], \ [a_i,a_j] \ \ \forall i\neq j \rangle
\end{array}$$
and $G_{6,m}$  is per Definition \ref{G6m} below. 
\end{thm}

To be clear, the assertions made in the table above include that $\Ann_{G_7}(n)\simeq \Ann_{G_8}(n)$, but we do not determine their growth rates. We also do not determine $\Ann_{G_{5,d}}(n)$ or $\Ann_{G_{6,m}}(n)$.  

The groups populating the table above are a mixture of the well-known and the novel.  

The Heisenberg group $G_1$, the Baumslag-Solitar group $G_2$, and the filiform groups $G_{5,d}$ are well-known. The groups $G_{6,m}$ are bespoke constructions of 2-step nilpotent groups  designed in \cite{BrRi1} to have having conjugator length functions growing like $n^{m+1}$:

\begin{defn}\label{G6m} The groups $G_{6,m}$ of \cite{BrRi1}   are,  for all $m \geq 1$,  the central extensions of $\Z^{m+2} = \langle a_1, \ldots, a_m, b_1, b_2\rangle$ by $\Z^m  = \langle c_1, \ldots, c_m \rangle$ defined by 
the relations $b_1a_i=a_ib_1c_i$ for $i=1,\ldots, m$, $b_2a_i=a_ib_2c_{i+1}^{-1}$ for $i=1,\ldots,m-1$, and commutation relations for the pair of generators $(b_2,a_m)$, and all pairs of generators $(a_i,a_j)$, $(b_i,b_j)$, $(c_i,c_j)$, $(b_i,c_j)$, and $(a_i,c_j)$.
\end{defn}

Combining $G_{5,d}$ and $G_{6,m}$ to make $G_7$ and $G_8$ is one of the novelties of this article.   The design of $G_4$ is another: it is an amalgamated free product of an ascending HNN-extension of $\mathcal{H}_3(\ZZ)$ and another amalgamated free product of two copies of $\BS(1,2)$ joined along the exponentially distorted subgroup. 

Proving that $G_4$ enjoys the properties claimed in the table is subtle and occupies, in Section~\ref{graphofgroups_BS12_and_heisenberg}, the most delicate and protracted portion of this article.

The $\BS(1,2)$ constituents of $G_4$ are the source of its exponential Dehn function.  The lower bound on $\Area_{G_4}(n)$ (Corollary~\ref{cor_base}) comes  from a straight-forward observation of a retraction.  The upper bound (Proposition~\ref{Dehn exp upper bound G4})
comes from an analysis of Britton's Lemma applied in the light of detailed estimates concerning words representing elements of  subgroups (Lemmas \ref{length detail} and \ref{area detail in E}).  

The copy of $\mathcal{H}_3(\ZZ)$ within $G_4$ seeds the quadratic conjugator length function.  Again, the lower bound comes from a retraction (Corollary~\ref{cor_base}).  But, achieving the matching upper bound (Proposition~\ref{G4_CL_Upper}) requires particular intricacy.    We first give precise combinatorial arguments (Lemmas~\ref{Aquadratic}--\ref{quadratic bound from stack}) that assume the existence of certain types of conjugators for particular elements of $G_4$ and conclude that there are such conjugators which also admit tight length bounds. (Lemma~\ref{quadratic bound from stack} is an especially delicate example.)  These estimates then  (in Proposition~\ref{G4_CL_Upper}) feed into a corridor analysis of annular diagrams to establish the quadratic upper bound.  

The composition $2^{n^2}$ of the conjugator length function and the Dehn function gives an upper bound on the annular Dehn function of $G_4$ (by Theorem~\ref{BCineq}). This is matched by the same lower bound (Proposition~\ref{lower bound on anng4 prop}) because we can exhibit, for all $n$, pairs of conjugate words of total length $\sim \! n$ for which a careful analysis of corridors establishes that any annular diagram must have at least that area.

The Dehn function estimates for $G_1$ and $G_2$ in our theorem are well-known; those for $G_{5,d}$  and $G_{6,m}$ are from \cite{BridsonPittet, GHR}. 
 The conjugator length estimates for  $G_1$, $G_2$, $G_{5,d}$,  and $G_{6,m}$ are from \cite{BRS, Sale3, BrRi2, BrRi1}. (We discuss these references in more detail in Section~\ref{Proof of main thm}.) The remaining estimates are established here.

We highlight two corollaries of Theorem~\ref{MainTheorem}.

\begin{cor}
The Dehn function, the conjugator length function, and the annular Dehn function are independent invariants for finitely presented groups.  That is, for any two of these invariants, there is a pair of finitely presented groups for which those two invariants agree, but the third differs (all up to $\simeq$).    
\end{cor}

\begin{proof}
By Theorem \ref{MainTheorem}, the groups $G_7$ and $G_8$ have different Dehn functions, but the other two functions the same. Likewise, compare $G_2$ and $G_3$ for conjugator length function, and compare $G_3$ and $G_4$ for annular Dehn function.   
\end{proof}

\begin{cor}
There exist finitely presented groups for which the inequalities  $$\Area(n) \ \stackrel{\textup{i}}{\lesssim} \ \Ann(n) \ \stackrel{\textup{ii}}{\lesssim} \  \Area(n + \CL(n) )$$
of Brick and Corson are  (1) both sharp, (2) both not sharp, (3) \textup{i} is sharp  but \textup{ii} is not, and (4) \textup{ii} is sharp  but \textup{i} is not, where by `sharp' we mean that $\lesssim$ can be replaced $\simeq$.    
\end{cor}

\begin{proof}
Our examples are: (1)  $G_2$,  (2) $G_1$,  (3) $G_3$, and (4)  $G_4$.  
\end{proof}

This article also includes accounts (Propositions~\ref{dirprodannular} and \ref{freeprodannular}) of how Dehn functions, annular Dehn functions, and conjugator length functions behave with respect to taking direct products and free products---operations which feature in our construction of the groups of Theorem~\ref{MainTheorem}.

\emph{The structure of this article.} Section~\ref{prelimimaries} provides preliminaries concerning van Kampen diagrams, annular diagrams, Dehn functions, annular Dehn functions, and conjugator length functions.    It also contains our accounts of direct products and free products.   Section~\ref{Proof of main thm} is our proof of Theorem~\ref{MainTheorem} except that we postpone our proof   that $\Ann_{\mathcal{H}_{3}(\ZZ)}(n) \simeq n^4$ to Section~\ref{heisensec}    and  our proofs concerning $G_4$ to Section~\ref{graphofgroups_BS12_and_heisenberg}.

\subsection*{Acknowledgement}   We are grateful to the referee  for a careful reading and for helpful comments.

\section{Preliminaries}\label{prelimimaries}

\subsection{Diagrams}\label{diags} The following brief account is based on  \cite{lyndon2001combinatorial}, to which we refer the reader for full details. A \emph{(singular disc) diagram}  $\Delta$ is a planar 2-complex that is obtained from a finite 2-complex homeomorphic to the 2-sphere by removing the interior of one face $f_{\infty}$ (``at infinity'').  So the attaching map of the removed face traverses the boundary $\partial \Delta$.  An \emph{annular diagram} $\Omega$ is obtained from such a $\Delta$ by removing the interior of a further face $f_{0}$.  Accordingly, $\partial \Omega$ is the union of the \emph{outer boundary} $\partial \Delta$  and the \emph{inner boundary} traversed by the attaching map of $f_{0}$.  We also allow the degenerate case where  $\Omega = \Delta$, in which case the \emph{inner boundary} is any vertex of $\Omega$.  

Let $G=\langle X \mid R\rangle$ be a finitely presented group. We consider singular disc diagrams and annular diagrams $\Delta$ such that all edges are directed and are labeled with elements of $X$. For every directed edge $e$ from vertex $u$ to vertex $v$, labeled by $x\in X$, we consider $e^{-1}$ to denote the  ``inverse'' edge going from $v$ to $u$ and labeled by $x^{-1}$. An edge-path $p$ in $\Delta$ is a sequence of edges $(e_1^{\delta_1}, \ldots, e_n^{\delta_n})$  with $\delta_i\in \{\pm1\}$ for all $i$  such that the terminal vertex of $e_i^{\delta_i}$ equals the starting vertex of $e_{i+1}^{\delta_{i+1}}$ for  $i=1, \ldots, n-1$. The \emph{word along $p$} is the word over $X$  read  off the labels of the $e_i^{\delta_i}$.

We say that a singular disc diagram $\Delta$  is  a \emph{van~Kampen diagram} for a word $w$  over $X^{\pm1}$ read around $\partial \Delta$ (from some starting vertex) if for every face $F$ of $\Delta$, the word read around $\del F$ (from some starting vertex on $F$) is in $R^{\pm1}$.

An annular diagram $\Omega$  \emph{for the pair of words $u$ and $v$ over $X^{\pm1}$} is when $\Omega$ is likewise labelled so that the words read around the two boundary components  (from some starting vertices) are $u$ and $v$, and the words read around the faces are in $R^{\pm1}$.  A van~Kampen diagram for a word $w$ is an annular diagram for the pair $w$ and the empty word.

The significance of these diagrams is apparent from the following results, which can be found in \cite{lyndon2001combinatorial}---see Proposition III.9.2 and Lemmas V.5.1 and V.5.2---and in Section~3 of \cite{BRS}.

\begin{lemma}[van~Kampen's Lemma]\label{VanKampensLemma} Let $G=\langle X\mid R\rangle$ be a finitely presented group. A word $w$ over $X$ represents the identity in $G$ if and only if there exists a van~Kampen diagram for $w$.
\end{lemma}

\begin{lemma}\label{annulardiagramlemma}
    Let $G=\langle X\mid R\rangle$ be a finitely presented group, and $u$ and $v$ words over $X$. Then $u$ and $v$ represent conjugate elements in $G$ if and only if there exists an annular diagram $\Omega$ for $u$ and $v$. Moreover, let $p_1$ and $p_2$ be vertices on the outer and inner boundaries (respectively) of $\Omega$  from which we read $u$ and $v$ (respectively). If $\gamma$ is an  edge-path from $p_1$ to $p_2$, and $w$ is the word along $\gamma$, then $wuw^{-1}=v$ in $G$.
\end{lemma}

Suppose  $G=\langle X\mid R\rangle$ is  a finitely presented group and $x \in X$. Suppose that, for all $r\in R$, either $r$ contains no $x$-letters or contains exactly one $x$ and one $x^{-1}$. Then, every face in a (van~Kampen or annular) diagram $\Delta$ will have either two or zero edges labeled $x$.  With the exception of cells adjacent to the boundary of $\Delta$, each cell with two $x$-edges must be adjacent to other such cells, thereby giving rise to a sequence of these cells which we call an \emph{$x$-corridor}---see Figure~\ref{fig:x-corrdior};  $\cdots \beta_{r_{i}}^{-1}\beta_{r_{i+1}}^{-1}\beta_{r_{i+2}}^{-1}\beta_{r_{i+3}}^{-1}\beta_{r_{i+4}}^{-1} \cdots$ and  $\cdots \alpha_{r_{i}}\alpha_{r_{i+1}}\alpha_{r_{i+2}}\alpha_{r_{i+3}}\alpha_{r_{i+4}} \cdots$, as shown in the figure, are the \emph{words along its sides}.

\begin{figure}
    \centering

\tikzset{every picture/.style={line width=0.75pt}} 

\begin{tikzpicture}[x=0.75pt,y=0.75pt,yscale=-1,xscale=1]

\draw    (100,100) -- (128,100) ;
\draw [shift={(130,100)}, rotate = 180] [color={rgb, 255:red, 0; green, 0; blue, 0 }  ][line width=0.75]    (10.93,-3.29) .. controls (6.95,-1.4) and (3.31,-0.3) .. (0,0) .. controls (3.31,0.3) and (6.95,1.4) .. (10.93,3.29)   ;
\draw    (130,100) -- (178,100) ;
\draw [shift={(180,100)}, rotate = 180] [color={rgb, 255:red, 0; green, 0; blue, 0 }  ][line width=0.75]    (10.93,-3.29) .. controls (6.95,-1.4) and (3.31,-0.3) .. (0,0) .. controls (3.31,0.3) and (6.95,1.4) .. (10.93,3.29)   ;
\draw    (180,100) -- (228,100) ;
\draw [shift={(230,100)}, rotate = 180] [color={rgb, 255:red, 0; green, 0; blue, 0 }  ][line width=0.75]    (10.93,-3.29) .. controls (6.95,-1.4) and (3.31,-0.3) .. (0,0) .. controls (3.31,0.3) and (6.95,1.4) .. (10.93,3.29)   ;
\draw    (230,100) -- (278,100) ;
\draw [shift={(280,100)}, rotate = 180] [color={rgb, 255:red, 0; green, 0; blue, 0 }  ][line width=0.75]    (10.93,-3.29) .. controls (6.95,-1.4) and (3.31,-0.3) .. (0,0) .. controls (3.31,0.3) and (6.95,1.4) .. (10.93,3.29)   ;
\draw    (280,100) -- (328,100) ;
\draw [shift={(330,100)}, rotate = 180] [color={rgb, 255:red, 0; green, 0; blue, 0 }  ][line width=0.75]    (10.93,-3.29) .. controls (6.95,-1.4) and (3.31,-0.3) .. (0,0) .. controls (3.31,0.3) and (6.95,1.4) .. (10.93,3.29)   ;
\draw    (350,150) -- (349.96,130) ;
\draw    (100,100) -- (100,128) ;
\draw [shift={(100,130)}, rotate = 270] [color={rgb, 255:red, 0; green, 0; blue, 0 }  ][line width=0.75]    (10.93,-3.29) .. controls (6.95,-1.4) and (3.31,-0.3) .. (0,0) .. controls (3.31,0.3) and (6.95,1.4) .. (10.93,3.29)   ;
\draw    (300,150.1) -- (299.96,130.1) ;
\draw    (150,100) -- (150,128) ;
\draw [shift={(150,130)}, rotate = 270] [color={rgb, 255:red, 0; green, 0; blue, 0 }  ][line width=0.75]    (10.93,-3.29) .. controls (6.95,-1.4) and (3.31,-0.3) .. (0,0) .. controls (3.31,0.3) and (6.95,1.4) .. (10.93,3.29)   ;
\draw    (250,150.19) -- (250,130) ;
\draw    (200,100) -- (200,128) ;
\draw [shift={(200,130)}, rotate = 270] [color={rgb, 255:red, 0; green, 0; blue, 0 }  ][line width=0.75]    (10.93,-3.29) .. controls (6.95,-1.4) and (3.31,-0.3) .. (0,0) .. controls (3.31,0.3) and (6.95,1.4) .. (10.93,3.29)   ;
\draw    (330,100) -- (350,100) ;
\draw    (349.96,149.52) -- (321.96,149.57) ;
\draw [shift={(319.96,149.58)}, rotate = 359.89] [color={rgb, 255:red, 0; green, 0; blue, 0 }  ][line width=0.75]    (10.93,-3.29) .. controls (6.95,-1.4) and (3.31,-0.3) .. (0,0) .. controls (3.31,0.3) and (6.95,1.4) .. (10.93,3.29)   ;
\draw    (319.96,149.58) -- (271.96,149.67) ;
\draw [shift={(269.96,149.67)}, rotate = 359.89] [color={rgb, 255:red, 0; green, 0; blue, 0 }  ][line width=0.75]    (10.93,-3.29) .. controls (6.95,-1.4) and (3.31,-0.3) .. (0,0) .. controls (3.31,0.3) and (6.95,1.4) .. (10.93,3.29)   ;
\draw    (269.96,149.67) -- (221.96,149.77) ;
\draw [shift={(219.96,149.77)}, rotate = 359.89] [color={rgb, 255:red, 0; green, 0; blue, 0 }  ][line width=0.75]    (10.93,-3.29) .. controls (6.95,-1.4) and (3.31,-0.3) .. (0,0) .. controls (3.31,0.3) and (6.95,1.4) .. (10.93,3.29)   ;
\draw    (219.96,149.77) -- (171.96,149.86) ;
\draw [shift={(169.96,149.87)}, rotate = 359.89] [color={rgb, 255:red, 0; green, 0; blue, 0 }  ][line width=0.75]    (10.93,-3.29) .. controls (6.95,-1.4) and (3.31,-0.3) .. (0,0) .. controls (3.31,0.3) and (6.95,1.4) .. (10.93,3.29)   ;
\draw    (169.96,149.87) -- (121.96,149.96) ;
\draw [shift={(119.96,149.96)}, rotate = 359.89] [color={rgb, 255:red, 0; green, 0; blue, 0 }  ][line width=0.75]    (10.93,-3.29) .. controls (6.95,-1.4) and (3.31,-0.3) .. (0,0) .. controls (3.31,0.3) and (6.95,1.4) .. (10.93,3.29)   ;
\draw    (119.96,149.96) -- (99.96,150) ;
\draw    (250,100) -- (250,128) ;
\draw [shift={(250,130)}, rotate = 270] [color={rgb, 255:red, 0; green, 0; blue, 0 }  ][line width=0.75]    (10.93,-3.29) .. controls (6.95,-1.4) and (3.31,-0.3) .. (0,0) .. controls (3.31,0.3) and (6.95,1.4) .. (10.93,3.29)   ;
\draw    (300,100) -- (300,128) ;
\draw [shift={(300,130)}, rotate = 270] [color={rgb, 255:red, 0; green, 0; blue, 0 }  ][line width=0.75]    (10.93,-3.29) .. controls (6.95,-1.4) and (3.31,-0.3) .. (0,0) .. controls (3.31,0.3) and (6.95,1.4) .. (10.93,3.29)   ;
\draw    (350,100) -- (350,128) ;
\draw [shift={(350,130)}, rotate = 270] [color={rgb, 255:red, 0; green, 0; blue, 0 }  ][line width=0.75]    (10.93,-3.29) .. controls (6.95,-1.4) and (3.31,-0.3) .. (0,0) .. controls (3.31,0.3) and (6.95,1.4) .. (10.93,3.29)   ;
\draw    (200,150.29) -- (200,130) ;
\draw    (150,150.38) -- (150,130) ;
\draw    (100,150.48) -- (100,130) ;

\draw (102,118.4) node [anchor=north west][inner sep=0.75pt]  [font=\small]  {$x$};
\draw (152,118.4) node [anchor=north west][inner sep=0.75pt]  [font=\small]  {$x$};
\draw (202,118.4) node [anchor=north west][inner sep=0.75pt]  [font=\small]  {$x$};
\draw (252,118.4) node [anchor=north west][inner sep=0.75pt]  [font=\small]  {$x$};
    \draw (302,118.4) node [anchor=north west][inner sep=0.75pt]  [font=\small]  {$x$};
\draw (352,118.4) node [anchor=north west][inner sep=0.75pt]  [font=\small]  {$x\ \ \ \ \cdots $};
\draw (111,82.4) node [anchor=north west][inner sep=0.75pt]  [font=\small]  {$\alpha _{r_{i}}$};
\draw (161,82.4) node [anchor=north west][inner sep=0.75pt]  [font=\small]  {$\alpha _{r_{i+1}}$};
\draw (211,82.4) node [anchor=north west][inner sep=0.75pt]  [font=\small]  {$\alpha _{r_{i+2}}$};
\draw (261,82.4) node [anchor=north west][inner sep=0.75pt]  [font=\small]  {$\alpha _{r_{i+3}}$};
\draw (311,82.4) node [anchor=north west][inner sep=0.75pt]  [font=\small]  {$\alpha _{r_{i+4}}$};
\draw (111,152.4) node [anchor=north west][inner sep=0.75pt]  [font=\small]  {$\beta _{r_{i}}{}$};
\draw (161,152.4) node [anchor=north west][inner sep=0.75pt]  [font=\small]  {$\beta _{r_{i+1}}{}$};
\draw (211,152.4) node [anchor=north west][inner sep=0.75pt]  [font=\small]  {$\beta _{r_{i+2}}$};
\draw (261,152.4) node [anchor=north west][inner sep=0.75pt]  [font=\small]  {$\beta _{r_{i+3}}$};
\draw (311,152.4) node [anchor=north west][inner sep=0.75pt]  [font=\small]  {$\beta _{r_{i+4}}$};
\draw (53.67,118.4) node [anchor=north west][inner sep=0.75pt]  [font=\small]  {$\cdots $};

\end{tikzpicture}

    \caption{A general $x$-corridor}
    \label{fig:x-corrdior}
\end{figure}
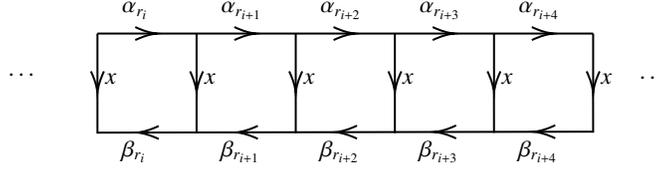

If an $x$-corridor has no cells adjacent to the boundary, and the words along both sides are the same, then it may be excised from the diagram without increasing the area, as shown in Figure \ref{fig:same_sides_excise_1} (depicting the case of an $x$-corridor in an annular diagram). Roughly speaking this excision is done by removing the annulus corridor and then identifying the paths along its two boundaries.  (A technical concern here is that this identification may break the planarity of the diagram.  How to navigate this issue is explained in Section~3.3 of \cite{BRS}.)

There are some important differences between $x$-corridors in van~Kampen diagrams and in annular diagrams. In a van~Kampen diagram $\Delta$, an $x$-corridor must either form a closed loop, which itself bounds a contractible subcomplex of $\Delta$, or it must run from the boundary of $\Delta$ to itself. However, for an annular diagram $\Omega$, there are two boundary components and $x$-corridor come in four types:
\begin{enumerate}
    \item\label{arch} \textit{$x$-arches} run from one boundary component of $\Omega$ to the same component,
    \item\label{radial} \textit{radial $x$-corridors} run from one boundary component of $\Omega$ to the other,
    \item\label{contractible} \textit{contractible $x$-rings} form annuli bounding \textit{contractible} subcomplexes of $\Omega$, and
    \item\label{noncontractible} \textit{non-contractible $x$-rings} form an annuli not bounding \textit{contractible} subcomplexes.
\end{enumerate}

\begin{figure}
    \centering

$$\begin{matrix}
\tikzset{every picture/.style={line width=0.75pt}} 

\begin{tikzpicture}[x=0.5pt,y=0.5pt,yscale=-1,xscale=1]

\draw  [line width=1.5]  (180,150) .. controls (180,83.73) and (233.73,30) .. (300,30) .. controls (366.27,30) and (420,83.73) .. (420,150) .. controls (420,216.27) and (366.27,270) .. (300,270) .. controls (233.73,270) and (180,216.27) .. (180,150) -- cycle ;
\draw  [line width=1.5]  (280,150) .. controls (280,138.95) and (288.95,130) .. (300,130) .. controls (311.05,130) and (320,138.95) .. (320,150) .. controls (320,161.05) and (311.05,170) .. (300,170) .. controls (288.95,170) and (280,161.05) .. (280,150) -- cycle ;
\draw   (238.75,150) .. controls (238.75,116.17) and (266.17,88.75) .. (300,88.75) .. controls (333.83,88.75) and (361.25,116.17) .. (361.25,150) .. controls (361.25,183.83) and (333.83,211.25) .. (300,211.25) .. controls (266.17,211.25) and (238.75,183.83) .. (238.75,150) -- cycle ;
\draw   (220,150) .. controls (220,105.82) and (255.82,70) .. (300,70) .. controls (344.18,70) and (380,105.82) .. (380,150) .. controls (380,194.18) and (344.18,230) .. (300,230) .. controls (255.82,230) and (220,194.18) .. (220,150) -- cycle ;
\draw  [draw opacity=0][dash pattern={on 0.84pt off 2.51pt}] (350,150) .. controls (350,150) and (350,150) .. (350,150) .. controls (350,177.61) and (327.61,200) .. (300,200) .. controls (272.39,200) and (250,177.61) .. (250,150) .. controls (250,122.39) and (272.39,100) .. (300,100) .. controls (318.51,100) and (334.66,110.05) .. (343.31,125) -- (300,150) -- cycle ; \draw  [dash pattern={on 0.84pt off 2.51pt}] (350,150) .. controls (350,150) and (350,150) .. (350,150) .. controls (350,177.61) and (327.61,200) .. (300,200) .. controls (272.39,200) and (250,177.61) .. (250,150) .. controls (250,122.39) and (272.39,100) .. (300,100) .. controls (318.51,100) and (334.66,110.05) .. (343.31,125) ;  
\draw  [dash pattern={on 0.84pt off 2.51pt}]  (343.31,125) .. controls (344.27,126.88) and (349.04,137.89) .. (349.89,148.06) ;
\draw [shift={(350,150)}, rotate = 268.21] [color={rgb, 255:red, 0; green, 0; blue, 0 }  ][line width=0.75]    (10.93,-3.29) .. controls (6.95,-1.4) and (3.31,-0.3) .. (0,0) .. controls (3.31,0.3) and (6.95,1.4) .. (10.93,3.29)   ;
\draw  [draw opacity=0][dash pattern={on 0.84pt off 2.51pt}] (390,150) .. controls (390,150) and (390,150) .. (390,150) .. controls (390,150) and (390,150) .. (390,150) .. controls (390,199.71) and (349.71,240) .. (300,240) .. controls (250.29,240) and (210,199.71) .. (210,150) .. controls (210,100.29) and (250.29,60) .. (300,60) .. controls (341.65,60) and (376.69,88.29) .. (386.95,126.7) -- (300,150) -- cycle ; \draw  [dash pattern={on 0.84pt off 2.51pt}] (390,150) .. controls (390,150) and (390,150) .. (390,150) .. controls (390,150) and (390,150) .. (390,150) .. controls (390,199.71) and (349.71,240) .. (300,240) .. controls (250.29,240) and (210,199.71) .. (210,150) .. controls (210,100.29) and (250.29,60) .. (300,60) .. controls (341.65,60) and (376.69,88.29) .. (386.95,126.7) ;  
\draw  [dash pattern={on 0.84pt off 2.51pt}]  (386.95,126.7) .. controls (387.92,128.58) and (389.47,138.09) .. (389.93,148.08) ;
\draw [shift={(390,150)}, rotate = 268.21] [color={rgb, 255:red, 0; green, 0; blue, 0 }  ][line width=0.75]    (10.93,-3.29) .. controls (6.95,-1.4) and (3.31,-0.3) .. (0,0) .. controls (3.31,0.3) and (6.95,1.4) .. (10.93,3.29)   ;

\draw (325,142.4) node [anchor=north west][inner sep=0.75pt]    {$w$};
\draw (395,142.4) node [anchor=north west][inner sep=0.75pt]    {$w$};

\end{tikzpicture}
&&&&

\tikzset{every picture/.style={line width=0.75pt}} 

\begin{tikzpicture}[x=0.5pt,y=0.5pt,yscale=-1,xscale=1]

\draw  [line width=1.5]  (180,150) .. controls (180,83.73) and (233.73,30) .. (300,30) .. controls (366.27,30) and (420,83.73) .. (420,150) .. controls (420,216.27) and (366.27,270) .. (300,270) .. controls (233.73,270) and (180,216.27) .. (180,150) -- cycle ;
\draw  [line width=1.5]  (280,150) .. controls (280,138.95) and (288.95,130) .. (300,130) .. controls (311.05,130) and (320,138.95) .. (320,150) .. controls (320,161.05) and (311.05,170) .. (300,170) .. controls (288.95,170) and (280,161.05) .. (280,150) -- cycle ;
\draw   (234.38,150) .. controls (234.38,113.76) and (263.76,84.38) .. (300,84.38) .. controls (336.24,84.38) and (365.63,113.76) .. (365.63,150) .. controls (365.63,186.24) and (336.24,215.63) .. (300,215.63) .. controls (263.76,215.63) and (234.38,186.24) .. (234.38,150) -- cycle ;
\draw  [draw opacity=0][dash pattern={on 0.84pt off 2.51pt}] (355,150) .. controls (355,180.38) and (330.38,205) .. (300,205) .. controls (269.62,205) and (245,180.38) .. (245,150) .. controls (245,119.62) and (269.62,95) .. (300,95) .. controls (320.36,95) and (338.13,106.06) .. (347.64,122.5) -- (300,150) -- cycle ; \draw  [dash pattern={on 0.84pt off 2.51pt}] (355,150) .. controls (355,180.38) and (330.38,205) .. (300,205) .. controls (269.62,205) and (245,180.38) .. (245,150) .. controls (245,119.62) and (269.62,95) .. (300,95) .. controls (320.36,95) and (338.13,106.06) .. (347.64,122.5) ;  
\draw  [dash pattern={on 0.84pt off 2.51pt}]  (347.64,122.5) .. controls (348.6,124.38) and (353.96,137.6) .. (354.88,148.04) ;
\draw [shift={(355,150)}, rotate = 268.21] [color={rgb, 255:red, 0; green, 0; blue, 0 }  ][line width=0.75]    (10.93,-3.29) .. controls (6.95,-1.4) and (3.31,-0.3) .. (0,0) .. controls (3.31,0.3) and (6.95,1.4) .. (10.93,3.29)   ;

\draw (330,139.4) node [anchor=north west][inner sep=0.75pt]    {$w$};

\end{tikzpicture}

\end{matrix}$$
    \caption{Excising an $x$-corridor}
    \label{fig:same_sides_excise_1}
\end{figure}
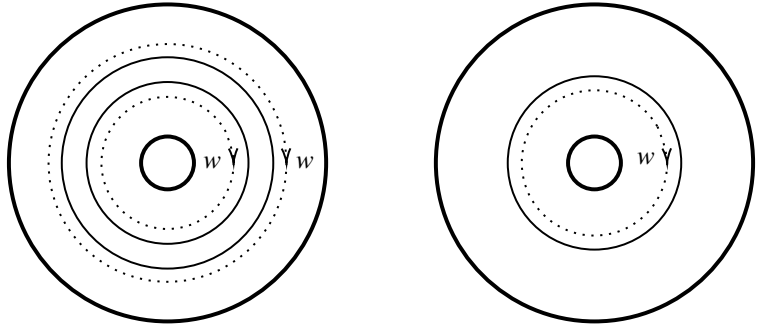

\begin{defn}\label{reduced} A diagram is \textit{reduced} if there are no cells $C_1,C_2$ sharing an edge $E$ such that the words along their boundaries, starting at $E$, are inverses. Two cells $C_1,C_2$ satisfying this property are called \textit{canceling cells}.  Every word that represents $1$ in a finitely presented group admits a reduced van~Kampen diagram over that presentations.  Every pair of words that represent conjugate elements admits a reduced annular diagram.
\end{defn}

\subsection{Dehn functions}\label{Dehnfn} 

We define $\Area(\Delta)$ to be the number of faces in a van~Kampen diagram $\Delta$.  For a word  $w$ on $X$ representing the identity of $G=\langle X\mid R\rangle$, we define $\Area(w)$ to be the minimum of $\Area(\Delta)$ among all van~Kampen diagrams for $w$.  A word $w$ on $X$ represents the identity in $G$ if and only if $w$ freely equals a product of the form $$\prod_{i=1}^kw_ir_i^{\delta_i}w_i^{-1},$$ where $r_i\in R$, $\delta_i\in \{\pm1\}$, and $w_i$ is a word on $X$ for all $i$.  It follows from the standard proof of van~Kampen's Lemma that for any word $w$ representing the identity in $G$, $\Area(w)$ is also the smallest $k$ such that $w$ freely equals a word of this form. 

Define the \emph{Dehn  function} $\Area : \N \to \N$ of $G$ relative to  a finite presentation $\langle X \mid R \rangle$ for $G$ by $$\Area(n) = \max  \set{ \Area(w) : |w|\leq n \text{ and } w=1 \text{ in  } G}.$$ 
See, for example, \cite{Bridson6, BrH, brady2007geometry} for  more details.

\subsection{Conjugator length functions and annular Dehn functions}\label{CL_and_ann}

Suppose $G$ is a group with finite generating set $X$.  When words $u$ and $v$   on $X$   represent conjugate elements in $G$, we write $u\sim v$ and we define $\CL(u,v)$ to be  $\min_{}|\g|$ over all   words $\gamma$ such that $\g u\g^{-1}=v$ in $G$.  

Suppose now that $G=\langle X\mid R\rangle$ is a finitely presented group. By Lemma 3.4 of  \cite{BRS},  $\CL(u,v)$ is equivalently the minimal length $L$ such that there is a annular diagram  $\Omega$ for $u$ and $v$, as per  Lemma~\ref{annulardiagramlemma}, for which there is an edge-path of length $L$ from $p_1$ to $p_2$.  

Define $\Ann(u,v)$ to be the minimum of  $\Area(\Omega)$  over all over  all annular diagrams  $\Omega$ for $u$ and $v$, or similarly equivalently,  the minimum of $\Area(\g u\g^{-1}v^{-1})$  over all words $\gamma$ such that $\g u\g^{-1}=v$ in $G$.     

A priori, a diagram witnessing $\CL(u,v)$ may not witness $\Ann(u,v)$.

Define the \emph{conjugator length function} $\CL : \N \to \N$ and the \emph{annular Dehn function} $\Ann~:~\N \to~\N$ of $G$ by 
$$\begin{array}{rl}
\CL(n)  \!\!\! & = \ \max \CL(u,v) \\ \Ann(n)  \!\!\! & = \ \max \Ann(u,v),
\end{array}$$
where both maxima are taken over all pairs of words $u$ and $v$ such that $|u|+|v|\leq n$ and $u$ and $v$ represent conjugate elements in $G$. 

Up to $\simeq$,   $\Area(n)$ and $\Ann(n)$ do not depend on the choice of finite presentation for $G$, and $\CL(n)$ does not depend on the choice of finite generating set.  

 We will call on these two technical results. 
\begin{lemma} \label{triangle ineq for A}
    If $u=v=w$ in $G=\langle X\mid R\rangle$, then $$\Area(uv^{-1})-\Area(vw^{-1}) \ \leq \  \Area(uw^{-1}) \ \leq \   \Area(uv^{-1})+\Area(vw^{-1}).$$
\end{lemma}
\begin{proof}
    The second inequality follows from the fact that $uw^{-1}$ freely equals $uv^{-1}vw^{-1}$. The first follows from the second after interchanging the roles of $v$ and $w$.
\end{proof}

\begin{lemma}\label{Gluing van Kampen to annular diagram}
    If $u \sim v$ in $G$ and $v=w$ in $G$, then $$\Ann(u,v)-\Area(vw^{-1}) \ \leq \  \Ann(u, w) \ \leq \  \Ann(u,v)+\Area(vw^{-1}).$$
\end{lemma}
\begin{proof}
This follows by a similar argument as the previous lemma, using the fact that $\Ann(u,v)=\Area(\gamma u\gamma ^{-1} v^{-1})$ for some $\gamma$.
\end{proof}

 \subsection{Direct products and free products}
 
The behaviour of $\CL(n)$, $\Ann(n)$ and $\Area(n)$ for direct and free products of finitely generated groups is summarized in the next two propositions, which in large part are from \cite{brick_corson_1998}. We will not call on \eqref{Ann better upper bound}, but we include it here as the natural improvement on the upper bound of \eqref{Ann bad bounds} in circumstances when the optimal conjugator length and annular Dehn function can  be simultaneously realized.

\begin{prop}[]\label{dirprodannular}
Assuming for  \eqref{CL bounds} that $G_1$ and $G_2$ are finitely generated groups,  for  \eqref{Ann bad bounds} and \eqref{Area bounds}   they are finitely presented, and  for \eqref{Area bounds}  that  they are finitely presented and    $\max\{\Area_{G_1}(n),\Area_{G_2}(n)\}\gtrsim n^2$, we have 
    \begin{flalign}   
    \CL_{G_1\times G_2}(n) & \ \simeq \ \max\{\CL_{G_1}(n),\CL_{G_2}(n)\}, \label{CL bounds} \\ 
          \max \{ \Ann_{G_1}(n), \Ann_{G_2}(n) \}    \  \lesssim  \ \Ann_{G_1\times G_2}(n) \hspace*{-13mm} &  \hspace*{13mm}  \ \lesssim \   \max\{ n \Ann_{G_1}(n), n \Ann_{G_2}(n)\}+n^2, \label{Ann bad bounds}  \\
           \Area_{G_1\times G_2}(n) & \ \simeq \ \max\{\Area_{G_1}(n),\Area_{G_2}(n)\}.  \label{Area bounds} 
       \end{flalign}
Further, suppose there exists a constant $D$ such that for all $u,v\in G_i$ with $|u|+|v|\leq n$, there exists $w$ where $w uw^{-1}=v$, $|w|\leq D \, \CL_{G_i}(n)+D$, and $\Area(w uw^{-1}v^{-1})\leq D \, \Ann_{G_i}(n)+D$. Then
           \begin{flalign}   
       \Ann_{G_1\times G_2}(n)  \  & \lesssim \  \Ann_{G_1}(n) +\Ann_{G_2}(n) + n \CL_{G_1}(n) + n \CL_{G_2}(n).  \label{Ann better upper bound}  
       \end{flalign}
\end{prop}

\begin{proof} The proofs of (\ref{CL bounds}) and (\ref{Area bounds}) are straight-forward.      

The first inequality of (\ref{Ann bad bounds}) follows directly from Theorem 2.1 of \cite{brick_corson_1998}. For the second, let $u,v\in G_1\times G_2$ be conjugate words such that $|u|+|v|=n$. Then there exist $u_1,v_1\in G_1$ and $u_2, v_2\in G$  such that $u=u_1u_2$ and $v=v_1v_2$ in $G_1\times G_2$, $u_1$ is conjugate to $v_1$ in $G_1$, and $u_1$ is conjugate to $u_2$ in $G_2$. Without loss of generality we may assume $|u_1|+|u_2|+|v_1|+|v_2|=~|u|+|v|=n$, and that our presentation for $G_1\times G_2$ contains a relator $[g_1,g_2]$ for all generators $g_1$ of $G_1$ and $g_2$ of $G_2$. 

For $i=1,2$, let $w_i$ be a word that conjugates $u_i$ to $v_i$ in $G_i$ and has $\Area(w_i u_i w_i^{-1} v_i^{-1}) =\Ann(u_i, v_i)$. Then $w=w_1w_2$ conjugates  $u$ to $v$. Then $|w_i|\leq M\cdot \Ann_{G_i}(n)+n$, where $M$ is the length of the longest defining relation in our presentations for $G_1$ and $G_2$. 
    We can transform $$w_1w_2 u_1u_2  \ \to \  w_1u_1w_2 u_2  \ \to \  w_1u_1 v_2 w_2 \ \to \  v_1 w_1 v_2  w_2 \ \to \  v_1v_2 w_1 w_2,$$
at a cost of applying at most $|w_2|\cdot |u_1|\leq nM\cdot \Ann_{G_2}(n)+n^2$, then $M\cdot \Ann_{G_2}(n)+n$, then $M\cdot \Ann_{G_1}(n)+n^2$, and then  $|v_2|\cdot|w_1|\leq nM\cdot \Ann_{G_1}(n)+n^2$ defining relations. The total number of relations used gives an upper bound for $\Area(w_iu_iw_i^{-1}v_i^{-1}),$ and thus for $\Ann(u_i,v_i)$, as desired. 

The proof of (\ref{Ann better upper bound}) is similar.
\end{proof}

We now turn to an analogue of Proposition~\ref{dirprodannular} for free products, for which we will use the following notation.
\begin{defn}\label{subnegClosure}
    For a finitely presented group $G$, define $\overline{\Ann_{G}}$ to be the function $\Naturals\to\Naturals$ such that $\overline{\Ann_{G}}(n)$ is the minimal integer for which $\overline{\Ann_{G}}(n) \geq \Ann_G(n)$  and  $\overline{\Ann_{G}}(n)\geq \overline{\Ann_{G}}(n_1)+\overline{\Ann_{G}}(n_2)$ for all integers $n_1, n_2 \geq 1$ such that $n_1 +n_2 =n$.  
\end{defn}

\begin{prop}[]\label{freeprodannular}
Suppose  for  \eqref{free prod CL bounds}  that $G_1$ and $G_2$ are finitely generated groups,   for \eqref{free prod Ann bounds} and \eqref{free prod  Area bounds}  that they are finitely presented. 

Then, for $i=1,2$, 
    \begin{flalign}   
    \CL_{G_1 \ast G_2}(n) & \ \simeq \ \max\{\CL_{G_1}(n),\CL_{G_2}(n)\}, \label{free prod CL bounds} \\ 
            \Ann_{G_1 \ast G_2}(n)   & \ \simeq \     \max\{ \overline{\Ann_{G_1}}(n), \overline{\Ann_{G_2}}(n)\}, \label{free prod Ann bounds}  \\
           \Area_{G_1 \ast G_2}(n) & \ \simeq \ \max\{\Area_{G_1}(n),\Area_{G_2}(n)\}.  \label{free prod Area bounds} 
       \end{flalign}
\end{prop}
\begin{proof}
For (\ref{free prod CL bounds}) and (\ref{free prod Ann bounds}), the claim is a restatement \cite{brick_corson_1998}[Corollary 3.3]. The upper bound for (\ref{free prod Area bounds}) is straightforward, and the lower bound follows from the existence of retracts $G_1\ast G_2\to G_i$ for $i=1,2$.
\end{proof}

\section{Proof of Theorem~\ref{MainTheorem}} \label{Proof of main thm}

That $G_1= \mathcal{H}_{3}(\ZZ)$ satisfies $\Area_{G_1}(n)\simeq n^3$ is proved in \cite[Chapter~8.1]{epstein1992word}. 
That $\CL_{G_1}(n) \simeq n^2$ is Theorem~4.9 in \cite{BRS}. We postpone proof that $\Ann_{G_1}(n)\simeq  n^4$ to Section~\ref{heisensec}.

The Baumslag-Solitar group $G_2$  is well known to have $\Area_{G_2}(n) \simeq 2^n$ (e.g.\  \cite[Chapter~7.4]{epstein1992word}, \cite[Theorem~8.8]{RileyDehn}).  That $\CL_{G_2}(n)\simeq n$  was proved first by \cite{Sale3}; there is an elementary proof in Section~4.4 of \cite{BRS}.  That $\Ann_{G_2}(n)\simeq 2^n$ then follows from Theorem~\ref{BCineq}.

Our estimates for $G_3$  follow from  those for $G_1$ and $G_2$ by Proposition~\ref{dirprodannular}.  

We postpone the estimates concerning $G_4$ to Section~\ref{graphofgroups_BS12_and_heisenberg}.

That  $\Area_{G_{5,d}}(n) \simeq n^{d+1}$ is  \cite[Theorem 6.3]{BridsonPittet}.  That $\Area_{G_{6,m}}(n) \lesssim n^3$ follows from \cite{GHR} since the groups are class-2 nilpotent.  Inspecting the presentation for $G_{6,m}$ given in Definition~\ref{G6m}, we see that killing all generators other than $a_1$, $b_1$, and $c_1$ retracts $G_{6,m}$ onto a copy of the Heisenberg group $G_1$, and thereby $\Area_{G_{6,m}}(n) \gtrsim n^3$  follows from $\Area_{G_1}(n)\simeq n^3$.  
That $\CL_{G_{5,d}}(n) \simeq n^{d}$ and $\CL_{G_{6,m}}(n) \simeq n^{m+1}$ are the main results of \cite{BrRi2} and \cite{BrRi1}, respectively. 

The computations of $\Area_{G_7}$, $\Area_{G_8}$, $\CL_{G_7}$, and $\CL_{G_8}$ follow via Proposition~\ref{freeprodannular}. Towards $\Ann_{G_7}(n)\simeq \Ann_{G_8}(n)$, using the notation of Definition~\ref{subnegClosure},    
    \begin{flalign}   \label{Ann n21 upper bound}
\overline{\Ann_{G_{6,20}}}(n) & \gtrsim \  \Ann_{G_{6,20}}(n) \ \gtrsim \  n^{21} \  \text{ and } \\  
 n^{21} & \gtrsim \ \max\{\Ann_{G_{5,3}}(n),\Ann_{G_{5,4}}(n)\}. \label{n21bound}
    \end{flalign}   
where the second inequality of \eqref{Ann n21 upper bound} combines    Theorem~\ref{BCineq}\eqref{CL by Ann} and $\CL_{G_{6,20}}(n) \simeq n^{21}$ discussed above, and  \eqref{n21bound} holds because Theorem~\ref{BCineq}\eqref{ann by area and cl} applied to the bounds given above tells us that $\Ann_{G_{5,3}}(n) \lesssim (n^3)^4$ and  $\Ann_{G_{5,4}}(n) \lesssim (n^4)^5$.  Because  $(n+\tilde{n})^{21}\geq n^{21}+\tilde{n}^{21}$ for all $n, \tilde{n} \in \N$,  \eqref{n21bound}  and the minimality of $\overline{\Ann}_{G_{5,d}}(n)$  give $n^{21}\gtrsim \max\left\{\overline{\Ann_{G_{5,3}}}(n),\overline{\Ann_{G_{5,4}}}(n)\right\}$. Since $G_7=G_{5,3}\ast G_{6,20}$ and $G_8=G_{5,4}\ast G_{6,20}$, Proposition~\ref{freeprodannular} then gives $\Ann_{G_7}\simeq \overline{\Ann_{G_{6,20}}}(n)=\Ann_{G_8}$.

Aside from the  estimates postponed to Sections~\ref{heisensec} and \ref{graphofgroups_BS12_and_heisenberg}, this completes our proof of Theorem~\ref{MainTheorem}

\section{The Heisenberg Group}\label{heisensec}

\begin{prop}\label{Heisenberg_upper}
    The annular Dehn function of $G_1$ satisfies $\Ann_{G_1}(n) \simeq n^4$.
\end{prop}

\begin{proof} 
We begin by proving that $\Ann_{G_1}(n) \lesssim n^4$. 
Suppose $n>0$ and that $u$ and $v$ are words on the generators of $$G_1  \ = \  \mathcal{H}_3(\ZZ) \ = \  \langle a ,b,c\mid [a,c], [b,c], [a,b]c^{-1} \rangle$$  such that $|u|+|v| \leq n$. By shuffling the $a$-letters to the front (at the expense of creating a $c$ or $c^{-1}$ whenever an $a$ passes a $b$) and shuffling all the $c$-letters to the end, we can transform $u$ to $a^{\alpha_1}b^{\beta_1}c^{\gamma_1}$ and $v$ to $a^{\alpha_2}b^{\beta_2}c^{\gamma_2}$ where $|\alpha_1|+|\alpha_2|+|\beta_1|+|\beta_2|\leq n$ and $|\gamma_1|+|\gamma_2|\leq n^2$ using at most $2n^3$ defining relations. This is the well-known normal form for $G_1$ (see the discussion in \cite{epstein1992word}), so $\alpha_1,\alpha_2,\beta_1,\beta_2,\gamma_1,\gamma_2$ are uniquely determined by $u$ and $v$.

Since $c$ is central in  $G_1$, if $\hat{w} = a^xb^yc^z$ satisfies $\hat{w} u = v \hat{w}$ in $G_1$, then  $w = a^xb^y$ satisfies $w u  =  v w$ in $G_1$. We follow the strategy of Section~4.3 of \cite{BRS} for finding $x$ and $y$ to suit our purposes.  Computing the normal form of $w u w^{-1}$: \begin{subequations}
\label{all}
\begin{align}
w u w^{-1} & = \  a^xb^ya^{\alpha_1}b^{\beta_1}c^{\gamma_1}b^{-y}a^{-x}\label{eqn0}\\  
& = \  a^xb^ya^{\alpha_1}b^{\beta_1}b^{-y}a^{-x}c^{\gamma_1}\label{eqn1}\\
& = \ a^xa^{\alpha_1}c^{-\alpha_1y}b^yb^{\beta_1}b^{-y}a^{-x}c^{\gamma_1}\label{eqn2}\\
& = \ a^xa^{\alpha_1}c^{-\alpha_1y}b^{\beta_1}a^{-x}c^{\gamma_1}\label{eqn3}\\
& = \ a^xa^{\alpha_1}b^{\beta_1}a^{-x}c^{-\alpha_1y+\gamma_1}\label{eqn4}\\
& = \ a^xa^{\alpha_1}a^{-x}b^{\beta_1}c^{-\alpha_1y+\beta_1x+\gamma_1}\label{eqn5}\\
& = \ a^{\alpha_1}b^{\beta_1}c^{-\alpha_1y+\beta_1x+\gamma_1}\label{eqn6},
\end{align}
\end{subequations}
   we see that  $w u w^{-1} =  v $ in $G_1$ if and only if $\alpha_1 = \alpha_2$, $\beta_1 = \beta_2$, and \begin{equation} \label{eq: y version}-\alpha_1 y + \beta_1 x = \gamma_2 - \gamma_1.\end{equation} If $|\alpha_1|=0=|\beta_1|$, then $u=v$ and we are done; otherwise we have two cases. For the first case, if $|\alpha_1|=|\beta_1|\neq 0$, then $\alpha_1$ divides $\gamma_2-\gamma_1,$ so $x=0,y=(\gamma_2-\gamma_1)/\alpha_1$, satisfies~\eqref{eq: y version} and $|y|\leq 2n^2$. For the second case, suppose $ |\alpha_1|\neq |\beta_1|$. Since there is an automorphism of $G_1$ interchanging $a$ and $b$, we may assume without loss of generality that $|\alpha_1| > |\beta_1|\geq 0$. Let $y' = y +  \lfloor (\gamma_2 - \gamma_1)/\alpha_1 \rfloor$. Then \begin{equation} \label{eq: y' version} -\alpha_1 y' + \beta_1 x = \gamma \end{equation} for some $\gamma$ satisfying $|\gamma| \leq |\alpha_1|$.  Equation \eqref{eq: y' version} has a solution  $(x,y')$ with  $|x|, |y'|  \leq \max \set{|\alpha_1|, |\beta_1|}$  by a quantification of Bezout's identity which can be found in  e.g.\ \cite{BFRT-Diophantine, BrRi2, Kornhauser}. Therefore \eqref{eq: y version} has a solution $(x,y)$ with $$|x|  \leq \max \set{|\alpha_1|, |\beta_1|} \ \text{ and }  \ |y|  \leq \max \set{|\alpha_1|, |\beta_1|} + |(\gamma_2 - \gamma_1)/\alpha_1|.$$  Thus there exist $x$ and $y$ such that $a^xb^yub^{-y}a^{-x}=v$, $|x| \leq n$, and $|y| \leq 2n^2$.

The equality (\ref{eqn1}) holds using $|\gamma_1|(|x|+|y|)\leq Cn^4$ defining relations of the form $[a,c]=1$ and $[b,c]=1$, (\ref{eqn2}) using $|\alpha_1|^2|y| \leq 2n^4$ defining relations of the form $[a,b]c^{-1}=1$, and (\ref{eqn3}) holds freely. Also, (\ref{eqn4}) holds using $|\beta_1||x|^2 \leq n^3$ defining relations of the form $[b,c]=1$, (\ref{eqn5}) holds using $|x||\beta_1|^2 \leq n^3$ defining relations of the form $[a,b]c^{-1}=1$, and (\ref{eqn6}) holds freely. Summing  gives   $\Ann_{G_1}(n) \lesssim n^4$, as desired.

    We now turn to proving that $\Ann_{G_1}(n) \gtrsim n^4$.  For $n \in \N$, let $u_n=b$ and $v_n=b[a^n,b^n]$. Then $|u_n|+|v_n| = 4n+2$   and $u_n$ and $v_n$ are conjugate in $G_1$: 
    \begin{equation} \label{un and uv conjugate}
    v_n   \ = \ a^{n^2}u_na^{-n^2}.
    \end{equation}

By applying  defining relations for $G_1$ approximately $n^3$ times, $v_n$ can be shown to equal $bc^{n^2}$ in $G_1$, so, via Lemma~\ref{Gluing van Kampen to annular diagram}, it suffices to show $\Ann(b,bc^{n^2})\gtrsim n^4$.
   
  Let $\Omega$ be any annular diagram witnessing the conjugacy of $u_n$ and $v_n$, and let $\gamma$ be any path from the starting point of $u_n$ to the starting point of $v_n$ such that $\gamma$ crosses every non-contractible $a$-corridor precisely once. Let $w$ be the word along $\gamma$. By cutting $\Omega$ along $\gamma$, we obtain a van Kampen diagram $\Delta$ for $w b w^{-1} (bc^{n^2})^{-1}$ such that $\Area(\Delta)=\Area(\Omega)$.  Say that an $a$-corridor in $\Delta$ is \emph{positively} (resp.\ \emph{negatively}) \emph{oriented} and {vertical} if it connects an $a$ (resp.\ an $a^{-1}$) in the $w$-portion of $\partial \Delta$ to an $a^{-1}$ (resp.\ an $a$) in the $w^{-1}$-portion.         (There may also be $a$-corridors that connect an $a$ to an $a^{-1}$ that are both in the $w$-portion or are both in the $w^{-1}$-portion, and there may be $a$-corridors in $\Delta$ that form annuli.)       

In Section~4.3 of \cite{BRS}, it is shown that the exponent sum of the  $a$-letters in any word $w$ conjugating $b$ to $bc^{n^2}$  is precisely $n^2$. It follows that there are precisely $n^2$ more positively-oriented than negatively-oriented vertical corridors in $\Delta$.

    There thus exist $n^2$ positively-oriented vertical $a$-corridors $Q_1,Q_2,\ldots,Q_{n^2}$ such that between $Q_i$ and $Q_{i+1}$  there are an equal number of positively-oriented and negatively-oriented  vertical $a$-corridors. Let $\beta_i$ and $\alpha_{i+1}$ be the words along the sides of $Q_i$, respectively, as shown in Figure~\ref{fig: Heisenberg_min_area_diagram}. Because $\gamma$ intersects each non-contractible $a$-corridor only once, we have that, as shown in the figure,  for all $i$, the words along the top and along bottom of $\Delta$ between $Q_{i-1}$ and $Q_{i}$ are the same---we name this word $w_i$.

    \begin{figure}
        \centering
\tikzset{every picture/.style={line width=0.75pt}} 

\begin{tikzpicture}[x=0.5pt,y=0.5pt,yscale=-1,xscale=1]

\draw   (542.5,300) -- (542.5,100) -- (67.5,100) -- (67.5,300) -- cycle ;
\draw   (461.72,100.02) -- (461.33,300.02) -- (420.62,299.98) -- (421.01,99.98) -- cycle ;
\draw    (461.07,140) -- (420.36,139.96) ;
\draw    (461.07,180) -- (420.36,179.96) ;
\draw   (312.43,100) -- (312.05,300) -- (271.34,299.96) -- (271.72,99.96) -- cycle ;
\draw    (311.79,140) -- (271.07,139.96) ;
\draw    (311.79,180) -- (271.07,179.96) ;
\draw   (189.82,100.02) -- (189.44,300.02) -- (149.19,299.98) -- (149.58,99.98) -- cycle ;
\draw    (189.64,140) -- (148.93,139.96) ;
\draw    (189.64,180) -- (148.93,179.96) ;
\draw    (108.21,100) -- (119.79,100) ;
\draw [shift={(121.79,100)}, rotate = 180] [color={rgb, 255:red, 0; green, 0; blue, 0 }  ][line width=0.75]    (10.93,-3.29) .. controls (6.95,-1.4) and (3.31,-0.3) .. (0,0) .. controls (3.31,0.3) and (6.95,1.4) .. (10.93,3.29)   ;
\draw    (108.21,300) -- (119.79,300) ;
\draw [shift={(121.79,300)}, rotate = 180] [color={rgb, 255:red, 0; green, 0; blue, 0 }  ][line width=0.75]    (10.93,-3.29) .. controls (6.95,-1.4) and (3.31,-0.3) .. (0,0) .. controls (3.31,0.3) and (6.95,1.4) .. (10.93,3.29)   ;
\draw    (223.57,100) -- (235.14,100) ;
\draw [shift={(237.14,100)}, rotate = 180] [color={rgb, 255:red, 0; green, 0; blue, 0 }  ][line width=0.75]    (10.93,-3.29) .. controls (6.95,-1.4) and (3.31,-0.3) .. (0,0) .. controls (3.31,0.3) and (6.95,1.4) .. (10.93,3.29)   ;
\draw    (223.57,300) -- (235.14,300) ;
\draw [shift={(237.14,300)}, rotate = 180] [color={rgb, 255:red, 0; green, 0; blue, 0 }  ][line width=0.75]    (10.93,-3.29) .. controls (6.95,-1.4) and (3.31,-0.3) .. (0,0) .. controls (3.31,0.3) and (6.95,1.4) .. (10.93,3.29)   ;
\draw    (359.29,100) -- (370.86,100) ;
\draw [shift={(372.86,100)}, rotate = 180] [color={rgb, 255:red, 0; green, 0; blue, 0 }  ][line width=0.75]    (10.93,-3.29) .. controls (6.95,-1.4) and (3.31,-0.3) .. (0,0) .. controls (3.31,0.3) and (6.95,1.4) .. (10.93,3.29)   ;
\draw    (359.29,300) -- (370.86,300) ;
\draw [shift={(372.86,300)}, rotate = 180] [color={rgb, 255:red, 0; green, 0; blue, 0 }  ][line width=0.75]    (10.93,-3.29) .. controls (6.95,-1.4) and (3.31,-0.3) .. (0,0) .. controls (3.31,0.3) and (6.95,1.4) .. (10.93,3.29)   ;
\draw    (162.5,100) -- (174.07,100) ;
\draw [shift={(176.07,100)}, rotate = 180] [color={rgb, 255:red, 0; green, 0; blue, 0 }  ][line width=0.75]    (10.93,-3.29) .. controls (6.95,-1.4) and (3.31,-0.3) .. (0,0) .. controls (3.31,0.3) and (6.95,1.4) .. (10.93,3.29)   ;
\draw    (162.5,300) -- (174.07,300) ;
\draw [shift={(176.07,300)}, rotate = 180] [color={rgb, 255:red, 0; green, 0; blue, 0 }  ][line width=0.75]    (10.93,-3.29) .. controls (6.95,-1.4) and (3.31,-0.3) .. (0,0) .. controls (3.31,0.3) and (6.95,1.4) .. (10.93,3.29)   ;
\draw    (284.64,100) -- (296.21,100) ;
\draw [shift={(298.21,100)}, rotate = 180] [color={rgb, 255:red, 0; green, 0; blue, 0 }  ][line width=0.75]    (10.93,-3.29) .. controls (6.95,-1.4) and (3.31,-0.3) .. (0,0) .. controls (3.31,0.3) and (6.95,1.4) .. (10.93,3.29)   ;
\draw    (284.64,300) -- (296.21,300) ;
\draw [shift={(298.21,300)}, rotate = 180] [color={rgb, 255:red, 0; green, 0; blue, 0 }  ][line width=0.75]    (10.93,-3.29) .. controls (6.95,-1.4) and (3.31,-0.3) .. (0,0) .. controls (3.31,0.3) and (6.95,1.4) .. (10.93,3.29)   ;
\draw    (433.93,100) -- (445.5,100) ;
\draw [shift={(447.5,100)}, rotate = 180] [color={rgb, 255:red, 0; green, 0; blue, 0 }  ][line width=0.75]    (10.93,-3.29) .. controls (6.95,-1.4) and (3.31,-0.3) .. (0,0) .. controls (3.31,0.3) and (6.95,1.4) .. (10.93,3.29)   ;
\draw    (433.93,300) -- (445.5,300) ;
\draw [shift={(447.5,300)}, rotate = 180] [color={rgb, 255:red, 0; green, 0; blue, 0 }  ][line width=0.75]    (10.93,-3.29) .. controls (6.95,-1.4) and (3.31,-0.3) .. (0,0) .. controls (3.31,0.3) and (6.95,1.4) .. (10.93,3.29)   ;
\draw    (495,100) -- (506.57,100) ;
\draw [shift={(508.57,100)}, rotate = 180] [color={rgb, 255:red, 0; green, 0; blue, 0 }  ][line width=0.75]    (10.93,-3.29) .. controls (6.95,-1.4) and (3.31,-0.3) .. (0,0) .. controls (3.31,0.3) and (6.95,1.4) .. (10.93,3.29)   ;
\draw    (495,300) -- (506.57,300) ;
\draw [shift={(508.57,300)}, rotate = 180] [color={rgb, 255:red, 0; green, 0; blue, 0 }  ][line width=0.75]    (10.93,-3.29) .. controls (6.95,-1.4) and (3.31,-0.3) .. (0,0) .. controls (3.31,0.3) and (6.95,1.4) .. (10.93,3.29)   ;
\draw    (162.5,140) -- (174.07,140) ;
\draw [shift={(176.07,140)}, rotate = 180] [color={rgb, 255:red, 0; green, 0; blue, 0 }  ][line width=0.75]    (10.93,-3.29) .. controls (6.95,-1.4) and (3.31,-0.3) .. (0,0) .. controls (3.31,0.3) and (6.95,1.4) .. (10.93,3.29)   ;
\draw    (162.5,180) -- (174.07,180) ;
\draw [shift={(176.07,180)}, rotate = 180] [color={rgb, 255:red, 0; green, 0; blue, 0 }  ][line width=0.75]    (10.93,-3.29) .. controls (6.95,-1.4) and (3.31,-0.3) .. (0,0) .. controls (3.31,0.3) and (6.95,1.4) .. (10.93,3.29)   ;
\draw    (284.64,140) -- (296.21,140) ;
\draw [shift={(298.21,140)}, rotate = 180] [color={rgb, 255:red, 0; green, 0; blue, 0 }  ][line width=0.75]    (10.93,-3.29) .. controls (6.95,-1.4) and (3.31,-0.3) .. (0,0) .. controls (3.31,0.3) and (6.95,1.4) .. (10.93,3.29)   ;
\draw    (284.64,180) -- (296.21,180) ;
\draw [shift={(298.21,180)}, rotate = 180] [color={rgb, 255:red, 0; green, 0; blue, 0 }  ][line width=0.75]    (10.93,-3.29) .. controls (6.95,-1.4) and (3.31,-0.3) .. (0,0) .. controls (3.31,0.3) and (6.95,1.4) .. (10.93,3.29)   ;
\draw    (433.93,140) -- (445.5,140) ;
\draw [shift={(447.5,140)}, rotate = 180] [color={rgb, 255:red, 0; green, 0; blue, 0 }  ][line width=0.75]    (10.93,-3.29) .. controls (6.95,-1.4) and (3.31,-0.3) .. (0,0) .. controls (3.31,0.3) and (6.95,1.4) .. (10.93,3.29)   ;
\draw    (433.93,180) -- (445.5,180) ;
\draw [shift={(447.5,180)}, rotate = 180] [color={rgb, 255:red, 0; green, 0; blue, 0 }  ][line width=0.75]    (10.93,-3.29) .. controls (6.95,-1.4) and (3.31,-0.3) .. (0,0) .. controls (3.31,0.3) and (6.95,1.4) .. (10.93,3.29)   ;
\draw    (67.5,200) -- (67.5,203) ;
\draw [shift={(67.5,205)}, rotate = 270] [color={rgb, 255:red, 0; green, 0; blue, 0 }  ][line width=0.75]    (10.93,-3.29) .. controls (6.95,-1.4) and (3.31,-0.3) .. (0,0) .. controls (3.31,0.3) and (6.95,1.4) .. (10.93,3.29)   ;
\draw    (148.93,190) -- (148.93,203) ;
\draw [shift={(148.93,205)}, rotate = 270] [color={rgb, 255:red, 0; green, 0; blue, 0 }  ][line width=0.75]    (10.93,-3.29) .. controls (6.95,-1.4) and (3.31,-0.3) .. (0,0) .. controls (3.31,0.3) and (6.95,1.4) .. (10.93,3.29)   ;
\draw    (189.64,190) -- (189.64,203) ;
\draw [shift={(189.64,205)}, rotate = 270] [color={rgb, 255:red, 0; green, 0; blue, 0 }  ][line width=0.75]    (10.93,-3.29) .. controls (6.95,-1.4) and (3.31,-0.3) .. (0,0) .. controls (3.31,0.3) and (6.95,1.4) .. (10.93,3.29)   ;
\draw    (542.5,200) -- (542.5,203) ;
\draw [shift={(542.5,205)}, rotate = 270] [color={rgb, 255:red, 0; green, 0; blue, 0 }  ][line width=0.75]    (10.93,-3.29) .. controls (6.95,-1.4) and (3.31,-0.3) .. (0,0) .. controls (3.31,0.3) and (6.95,1.4) .. (10.93,3.29)   ;
\draw    (271.07,190) -- (271.07,203) ;
\draw [shift={(271.07,205)}, rotate = 270] [color={rgb, 255:red, 0; green, 0; blue, 0 }  ][line width=0.75]    (10.93,-3.29) .. controls (6.95,-1.4) and (3.31,-0.3) .. (0,0) .. controls (3.31,0.3) and (6.95,1.4) .. (10.93,3.29)   ;
\draw    (311.79,190) -- (311.79,203) ;
\draw [shift={(311.79,205)}, rotate = 270] [color={rgb, 255:red, 0; green, 0; blue, 0 }  ][line width=0.75]    (10.93,-3.29) .. controls (6.95,-1.4) and (3.31,-0.3) .. (0,0) .. controls (3.31,0.3) and (6.95,1.4) .. (10.93,3.29)   ;
\draw    (420.36,190) -- (420.36,203) ;
\draw [shift={(420.36,205)}, rotate = 270] [color={rgb, 255:red, 0; green, 0; blue, 0 }  ][line width=0.75]    (10.93,-3.29) .. controls (6.95,-1.4) and (3.31,-0.3) .. (0,0) .. controls (3.31,0.3) and (6.95,1.4) .. (10.93,3.29)   ;
\draw    (461.07,190) -- (461.07,203) ;
\draw [shift={(461.07,205)}, rotate = 270] [color={rgb, 255:red, 0; green, 0; blue, 0 }  ][line width=0.75]    (10.93,-3.29) .. controls (6.95,-1.4) and (3.31,-0.3) .. (0,0) .. controls (3.31,0.3) and (6.95,1.4) .. (10.93,3.29)   ;

\draw (167.96,74.4) node [anchor=north west][inner sep=0.75pt]    {$Q_{1}$};
\draw (549.75,182.4) node [anchor=north west][inner sep=0.75pt]  [font=\small]  { $b$};
\draw (126.18,187.4) node [anchor=north west][inner sep=0.75pt]  [font=\small]  {$\beta _{1}$};
\draw (200.79,187.4) node [anchor=north west][inner sep=0.75pt]  [font=\small]  {$\alpha _{2}$};
\draw (242.89,187.4) node [anchor=north west][inner sep=0.75pt]  [font=\small]  {$\beta _{2}$};
\draw (323.14,187.4) node [anchor=north west][inner sep=0.75pt]  [font=\small]  {$\alpha _{3}$};
\draw (159.69,115.87) node [anchor=north west][inner sep=0.75pt]  [font=\small]  {$a$};
\draw (159.69,147.87) node [anchor=north west][inner sep=0.75pt]  [font=\small]  {$a$};
\draw (159.69,187.87) node [anchor=north west][inner sep=0.75pt]  [font=\small]  {$a$};
\draw (378.93,187.4) node [anchor=north west][inner sep=0.75pt]  [font=\small]  {$\beta _{n^{2} -1}$};
\draw (473.32,187.4) node [anchor=north west][inner sep=0.75pt]  [font=\small]  {$\alpha _{n^{2}}$};
\draw (283.32,74.4) node [anchor=north west][inner sep=0.75pt]    {$Q_{2}$};
\draw (425.54,74.4) node [anchor=north west][inner sep=0.75pt]    {$Q_{n^{2}}$};
\draw (22,182.4) node [anchor=north west][inner sep=0.75pt]  [font=\small]  {$bc{^{n^2}}$};
\draw (228.5,107.4) node [anchor=north west][inner sep=0.75pt]  [font=\small]  {$w_{1}$};
\draw (228.5,272.4) node [anchor=north west][inner sep=0.75pt]  [font=\small]  {$w_{1}$};
\draw (96.86,107.4) node [anchor=north west][inner sep=0.75pt]  [font=\small]  {$w_{0}$};
\draw (92.79,272.4) node [anchor=north west][inner sep=0.75pt]  [font=\small]  {$w_{0}$};
\draw (491.32,271.4) node [anchor=north west][inner sep=0.75pt]  [font=\small]  {$w_{n^{2}}$};
\draw (498.11,107.4) node [anchor=north west][inner sep=0.75pt]  [font=\small]  {$w_{n^{2}}$};
\draw (153.14,227.4) node [anchor=north west][inner sep=0.75pt]    {$\vdots $};
\draw (284.55,117.37) node [anchor=north west][inner sep=0.75pt]  [font=\small]  {$a$};
\draw (284.55,149.37) node [anchor=north west][inner sep=0.75pt]  [font=\small]  {$a$};
\draw (281.83,189.37) node [anchor=north west][inner sep=0.75pt]  [font=\small]  {$a$};
\draw (285.26,207.37) node [anchor=north west][inner sep=0.75pt]    {$\vdots $};
\draw (431.12,115.87) node [anchor=north west][inner sep=0.75pt]  [font=\small]  {$a$};
\draw (431.12,147.87) node [anchor=north west][inner sep=0.75pt]  [font=\small]  {$a$};
\draw (431.12,187.87) node [anchor=north west][inner sep=0.75pt]  [font=\small]  {$a$};
\draw (431.83,205.87) node [anchor=north west][inner sep=0.75pt]    {$\vdots $};
\draw (351,192.4) node [anchor=north west][inner sep=0.75pt]    {$...$};

\end{tikzpicture}
        \caption{A van Kampen diagram $\Delta$ for $wbw^{-1}(bc^{n^2})^{-1}$}
        \label{fig: Heisenberg_min_area_diagram}
    \end{figure}
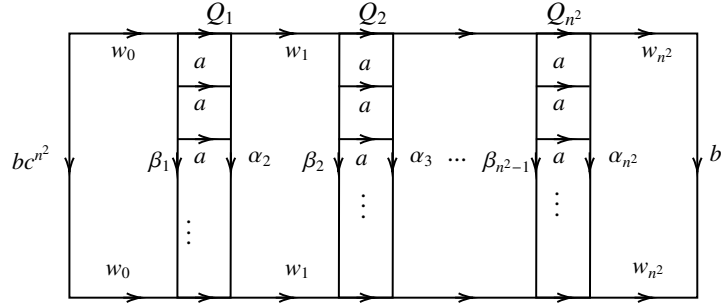

For all $0 \leq i \leq n^2$, the exponent sum  of the $a^{\pm 1}$ in $w_i$ is $0$, so $w_i\in \langle b,c\rangle=\ZZ^2$. This group is abelian, and $\alpha_i$ is a word on $\{b,c\}$, so $w_i$ commutes with $\alpha_i$. Thus, $\beta_i=\alpha_{i-1}$, and proceeding inductively from $\alpha_1=b$ we see that $\alpha_i=\beta_{i-1}^a=bc^i$. Every word on $\{b,c\}$ equal to $bc^i$ has at least $i$ many $c$-letters, therefore $Q_i$ at least $i$ many cells. So $Q_1,\ldots,Q_{n^2}$ have total area at least $1 + 2 + \cdots + (n^2-1) + n^2 = n^2(n^2+1)/2$, giving the desired quartic lower bound.
\end{proof}

\section{Computations of \texorpdfstring{$\Area_{G_4}(n)$, $\CL_{G_4}(n)$, and $\Ann_{G_4}(n)$}{Area, CL, and Ann of G4}}\label{graphofgroups_BS12_and_heisenberg}

 We devote the rest of this paper to computing the invariants of $G_4$. We begin in  Lemma \ref{structureG4} by detailing the structure of $G_4$ as an iterated amalgamated free product built from $\mathcal{H}_3(\ZZ)=\langle a,b,c\rangle_{G_4}, \ZZ^2=\langle b,d\rangle_{G_4},$ and two copies $\langle a,s\rangle_{G_4}$ and $\langle d,s\rangle_{G_4}$ of $\BS(1,2)$.  This structure will then be used in Corollary~\ref{cor_base} to establish the lower bounds $2^n\lesssim \Area_{G_4}(n)$ and $n^2\lesssim \CL_{G_4}(n)$. Lemmas~\ref{length detail}--\ref{area detail A} all aid in proving Proposition~\ref{Dehn exp upper bound G4}, which gives the $\Area_{G_4}(n)\lesssim 2^n$. That $\CL_{G_4}(n)\lesssim n^2$ is proved in Proposition~\ref{G4_CL_Upper} after another chain of supporting lemmas --- we undertake a detailed analysis of diagrams over both $G_4$ and certain subgroups with  Lemma~\ref{quadratic bound from stack} being of particular intricacy.  Once the upper bounds $\Area_{G_4}(n)\lesssim 2^n$ and $\CL_{G_4}(n)\lesssim n^2$ have been established, Theorem~\ref{BCineq} will gives $\Ann_{G_4}(n)\lesssim 2^{n^2}$, so all that will remain will be to prove the corresponding lower bound, which we will do  in Proposition~\ref{lower bound on anng4 prop} by exhibiting conjugate pairs of words for which any annular diagram contains corridors with  $\gtrsim 2^{n^2}$ faces.

\begin{lemma}\label{structureG4}  
The group   $$G_4 \ = \ \langle a,b,c,d,s\mid [a,b]c^{-1}, [a,c], [b,c], [b,d],  asa^{-1} s^{-2},  dsd^{-1}  s^{-2} \rangle$$
is a free product with amalgamation $A *_C B$ of its subgroups  $A=\langle a,b,c,d\rangle$ and $B=\langle a,d,s\rangle$ along the subgroup $C=\langle a,d\rangle$.  It is also an HNN-extension of $$E   \ = \   \langle a, c,d,s \mid  [a,c],   asa^{-1}  s^{-2},  dsd^{-1}    s^{-2} \rangle$$  with stable letter $b$ and an HNN-extension of  $$\langle a,b,c,s\mid [a,b]c^{-1}, [a,c], [b,c],   asa^{-1} s^{-2}  \rangle$$ with stable letter $d$.  Moreover, 
    \begin{enumerate}
        \item \label{A subgroup property} $A = \langle a,b,c,d \mid [a,b]c^{-1}, [a,c], [b,c], [b,d]  \rangle$  is the HNN-extension of $K =  \langle a, c,d  \mid   [a,c]   \rangle \cong \ZZ^2 * \ZZ$ with stable letter $b$ acting via the automorphism of $K$ that maps  $a \mapsto c^{-1}a$,  $c \mapsto c$, and $d \mapsto d$;
        \item  \label{B subgroup property} $B = \langle a, d,s\mid   asa^{-1}  s^{-2},  dsd^{-1}   s^{-2} \rangle$, the amalgamated free product of two copies of $\BS(1,2)$, denoted $D_1=\langle a,s\mid asa^{-1}=s^2\rangle$ and $D_2=\langle d,s\mid dsd^{-1}=s^2\rangle$, along   $\langle s \rangle \cong \Z$.
           \item \label{C subgroup property} $C = \langle a, d \mid \ \rangle$ is free of rank 2.
        \item \label{Heisenberg retract} Killing $d$ and $s$ retracts $G_4$ onto a subgroup $\langle a ,b,c\mid [a,c], [b,c], [a,b]c^{-1} \rangle$, which is the Heisenberg group.
          \item \label{BS retract} Killing $b$ and $c$ and mapping $d \mapsto a$ retracts $G_4$ onto its subgroup $D_1 = \langle a, s \mid   asa^{-1} s^{-2} \rangle$.
        \item \label{L property} The subgroup $L=\langle b,c,s\rangle$ is $\ZZ^2 \ast \ZZ =\langle b,c,s \mid [b,c] \rangle$.      
    \end{enumerate}
\end{lemma}

\begin{proof}
Killing $s$ retracts $G_4$ onto $A$, which therefore has the presentation claimed in \eqref{A subgroup property}, and the remaining claims of  \eqref{A subgroup property} then follow, as does the claim that $G_4$ is an HNN-extension with stable letter $b$.   Similarly killing $b$ and $c$ retracts $G_4$ onto $B$ so as to give  \eqref{B subgroup property}.   Killing $b$ and $c$ retracts $A$ onto $C = \langle a, d \mid \ \rangle$, giving \eqref{C subgroup property}.  That $G_4\cong A *_C B$ follows.  Claims \eqref{Heisenberg retract} and \eqref{BS retract} are straight-forwardly verified.  
Claim \eqref{L property} holds on account of a free-product-with-amalgamation normal form for $A *_C B$, given that  $[b,c]=1$ in $G_4$ and  $\langle b,c \rangle \cap C = \langle s \rangle \cap C = \set{1}$.     It follows  that $b$ and $s$ generate a free subgroup and mapping $b \mapsto b$ and $s \mapsto s^2$ defines an isomorphism between rank-2 free subgroups of $G_4$, and so $G_4$ an HNN-extension with stable letter $d$.  
\end{proof}

\begin{cor}\label{G4Cor} ${}$ 
\vspace*{-2mm}
\begin{enumerate}
    \item\label{ABUndistorted} $A$ and $B$ are undistorted in $G_4$.
    \item\label{conjInA} If $x,y\in A$ are conjugate in $G_4$, then they are conjugate in $A$.
    \item \label{conjInB} If $x,y\in B$ are conjugate in $G_4$, then they are conjugate in $B$.
\end{enumerate}
\end{cor}
 
\begin{proof}
These properties are consequences of   $A$ and $B$ being  retracts of $G_4$.
\end{proof}

\begin{cor}\label{cor_base}
 $2^n\lesssim\Area_{G_4}(n)$ and  $n^2\lesssim \CL_{G_4}(n)$.
\end{cor}
\begin{proof}
    We  noted that $2^n\lesssim \Area_{\BS(1,2)}(n)$  and $ n^2 \lesssim\CL_{\mathcal{H}_3(\ZZ)}(n)$  in Section~\ref{Proof of main thm}.  By \eqref{BS retract} and  \eqref{Heisenberg retract}, respectively, of  Lemma~\ref{structureG4}, $D_1 \cong \BS(1,2)$ and $\mathcal{H}_3(\ZZ)$ are retracts of $G_4$.  The claimed   bounds follow.
\end{proof}

  Next we will prove that $\Area_{G_4}(n) \lesssim 2^n$ by means of the following four lemmas.  We write $\text{exp}_x(u)$ for the exponent sum of the letters $x$ in a word $u$, and $|u|_x$ for the number of letters $x^{\pm 1}$.

\begin{lemma} \label{length detail} If $g \in \langle a ,c,d \rangle \leq G_4$ and $|g|_{G_4} =n$, then there is a word $\sigma$ on $a,c,d$ such that $g=\sigma$ in $G_4$, $| \sigma |_a + | \sigma |_d  \leq n$, and  $| \sigma |_c  \leq 3 n^2$.  
\end{lemma}

\begin{proof}
  Let $\tau$ be a minimal length word on $a, b,c,d$ representing $g$.  Then $|\tau|=n$, because $g \in A$ and killing $s$ retracts $G_4$ onto $A$.
 
 The subgroup $\langle a ,c,d \rangle$ is  $\Z^2 \ast \Z$ with $a$ and $c$ generating the $\Z^2$-factor and $d$ the $\Z$-factor.  Since $g\in \langle a,c,d\rangle$, there exists a representative word $\sigma$ of $g$ of the form $\sigma = \sigma_0 d^{\delta_1} \sigma_1 \cdots  d^{\delta_m} \sigma_m$ where  $\delta_i \neq 0$ and $\sigma_i = a^{\alpha_i} c^{\gamma_i}$ for $i=0,\ldots, m$, and  $\sigma_j \neq 1$ for $j=1,\ldots, m-1$.  Since $\sigma=\tau$ in $A$, there exists a reduced van~Kampen diagram $\Delta$ for $\sigma \tau^{-1}$ over $A$. 
 
 We claim that there is no  $a$- or $d$-corridor connecting an $a$ to an $a^{-1}$ or a $d$ to a $d^{-1}$ in $\sigma$. Suppose, for a contradiction, that there is such a corridor $Q$.  Let  $\sigma'$ be the subword of $\sigma$ whose first and last letters label the edges at the ends of $Q$.  Then $Q$ bounds a subdiagram $\Delta'$ of $\Delta$. No $a$-corridor can cross a $d$-corridor, so there is an \emph{innermost} such $Q$ --- that is, a $Q'$ such that no  $a$- or $d$-corridors connects an $a$ to an $a^{-1}$ or a $d$ to a  $d^{-1}$ in the subword of $\sigma'$ strictly between the first and last edges of $Q'$. Without loss of generality, assume $Q=Q'$. The first and last edges of $Q$ cannot be contained in the same subword $\sigma_i$, nor in the same subword $d^{\delta_i}$, since $\sigma_i$ and $d^{\delta_i}$  are both freely reduced. So, if $Q$ is an $a$-corridor, then $\sigma'$ contains at least one $d$-letter, and if $Q$ is a $d$-corridor then $\sigma'$ contains at least one $a$- or $c$-letter. Any $a$- or $d$-letter in $\sigma'$ must be part of an $a$- or $d$-corridor $Q''$. However $Q''$ can neither cross $Q$ nor connect two letters of $\sigma'$. If, on the other hand, $\sigma'$ contains no $a$- or $d$- corridor, then $\sigma'=c^k$ for some $k\in \ZZ\smallsetminus\{0\}$, and $Q$ is a $d$-corridor. Yet, $\langle b \rangle \leq G_4$ and $\langle c \rangle \leq G_4$ have trivial intersection by inspection of the  normal form of $\langle a, b, c \rangle \cong \mathcal{H}_3(\ZZ) \leq G_4$. In either case we have a contradiction, so our claim is proved.

 Thus, every $a$ or $d$ in $\sigma$ is connected to an  $a$ or $d$ in $\tau$ by an $a$- or $d$-corridor in $\Delta$.    So $| \sigma |_a + | \sigma |_d  \leq n$.

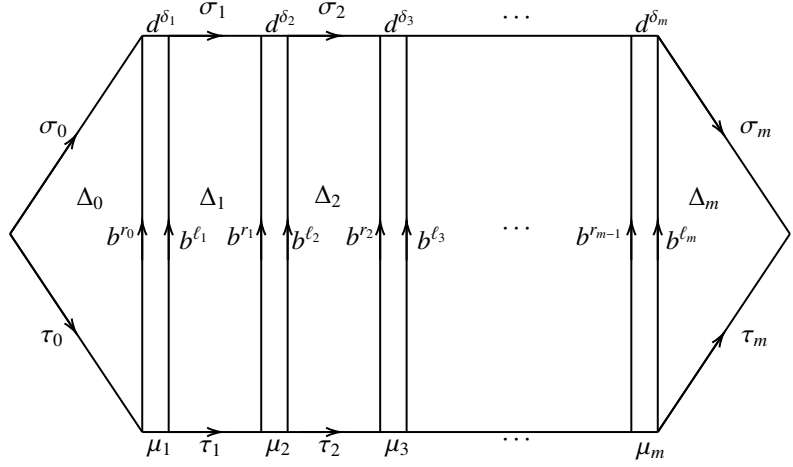
\begin{figure}
    \centering

\tikzset{every picture/.style={line width=0.75pt}} 

\begin{tikzpicture}[x=0.5pt,y=0.5pt,yscale=-1,xscale=1]

\draw    (60,200) -- (160,50) ;
\draw    (550,50) -- (160,50) ;
\draw    (650,200) -- (550,50) ;
\draw    (160,350) -- (60,200) ;
\draw    (550,350) -- (650,200) ;
\draw    (550,350) -- (160,350) ;
\draw    (160,50) -- (160,350) ;
\draw    (180,50) -- (180,350) ;
\draw    (250,50) -- (250,350) ;
\draw    (270,50) -- (270,350) ;
\draw    (530,50) -- (530,350) ;
\draw    (550,50) -- (550,350) ;
\draw    (340,50) -- (340,350) ;
\draw    (360,50) -- (360,350) ;
\draw    (60,200) -- (108.89,126.66) ;
\draw [shift={(110,125)}, rotate = 123.69] [color={rgb, 255:red, 0; green, 0; blue, 0 }  ][line width=0.75]    (10.93,-3.29) .. controls (6.95,-1.4) and (3.31,-0.3) .. (0,0) .. controls (3.31,0.3) and (6.95,1.4) .. (10.93,3.29)   ;
\draw    (60,200) -- (108.89,273.34) ;
\draw [shift={(110,275)}, rotate = 236.31] [color={rgb, 255:red, 0; green, 0; blue, 0 }  ][line width=0.75]    (10.93,-3.29) .. controls (6.95,-1.4) and (3.31,-0.3) .. (0,0) .. controls (3.31,0.3) and (6.95,1.4) .. (10.93,3.29)   ;
\draw    (550,350) -- (598.89,276.66) ;
\draw [shift={(600,275)}, rotate = 123.69] [color={rgb, 255:red, 0; green, 0; blue, 0 }  ][line width=0.75]    (10.93,-3.29) .. controls (6.95,-1.4) and (3.31,-0.3) .. (0,0) .. controls (3.31,0.3) and (6.95,1.4) .. (10.93,3.29)   ;
\draw    (550,50) -- (598.89,123.34) ;
\draw [shift={(600,125)}, rotate = 236.31] [color={rgb, 255:red, 0; green, 0; blue, 0 }  ][line width=0.75]    (10.93,-3.29) .. controls (6.95,-1.4) and (3.31,-0.3) .. (0,0) .. controls (3.31,0.3) and (6.95,1.4) .. (10.93,3.29)   ;
\draw    (180,50) -- (218,50) ;
\draw [shift={(220,50)}, rotate = 180] [color={rgb, 255:red, 0; green, 0; blue, 0 }  ][line width=0.75]    (10.93,-3.29) .. controls (6.95,-1.4) and (3.31,-0.3) .. (0,0) .. controls (3.31,0.3) and (6.95,1.4) .. (10.93,3.29)   ;
\draw    (270,50) -- (297,50) -- (308,50) ;
\draw [shift={(310,50)}, rotate = 180] [color={rgb, 255:red, 0; green, 0; blue, 0 }  ][line width=0.75]    (10.93,-3.29) .. controls (6.95,-1.4) and (3.31,-0.3) .. (0,0) .. controls (3.31,0.3) and (6.95,1.4) .. (10.93,3.29)   ;
\draw    (280,350) -- (308,350) ;
\draw [shift={(310,350)}, rotate = 180] [color={rgb, 255:red, 0; green, 0; blue, 0 }  ][line width=0.75]    (10.93,-3.29) .. controls (6.95,-1.4) and (3.31,-0.3) .. (0,0) .. controls (3.31,0.3) and (6.95,1.4) .. (10.93,3.29)   ;
\draw    (180,350) -- (218,350) ;
\draw [shift={(220,350)}, rotate = 180] [color={rgb, 255:red, 0; green, 0; blue, 0 }  ][line width=0.75]    (10.93,-3.29) .. controls (6.95,-1.4) and (3.31,-0.3) .. (0,0) .. controls (3.31,0.3) and (6.95,1.4) .. (10.93,3.29)   ;
\draw    (530,220) -- (530,192) ;
\draw [shift={(530,190)}, rotate = 90] [color={rgb, 255:red, 0; green, 0; blue, 0 }  ][line width=0.75]    (10.93,-3.29) .. controls (6.95,-1.4) and (3.31,-0.3) .. (0,0) .. controls (3.31,0.3) and (6.95,1.4) .. (10.93,3.29)   ;
\draw    (550,220) -- (550,192) ;
\draw [shift={(550,190)}, rotate = 90] [color={rgb, 255:red, 0; green, 0; blue, 0 }  ][line width=0.75]    (10.93,-3.29) .. controls (6.95,-1.4) and (3.31,-0.3) .. (0,0) .. controls (3.31,0.3) and (6.95,1.4) .. (10.93,3.29)   ;
\draw    (340,220) -- (340,192) ;
\draw [shift={(340,190)}, rotate = 90] [color={rgb, 255:red, 0; green, 0; blue, 0 }  ][line width=0.75]    (10.93,-3.29) .. controls (6.95,-1.4) and (3.31,-0.3) .. (0,0) .. controls (3.31,0.3) and (6.95,1.4) .. (10.93,3.29)   ;
\draw    (360,220) -- (360,192) ;
\draw [shift={(360,190)}, rotate = 90] [color={rgb, 255:red, 0; green, 0; blue, 0 }  ][line width=0.75]    (10.93,-3.29) .. controls (6.95,-1.4) and (3.31,-0.3) .. (0,0) .. controls (3.31,0.3) and (6.95,1.4) .. (10.93,3.29)   ;
\draw    (250,220) -- (250,192) ;
\draw [shift={(250,190)}, rotate = 90] [color={rgb, 255:red, 0; green, 0; blue, 0 }  ][line width=0.75]    (10.93,-3.29) .. controls (6.95,-1.4) and (3.31,-0.3) .. (0,0) .. controls (3.31,0.3) and (6.95,1.4) .. (10.93,3.29)   ;
\draw    (270,220) -- (270,192) ;
\draw [shift={(270,190)}, rotate = 90] [color={rgb, 255:red, 0; green, 0; blue, 0 }  ][line width=0.75]    (10.93,-3.29) .. controls (6.95,-1.4) and (3.31,-0.3) .. (0,0) .. controls (3.31,0.3) and (6.95,1.4) .. (10.93,3.29)   ;
\draw    (160,220) -- (160,192) ;
\draw [shift={(160,190)}, rotate = 90] [color={rgb, 255:red, 0; green, 0; blue, 0 }  ][line width=0.75]    (10.93,-3.29) .. controls (6.95,-1.4) and (3.31,-0.3) .. (0,0) .. controls (3.31,0.3) and (6.95,1.4) .. (10.93,3.29)   ;
\draw    (180,220) -- (180,192) ;
\draw [shift={(180,190)}, rotate = 90] [color={rgb, 255:red, 0; green, 0; blue, 0 }  ][line width=0.75]    (10.93,-3.29) .. controls (6.95,-1.4) and (3.31,-0.3) .. (0,0) .. controls (3.31,0.3) and (6.95,1.4) .. (10.93,3.29)   ;

\draw (80,112.4) node [anchor=north west][inner sep=0.75pt]    {$\sigma _{0}$};
\draw (201,22.4) node [anchor=north west][inner sep=0.75pt]    {$\sigma _{1}$};
\draw (377,20.4) node [anchor=north west][inner sep=0.75pt]    {$\ $};
\draw (531,28.4) node [anchor=north west][inner sep=0.75pt]    {$d^{\delta _{m}} \ $};
\draw (341,28.4) node [anchor=north west][inner sep=0.75pt]    {$d^{\delta _{3}} \ $};
\draw (161,28.4) node [anchor=north west][inner sep=0.75pt]    {$d^{\delta _{1}} \ $};
\draw (291,22.4) node [anchor=north west][inner sep=0.75pt]    {$\sigma _{2}$};
\draw (81,270.4) node [anchor=north west][inner sep=0.75pt]    {$\tau _{0}$};
\draw (201,354.4) node [anchor=north west][inner sep=0.75pt]    {$\tau _{1}$};
\draw (610,112.4) node [anchor=north west][inner sep=0.75pt]    {$\sigma _{m}$};
\draw (611,270.4) node [anchor=north west][inner sep=0.75pt]    {$\tau _{m}$};
\draw (131,192.4) node [anchor=north west][inner sep=0.75pt]    {$b^{r_{0}}$};
\draw (186,192.4) node [anchor=north west][inner sep=0.75pt]    {$b^{\ell _{1}}$};
\draw (554,192.4) node [anchor=north west][inner sep=0.75pt]    {$b^{\ell _{m}}$};
\draw (430,32.4) node [anchor=north west][inner sep=0.75pt]    {$\cdots $};
\draw (430,192.4) node [anchor=north west][inner sep=0.75pt]    {$\cdots $};
\draw (430,352.4) node [anchor=north west][inner sep=0.75pt]    {$\cdots $};
\draw (532,355.4) node [anchor=north west][inner sep=0.75pt]    {$\mu _{m} \ $};
\draw (342,353.4) node [anchor=north west][inner sep=0.75pt]    {$\mu _{3}$};
\draw (162,353.4) node [anchor=north west][inner sep=0.75pt]    {$\mu _{1} \ $};
\draw (252,353.4) node [anchor=north west][inner sep=0.75pt]    {$\mu _{2} \ $};
\draw (108,162.4) node [anchor=north west][inner sep=0.75pt]    {$\Delta _{0}$};
\draw (201,162.4) node [anchor=north west][inner sep=0.75pt]    {$\Delta _{1}$};
\draw (571,162.4) node [anchor=north west][inner sep=0.75pt]    {$\Delta _{m}$};
\draw (251,28.4) node [anchor=north west][inner sep=0.75pt]    {$d^{\delta _{2}} \ $};
\draw (222,192.4) node [anchor=north west][inner sep=0.75pt]    {$b^{r_{1}}$};
\draw (271,192.4) node [anchor=north west][inner sep=0.75pt]    {$b^{\ell _{2}}$};
\draw (312,192.4) node [anchor=north west][inner sep=0.75pt]    {$b^{r_{2}}$};
\draw (366,192.4) node [anchor=north west][inner sep=0.75pt]    {$b^{\ell _{3}}$};
\draw (486,192.4) node [anchor=north west][inner sep=0.75pt]    {$b^{r_{m-1}}$};
\draw (288,160.4) node [anchor=north west][inner sep=0.75pt]    {$\Delta _{2}$};
\draw (291,354.4) node [anchor=north west][inner sep=0.75pt]    {$\tau _{2}$};

\end{tikzpicture}

\caption{The van~Kampen diagram $\Delta$ for $\sigma \tau^{-1}$.  }
    \label{fig:acd diagram}
\end{figure}

 Per Figure~\ref{fig:acd diagram}, for $i=1,\ldots, m-1$, let $\Delta_i$ be the subdiagram of $\Delta$   between the pair of $d$-corridors starting at the last letter of $d^{\delta_{i}}$ and the first letter of $d^{\delta_{i+1}}$. Let $\Delta_0$ be the subdiagram to the left of  
the $d$-corridor starting at the first $d$-letter of $d^{\delta_1}$ and let $\Delta_m$ be that 
to the right of the $d$-corridor starting at the last $d$-letter of $d^{\delta_m}$.   
Thereby, $\sigma = \sigma_0  d^{\delta_{1}} \sigma_1   \cdots  d^{\delta_{m}} \sigma_m$ and  $\tau = \tau_0 \mu_1 \tau_1   \cdots  \mu_m \tau_m $ (as words) so $\Delta_i$ is a subdiagram for $b^{\ell_{i}} \sigma_i  b^{-r_{i}}  \tau_i^{-1}$ where the $b^{\ell_i}$ and $b^{r_i}$ are words along the sides of $d$-corridors, as shown, and $\ell_0=r_m=0$.

Now, for all $i$, as witnessed by the subdiagram of $\Delta$ to the left of the path labelled by $b^{\ell_i}$, we have that $b^{\ell_i} = \hat{\tau}_i^{-1}\hat{\sigma}_i$ and $b^{r_i} = \bar{\tau}_i^{-1}\bar{\sigma}_i$ in $G_4$ for some  prefixes  $\hat{\sigma}_i$ and $\hat{\tau}_i$ and suffixes    $\bar{\sigma}_i$ and $\bar{\tau}_i$  of $\sigma$ and $\tau$, respectively.    So mapping $A \onto \langle b \rangle = \Z$ by killing $a,c,d$ gives that
$|\ell_i|+|r_i| \leq n$.

 Let $\tilde{\tau_i}$ be $\tau_i$ with all $d$-letters removed.
Because  $\sigma_i   = b^{\ell_{i}}  \tau_i b^{-r_{i}}$ in $G_4$, we learn  from Lemma~\ref{structureG4}\eqref{Heisenberg retract} that $\sigma_i   = b^{\ell_{i}}  \tilde{\tau_i} b^{-r_{i}}$  in $\mathcal{H}_3(\ZZ)$.
 Therefore $$|\sigma_i| \ \leq \ |\tau_i|_c +  \abs{\text{exp}_a ( \tilde{\tau_i})}  \,  \abs{\text{exp}_b (b^{\ell_i} \tilde{\tau_i}  b^{-r_{i}})}.$$ The result follows.
\end{proof}

\begin{lemma} \label{no b-rings}
Every word $w$ representing the identity element in $G_4$ admits a reduced van~Kampen diagram with no $b$-rings. Further, every pair of words representing conjugate elements of $G_4$ admits a reduced annular diagram with no contractible $b$- or $d$-rings. 
\end{lemma}

\begin{proof}
Given that $G_4$ is, by Lemma~\ref{structureG4}, an HNN-extension with stable letter $b$ and an HNN-extension with stable letter $d$, this claim then follows from Corollary~3.7 in \cite{BRS} which sets out how contractible $b$- or $d$-rings can be eliminated.
\end{proof}

\begin{lemma} \label{area detail in E}
If $u=1$ in  $E   =   \langle a, c,d,s \mid  [a,c],   asa^{-1} s^{-2}, dsd^{-1}   s^{-2} \rangle$, then  $$\Area_E(u) \leq (|u|_c + |u|_s) 2^{( |u|_a + |u|_d )}.$$  
\end{lemma}

\begin{proof}

Let $E_0=\langle c,s\mid\rangle$,   let $\varphi_a$ be the monomorphism $E_0\hookrightarrow E_0$ given by $c\mapsto c, s\mapsto s^2$, and let $\varphi_d$ be the monomorphism $\langle s\rangle_{E_0}$ given by $s\mapsto s^2$. Then $E$ is the multiple HNN-extension of $E_0$ along $\varphi_a$ and $\varphi_d$ with stable letters $a$ and $d$, respectively. Our claim now follows from Britton's lemma, observing that $\varphi_a$ and $\varphi_d$ increase the length of a word by at most a factor of 2, that, for $e\in \{a,d\}$, $x\in E_0$, a subword of $u$ of the form $e^{\pm1}\varphi_e^{\mp1}(x)e^{\pm1}$ requires $\abs{\varphi_e^{\mp1}(x)}$ relations to transform into the word $x$, and that this transformation strictly decreases the number of $a$- or $d$-letters in $u$.
\end{proof}

\begin{lemma}\label{area detail A} If $u=1$ in $A$, then $\Area_A(u)\leq \lambda|u|^3$ for some constant $\lambda>0$.
\end{lemma}
\begin{proof}
    This follows immediately from \cite[Lemma III.$\Gamma$.6.20]{BrH}, the Dehn function for $\mathcal{H}_{3}(\Z)$, and the fact that $A$ is a trivial HNN-extension of $\mathcal{H}_{3}(\Z)$ with stable letter $d$ along the undistorted subgroup $\langle b\rangle_{A}$.
\end{proof}

\begin{prop} \label{Dehn exp upper bound G4}
$\Area_{G_4}(n) \lesssim 2^n$.
\end{prop}

\begin{proof}
Suppose $w$ is a word of length $n$ representing $1$ in $G_4$.  By Lemma~\ref{no b-rings}, $w$ admits a reduced van~Kampen diagram $\Delta$ in which there are no $b$-rings.

Consider a $b$-corridor $\beta$ in $\Delta$. If $w_1$ is the subword of $w$ along the portion of $\partial \Delta$ between the first and last $b$-edges of $\beta$ (inclusive), then $w_1$ represents an element of $\langle a,c,d\rangle_{G_4}$ and so, by Lemma~\ref{length detail}, there exists a word $w_2$ on $\{a,c,d\}$ that equals $w_1$ in $G_4$ and has   $|w_2|_a + |w_2|_d  \leq |w_1|$ and $|w_2|_c  \leq 3|w_1|^2$. So $|w_2| \leq 4|w_1|^2$.  Then, per Lemma~\ref{area detail A}, $$\Area(w_2 w_1^{-1})  \ \leq \ \lambda |w_1 w_2^{-1}|^3 \  \leq  \  \lambda(|w_1|+|w_2|)^3 \  \leq \ \lambda ( |w_1| + 4|w_1|^2)^3 \  \leq  \ 5^3 \lambda |w_1|^6.$$

There is a family  of disjoint subwords  in $w$ of the form of $w_1$ described above that together include all the $b$-letters in $w$. After replacing each one with its corresponding $w_2$, we get a word $w'$ on $a$, $c$, $d$, and $s$ such that $w = w'$ in $G_4$,  $|w'|_a + |w'|_d + |w'|_s  \leq n$,  $|w'|_c  \leq 3 n^2$, and $\Area_{G_4}(w'  w ^{-1}) \leq 5^3 \lambda n^6$.       

Because $w'$ contains no $b$-letters, it represents $1$ in  $E$ (by Lemma~\ref{structureG4}).  So  Lemma~\ref{area detail A} applies and tells us that  $\Area_E(w') \leq (3 n^2 + n) 2^n$. 
But  $\Area_{G_4}(w') \leq \Area_E(w')$ and, by Lemma~\ref{triangle ineq for A},  $\Area_{G_4}(w) \leq \Area_{G_4}(w'  w^{-1})  + \Area_{G_4}(w')$.  The result follows.  
\end{proof}

\begin{lemma} \label{Aquadratic}  $\CL_A(n) \lesssim n^2$ for $A = \langle a,b,c,d \mid [a,b]c^{-1}, [a,c], [b,c], [b,d]  \rangle$.
\end{lemma}

\begin{proof} 
Recall that we denote  the Heisenberg group $\langle a,b,c  \mid [a,c],[b,c], [a,b]c^{-1}   \rangle$ by $G_1$.  The group $A$ is an HNN-extension of $G_1$ with stable letter $d$.  Let   $\psi:A \to G_1$ be  the surjection   killing $d$.   

Let $u$ and $v$ be words on $\set{a,b,c,d}$ which represent conjugate elements of $A$, and let $n=|u|+|v|$. Let $\Omega$ be a reduced annular diagram witnessing $u\sim v$ in $A$.  We have two cases.

Suppose first that there are no radial $d$-corridors in $\Omega$. 
Then every $d$-corridor runs to and from the same boundary component. It follows that we can replace $u$ and $v$ by some  cyclic conjugates that are equal in $A$ to some words  over $\{a,b,c\}$.  But then  $u = \psi(u)$ and $v = \psi(v)$ in $A$, and the result then  follows from $\CL_{G_1}(n)\lesssim n^2$ and $|\psi(u)|+|\psi(v)|\leq |u|+|v|$.

Suppose, on the other hand, that $\Omega$ has $m \geq 1$ radial $d$-corridors.
Then, after replacing $u$ and $v$ with cyclic conjugates,  we can express them as concatenations of subwords $$\begin{array}{rl} u  & \!\!\! = \   d^{\epsilon_1} u_1 \cdots d^{\epsilon_m} u_m \\ v &  \!\!\! = \    d^{\epsilon_1} v_1 \cdots d^{\epsilon_m} v_m \end{array}$$ where, for all $i$, $\epsilon_i = \pm 1$ and the edge labelled by the $d^{\epsilon_i}$ in  $u$ is joined by a radial $d$-corridor to the edge labelled by the $d^{\epsilon_i}$ in $v$.

Every element of $G_1$ can be expressed uniquely as  $a^{\alpha}b^{\beta}c^{\gamma}$ for some $\alpha, \beta, \gamma \in \Z$.  Let $\bar{u}_i = a^{\alpha_i}b^{\beta_i}c^{\gamma_i}$ and $\bar{v}_i = a^{\alpha'_i}b^{\beta'_i}c^{\gamma'_i}$ be the normal forms of $\psi(u_i)$ and $\psi(v_i)$.

For all $i$, any $d$-corridor in $\Omega$ emanating from a letter $d^{\pm 1}$ in $u_i$ (or $v_i$) must be a $d$-arch ending at some letter $d^{\mp 1}$ in $u_i$ (or $v_i$). So 

\begin{equation}
\psi(u_i)  \ =  \ u_i  \ = \  \bar{u}_i \ \text{ and }  \ \psi(v_i) \ = \  v_i  \ = \  \bar{v}_i   \label{uuinA}
\end{equation}
in   $A$ for all $i$.  A word on $\set{a,b,c}$ can be converted to its $G_1$-normal form by shuffling letters,  with the proviso that interchanging an $a$ and $b$ introduces a $c^{\pm 1}$.   So, because   $\bar{u}_i$ and $\bar{v}_i$ can be obtained from $u_i$ and $v_i$ by deleting their $d$-letters and shuffling letters in this manner, 
\begin{equation}
\sum_{i=1}^m \left( |\alpha_i| + |\beta_i| + |\alpha_i'| + |\beta_i'| \right) \leq n  \  \text{ and } \   \sum_{i=1}^m \left( |\gamma|+|\gamma'| \right) \leq n^2. \label{sizes of exponents}
\end{equation}

Let $b^{k_i}$ be the word read along the sides of the $i$-th radial $d$-corridor in $\Omega$ so that 
\begin{equation} \label{bu=vb}
b^{k_i} u_i \ = \  v_i b^{k_{i+1}}
\end{equation}
in $A$ for all $i$ (subscripts modulo $m$). Now, \eqref{uuinA} and \eqref{bu=vb} imply that       
\begin{flalign} \label{bu=vb in H3}
b^{k_i} \ a^{\alpha_i}b^{\beta_i}c^{\gamma_i} \ = \  a^{\alpha'_i}b^{\beta'_i}c^{\gamma'_i} \ b^{k_{i+1}} 
\end{flalign}
in $A$, and therefore in $G_1$.  
The normal forms of left and right sides of \eqref{bu=vb in H3} are $a^{\alpha_i} b^{\beta_i+ k_i}c^{\gamma_i - \alpha_i k_i}$ and  $a^{\alpha'_i}b^{\beta'_i + k_{i+1}}c^{\gamma'_i}$, respectively.  So, for all $i$ (subscripts modulo $m$),
\begin{flalign}  
\alpha_i & = \alpha'_i,  \label{alphas}  \\ 
\beta_i+ k_i & = \beta'_i + k_{i+1}, \\ 
\gamma_i - \alpha_i k_i  & = \   \gamma'_i.   \label{gamma eqn}
\end{flalign}

Suppose $\alpha_i \neq 0$ for some $i$.  Then \eqref{gamma eqn} gives $k_i = (\gamma_i'-\gamma_i)/ \alpha_i$, whence $|k_i| \leq |\gamma_i'| + |\gamma_i| \leq n^2$, which proves the result in this case because $b^{k_i}$ conjugates a cyclic conjugate of $u$ to  a cyclic conjugate of $v$.    

Suppose, on the other hand, that $\alpha_i = 0$ for all $i$.   Then  $\alpha'_i=0$ for all $i$ by \eqref{alphas}. 
Define
 $$\begin{array}{rl} \bar{u}  & \!\!\! = \   d^{\epsilon_1} \bar{u}_1 \cdots d^{\epsilon_m} \bar{u}_m \\ \bar{v} &  \!\!\! = \    d^{\epsilon_1} \bar{v}_1 \cdots d^{\epsilon_m} \bar{v}_m. \end{array}$$ 
 Then  $u = \bar{u}$ and $v = \bar{v}$ in $A$ by \eqref{uuinA}.  

Let $M = \langle b,c,d \mid  [b,c], [b,d]  \rangle$.  Mapping $b \mapsto cb$ and $c \mapsto c$ defines an automorphism of the subgroup $\Z^2 = \langle b, c \rangle$ of $M$. Thereby, $A$ is an HNN-extension of $M$ with stable letter $a$.  Because there are no $a$-letters in $\bar{u}$ and $\bar{v}$, they are conjugate in $M$.  The  conjugator length function of $M$ is linear per  Servatius' solution to the conjugacy problem for RAAGs in \cite{Servatius}.  The required bound then follows: for some constant $C >0$, $$\CL_A(u,v)  \ \leq  \ \CL_M(\bar{u},\bar{v}) \ \leq  \  C \left( |\bar{u}| + |\bar{v}| \right)  \ \leq  \  C (n+ n^2),$$ with the final inequality following from  \eqref{sizes of exponents}.        
\end{proof}

\begin{lemma} \label{Blinear}
$\CL_B(n) \lesssim n$ for  $B = \langle a, d,s\mid  a s a^{-1} = s^{2}, d s d^{-1} =   s^{2} \rangle$.
\end{lemma}

          \begin{proof}
Suppose $g$ and $h$ are conjugate elements of $B$.  
  Let $u_0$ and $v_0$ be minimal length words  representing $g$ and $h$, respectively, let  $n = |u_0|+|v_0|$, and let $\Delta$ be an  annular van~Kampen diagram for the pair $u_0$ and $v_0$.        
  
\emph{Case: $\Delta$ has no radial $a$- or $d$-corridors.}  In this event, all the $a$- or $d$-corridors originating at a boundary component must be arches of $\Delta$. Indeed, we claim that there exist some cyclic conjugates $u_1$ and $v_1$ of $u_0$ and $v_0$ equal to $s^{k}$ and $s^{\ell}$, respectively,  for some integers $k$, $\ell$ such that $|k|+|\ell| \leq 2^n$. To see this, take $u_1$ and $v_1$ to be the cyclic conjugates of $u_0$ and $v_0$ such that there exists a simple path in the 1-skeleton of $\Delta$ from the initial vertex of $u_1$ to the initial vertex of $v_1$ which crosses no $a$- or $d$-arch. Then $u_1$ and $v_1$ can be converted to $s^{k}$ and $s^{\ell}$, respectively, by eliminating all \emph{pinches}: that is, if $u_1$ or $v_1$ contains a subword of the form $a s^{i} a^{-1}$, $d s^{i} d^{-1}$, $a^{-1} s^{2i} a$, or $d^{-1} s^{2i} d$, where $i \in \Z$, then replace it with $s^{2i}$, $s^{2i}$, $s^{i}$, or, $s^{i}$, respectively.   We may do this because arches in $\Delta$ correspond to pinches in cyclic conjugates of $u_0$ and $v_0$, and  $u_1$ and $v_1$  contain all such pinches by definition. Therefore, we may exhaustively freely reduce and make such substitutions until no such subwords remain. In particular, every $a$- or $d$-letter will be removed from $u_1$ and $v_1$, so the result will be powers $s^{k},s^{\ell}$ of $s$; respectively. Moreover, since at most $n$ pinches are present in $u_1$ and $v_1$, and every substitution at most doubles the total number of $s$-letters on both boundary components, so $|k|+|\ell|\leq 2^n$ as desired.

But then there exists a  word $w$ such that  $w s^{k} w^{-1} = s^{\ell}$ in $B$. So $w s^{k} w^{-1}$ can be converted to $s^{\ell}$ by successively eliminating pinches, and therefore    $s^{\ell} =  s^{2^r k} = a^{r} s^k a^{-r}$  for $r = \exp_a(w) + \exp_d(w)$.  So $\ell = 2^r k$ and $r \lesssim n$, which implies $\CL(u, v) \lesssim n$.

\emph{Case: there is a radial $a$- or $d$-corridor in $\Delta$.}   Take cyclic conjugates $u_1$ and $v_1$ of  $u_0$ and $v_0$, respectively,  beginning at the same radial $a$- or $d$-corridor $Q$.  Then $u_1 = \alpha_1 e_1 \cdots \alpha_m e_m$ and $v_1 = \beta_1 f_1 \cdots \beta_m f_m$ for some number $m \in \N$, some letters $e_1, \ldots, e_m, f_1, \ldots, f_m \in \set{a^{\pm 1},d^{\pm 1}}$, and some words $\alpha_1, \ldots, \alpha_m, \beta_1, \ldots, \beta_m$ on $a,d,s$.  The power of $s$ read along one side of $Q$ conjugates $u_1$ to $v_1$ in $B$.  Furthermore, for all $i$, a radial corridor connects the edge labelled by $e_i$ in $u_1$ to that labelled by $f_i$ in $v_1$, and each $\alpha_i$ and $\beta_i$ can be converted to a power of $s$ by eliminating pinches.  Accordingly, $u_1$ and $v_1$ are equal in $B$ to words $u = s^{\chi_1} e_1 \cdots s^{\chi_m} e_m$ and $v = s^{\xi_1} e_1 \cdots s^{\xi_m} e_m$, respectively, for some $\chi_1, \ldots, \chi_m, \xi_1, \ldots, \xi_m \in \Z$ such that $|u| + |v| = 2^n$. So $u_1 s^k v_1^{-1}  = u s^k v^{-1}  = s^k$ in $B$ for some $k \in \Z$.    Let $\mu = \exp_a(u) + \exp_d(u)$.  

\emph{Subcase: $\mu =0$.} By a propitious choice of which radial $a$- or $d$-corridor to read $u_1$ and $v_1$ from, we may assume that every suffix $\hat{u}$ of $u$ has the property that $\exp_a(\hat{u}) + \exp_d(\hat{u}) \geq 0$. This ensures that $u s u^{-1}$ can be converted to a power of $s$ by replacing the pinch between $e_m$ in $u$ and $e_m^{-1}$ in $u^{-1}$ by a power of $s$, then likewise replacing the pinch between $e_{m-1}$  and $e_{m-1}^{-1}$, and so on.  Indeed, we get $u s u^{-1} =s$.  And then, given that $u s^k v^{-1}  = s^k$, we deduce that $u=v$ in $B$. This implies $\CL(g,h) \lesssim n$, as desired.

\emph{Subcase: $\mu  \neq 0$.}  Then no non-zero power of $s$ commutes with $u$ and so the $k \in \Z$ such that  $u s^k v^{-1}  = s^k$ is unique.  For $i=1, \ldots, m$ let $\mu_i$ be the exponent sum of $e_1 \cdots e_{i-1}$.  The existence of the diagram $\Delta$  implies that  $$u s^k v^{-1} \ = \ s^{\chi_1} e_1 \cdots s^{\chi_m} e_m  \ s^k \  e_m^{-1} s^{-\xi_m} \cdots e_1^{-1} s^{-\xi_1}$$ can be converted to $s^k$ by eliminating the pinch bookended by $e_m$ and $e_m^{-1}$, then eliminating that  bookended by $e_{m-1}$ and $e_{m-1}^{-1}$, and so on.  Doing so, we calculate that $$k \ = \  2^\mu k  + \sum_{i=1}^m 2^{\mu_i}( \chi_i - \xi_i).$$ 
Now $|\mu|, |\mu_i| \leq |n|$ for all $i$ and $\sum_{i=1}^m ( |\chi_i| + |\xi_i|) \leq 2^n$, and so $k = \frac{1}{1-2^{\mu}} \sum_{i=1}^m 2^{\mu_i}( \chi_i - \xi_i)  \lesssim 2^n$.  Therefore $|s^k|_B \lesssim n$, from which we deduce that $\CL(g,h)  \lesssim n$.           
\end{proof}

\begin{lemma} \label{CL of E}
$\CL_E(n) \lesssim n$ for
 $E   =   \langle a, c,d,s \mid  [a,c],  asa^{-1} s^{-2},  dsd^{-1}    s^{-2} \rangle$.  
\end{lemma}

\begin{proof}
Suppose $u$ and $v$ are words representing conjugate elements of $E$ and that $\Delta$ is a reduced annular diagram for $u$ and $v$.  If there are no radial $c$-corridors in $\Delta$ then, perhaps after replacing $u$ and $v$ by cyclic conjugates, $u,v\in B$. Also, since the only relation involving $c$ is $[a,c]=1$, we see $|u|_B+|v|_B=|u|_E+|v|_E\leq n$. Thus $\CL(u,v)\leq \CL_B(n)$, which is linear by Lemma~\ref{Blinear}.  If there is at least one radial $c$-corridor, then the word $\tau$ along its sides is a reduced power of $a$,  since any canceling pair $a^{\pm1}a^{\mp1}$ would have to lie on the boundary of two cancelling cells. Up to cyclic conjugation, $u$ is therefore conjugated to $v$ by a power $\tau=a^p$ of $a$ for some $p\in \ZZ$.  If there is a also radial $d$-corridor in $\Delta$, then every letter in $\tau$ is connected to a different  $a$-letter in the boundary words $u$ or $v$  by an $a$-corridor, since they cannot cross the $d$-corridor nor, because $\tau$ is reduced, return to cross the $c$-corridor. So $|\tau| \leq |u| + |v|$.  If there   is a $c$-corridor but no radial $d$-corridors in $\Delta$, then by excising the $d$-arches of $\Delta$ we may replace $u,v$ with $u',v'$ such that $u',v'$ contain no $d$-letters, $u'=u,v'=v$ in $E$, $|u'|+|v'|\leq 2^{n}$, and $|u'|_a+|v'|_a\leq n$. Moving each instance of $a$ in $u'$ and $v'$ to the right and each instance of $a^{-1}$ to the left, we have $u'=a^{-k_1}w_1 a^{\ell_1},v'=a^{-k_2}w_2 a^{\ell_2}$, for some $k_1,k_2,\ell_1,\ell_2\in\Naturals$ and words $w_1,w_2$ on $c$ and $s$ such that $|w_1|+|w_2|\leq 2^{2n}$. Conjugating $u'$ and $v'$ by $a^{\ell_1}$ and $a^{\ell_2}$ respectively, we obtain the words $u''=a^{-k_1+\ell_1}w_1,v''=a^{-k_2+\ell_2}w_2$, which are also conjugate by a power $a^{p-\ell_1+\ell_2}$ of $a$. Since $|-\ell_1+\ell_2|\leq |u'|_a+|v'|_a\leq n$ and $\CL(u,v)\leq p$, to show $\CL(u,v)\lesssim n$ it suffices to find a linear bound for $p'=p-\ell_1+\ell_2$. Indeed, since $E$ is an HNN-extension with stable letter $a$ (see Lemma \ref{area detail in E}), we must have $-k_1+\ell_1=-k_2+\ell_2$, so $a^{p'}w_1a^{-p'}=w_2$. Applying to this equation the retraction $E\onto B$ given by killing $c$, we have $a^{p'}s^{q_1}a^{-p'}=s^{q_2}$ in $B$, where $|q_1|+|q_2|\leq |w_1|+|w_2|\leq 2^{2n}$. This gives $|p'|=\log_2\left(\abs{q_1-q_2}\right)\leq 2n$, and we are done.
\end{proof} 

Per Lemma~\ref{structureG4}\eqref{L property}, the subgroup $L=\langle b,c,s\rangle$ of $G_4$ is $\ZZ^2 \ast \ZZ =\langle b,c,s \mid [b,c] \rangle$.  The following lemma 
provides quantitative details of the normal form for elements $L$: 
\begin{equation} 
u  \ = \ b^{\beta_0} c^{\gamma_0} s^{\mu_1} b^{\beta_1} c^{\gamma_1} \cdots s^{\mu_{k}} b^{\beta_{k}} c^{\gamma_{k}}, \label{L normal form}
\end{equation}
where  $\mu_1, \ldots, \mu_{k}$ are non-zero and none of $b^{\beta_1} c^{\gamma_1}, \ldots, b^{\beta_{k -1}} c^{\gamma_{k -1}}$ are the identity in $\Z^2$. 

\begin{lemma} \label{length detail for L} If  $u$ of \eqref{L normal form} equals in $G_4$ a word  $v$ on $\set{a^{\pm 1}, b^{\pm 1}, c^{\pm 1}, d^{\pm 1}, s^{\pm 1}}$, then  $| \beta_0 | + \cdots + | \beta_k | \leq |v|$, $| \gamma_0| + \cdots + | \gamma_k |  \leq |v|^2$ and  $| \mu_1 | + \cdots + | \mu_k |  \leq 2^{|v|}$.  
\end{lemma}

\begin{proof}
We proceed by induction on $|v|$. The base case of $|v|=1$ is trivial. For the inductive step, we have by Britton's Lemma that $$v=\alpha_1e_1\beta_1e_1^{-1}\alpha_2\cdots \alpha_me_m\beta_me_m^{-1}\alpha_{m+1},$$ where $\alpha_i$ is a word on $\{b,c,s\}$, $e_i\in \{a^{\pm1},d^{\pm1}\}$, and  $\beta_i$ is a word representing:\begin{itemize} \item  an element of $\langle b,c,s\rangle_{G_4}$ if $e_i=a$, \item an element of $\langle b,c,s^2\rangle_{G_4}$ if $e_i=a^{-1}$, \item an element of $\langle b,s\rangle_{G_4}$ if $e_i=d$, or\item an element of $\langle b,s^2\rangle_{G_4}$ if $e_i=d^{-1}$\end{itemize} for all $i$. In any of these cases, the inductive hypothesis gives that $\beta_i$ is equal to a word on $\{b,c,s\}$ with $|\beta_i|$ many $b$-letters, $|\beta_i|^2$ many $c$-letters, and $2^{|\beta_i|}$ many $s$-letters. Conjugating by $e_i$ adds at most $|\beta_i|$ many $c$-letters and $2^{|\beta_i|}$ many $s$-letters, so $e_1\beta_1e_1^{-1}$ can be written as a word on $\{b,c,s\}$ with $|\beta_i|$ many $b$-letters, $|\beta_i|^2+|\beta_i|\leq (|\beta_i|+1)^2\leq |e_1\beta_1e_1^{-1}|^2$ many $c$-letters, and $2^{|\beta_i|+1}\leq 2^{|e_1\beta_1e_1^{-1}|}$ many $s$-letters. The claim then follows by the total number of each letter and by the structure of $L$.  \end{proof}

Our next lemma  concerns manipulations of  words in the  normal form \eqref{L normal form}.

\begin{lemma} \label{block}
Suppose $z = f_m f_{m-1} \cdots f_1$ is a reduced word where  $f_1, \ldots, f_m \in \set{a^{\pm 1}, d^{\pm 1}}$.  Suppose that $u_0, \ldots, u_m$ are words  in the  normal form \eqref{L normal form} and that 
\begin{equation} \label{normal form of u_0}  
u_0  \ = \ b^{\beta_0} c^{\gamma_0} s^{\mu_1} b^{\beta_1} c^{\gamma_1} \cdots s^{\mu_{k}} b^{\beta_{k}} c^{\gamma_{k}}. 
\end{equation} 
Assume that  $u_i = f_i u_{i-1} f_i^{-1}$ in $G_4$ for $i=1, \ldots, m$ and that some (and so every) $u_i$ is not an element of $\langle s \rangle$.  
Then 
\begin{itemize}
\item either $z = a^{\lambda}$ (as words) for some $\lambda \in \Z$ such that $m = |\lambda|$ and   
$$\begin{array}{rl} 
u_m   & \!\! = \ b^{\beta_0} c^{\gamma_0 + \lambda \beta_0} s^{2^{\lambda} \mu_1} b^{\beta_1} c^{\gamma_1 + \lambda \beta_1} \cdots s^{2^{\lambda} \mu_{k}} b^{\beta_{k}} c^{\gamma_k + \lambda \beta_k},
\end{array}$$
\item or  $z= a^{\lambda_2}d^{\xi} a^{\lambda_1}$ (as words) for some $\lambda_1, \xi, \lambda_2 \in \Z$ such that $m = |\lambda_1|+ |\xi| + |\lambda_2|$ and
$$\begin{array}{rl} 
u_{|\lambda_1|}   & \!\! = \ b^{\beta_0}   s^{2^{\lambda_1} \mu_1} b^{\beta_1}   \cdots s^{2^{\lambda_1} \mu_{k}} b^{\beta_{k}} , \\ 
u_{|\lambda_1| + |\xi|}   & \!\! = \ b^{\beta_0}   s^{2^{\lambda_1 + \xi} \mu_1} b^{\beta_1}   \cdots s^{2^{\lambda_1 + \xi} \mu_{k}} b^{\beta_{k}}  , \\ 
u_m   & \!\! = \ b^{\beta_0} c^{\gamma'_0} s^{2^{\lambda_1 + \xi + \lambda_2} \mu_1} b^{\beta_1} c^{\gamma'_1} \cdots s^{2^{\lambda_1 + \xi + \lambda_2} \mu_{k}} b^{\beta_{k}} c^{\gamma'_{k}}m
\end{array}$$
and $\gamma_i= -\beta_i \lambda_1$ and $\gamma'_i= \beta_i \lambda_2$ for $i=0, \ldots, k$.  
\end{itemize}
\end{lemma}

\begin{proof}
Here are the key points.  When $f_i = d^{\pm 1}$, the relation $u_i = f_i u_{i-1} f_i^{-1}$ in $G_4$ necessitates that there be no $c$-letters in the normal form of $u_{i-1}$.  And when $f_i = a^{\pm 1}$, we can relate the $c$-letters in $u_{i-1}$ to those in $u_{i}$---for example, because $u_{i-1} , u_{i} \notin \langle s \rangle$, at least one of $u_{i-1}$ and $u_{i}$ contains $c$-letters. 
\end{proof}

We are now ready to bound the lengths of conjugators. Let $\exp(w)$ denote the exponent sum of $w$ over \textit{all} letters appearing in $w$ --- for example, $\exp(ab^{-2}c^3)=2$.   

\begin{lemma} \label{quadratic bound from stack}
Suppose $x$ and $y$ are words on $\set{a^{\pm 1}, b^{\pm 1}, c^{\pm 1}, d^{\pm 1}, s^{\pm 1}}$.  Let  $n = |x|+|y|$.  If there exists a word $w$ on $\set{a^{\pm 1},  d^{\pm 1}}$ such that $wx = y w$ in $G_4$, then there exists such a $w$ for which $|w| \leq 2n +  8 n^2$.
\end{lemma}

\begin{proof}
Suppose $w$ is a word on $\set{a^{\pm 1},  d^{\pm 1}}$ such that $wx = y w$ in $G_4$. Further assume that $w$ is of minimal length among all such $w$ (so in particular is reduced). Let $\Delta$ be a van~Kampen diagram for $w x w^{-1} y^{-1}$.  

No two $a$- or $d$-corridors in $\Delta$ can cross. No pair of $a$-edges (resp.\ $d$-edges) in one of the two portions of $\partial \Delta$ labelled by $w$ can be connected by an $a$-corridor (resp.\ $d$-corridor) --- after all, consider an innermost such $a$- or $d$-corridor $Q$; the subword $\sigma$ of $w$ starting and ending at the two end-edges of $Q$ is reduced and non-trivial, yet every letter is part of a corridor running to and from $\sigma$, giving a contradiction.

So, per Figure~\ref{vK diagram with corridor block}, $w$ can be expressed in two ways as concatenations of subwords:  $$w \ = \  w_0 z w_1   \  = \  \tilde{w}_0  z \tilde{w}_1$$  where $m = |z|$  and the  $z = f_m f_{m-1} \cdots f_1$ subwords label two portions of $\partial \Delta$  so that for all $1 \leq i \leq m$,  the  letter $f_i$ in the $z$ in one copy of $w$ is joined by an $a$- or $d$-corridor to the $f_i$ in the other, and  all $a$- and $d$-corridors starting from $w_0$ or $\tilde{w}_0$ end in $y$, and those starting from $w_1$ or $\tilde{w}_1$ end in $x$.  The latter implies that for $x_1 := w_1 x \tilde{w}_1^{-1}$ and  $y_1 := w_0^{-1} y \tilde{w}_0$ (as words),
      \begin{align}
    |w_0| + |w_1| + |\tilde{w}_0| + |\tilde{w}_1|  &  \ \leq \ n, \text{ and}  \label{u1}  \\
    |x_1| + |y_1|  & \ \leq \ 2n.  \label{u2}
    \end{align}

For $i=0, 1, \ldots, m$, the word $$f_i  f_{i-1} \cdots f_1  \, x_1 \,  f_1^{-1}   \cdots  f^{-1}_{i-1}  f^{-1}_{i}$$ around part of  $\partial \Delta$ equals in $G_4$ a word along one side of  one such corridor, so represents an   element of $L$.  Let $u_i$ be its normal form per \eqref{L normal form}.  Then $x_1 = u_0$ and $y_1= u_m$ in $G_4$ and $z u_0 =u_m z$ in $L$.

\begin{figure}[ht]
\tikzset{every picture/.style={line width=0.75pt}} 

\begin{tikzpicture}[x=0.5pt,y=0.5pt,yscale=-1,xscale=1]

\draw   (71.45,49.45) -- (541.86,49.45) -- (541.86,431.27) -- (71.45,431.27) -- cycle ;
\draw    (297.25,49.45) -- (314.07,49.45) ;
\draw [shift={(316.07,49.45)}, rotate = 180] [color={rgb, 255:red, 0; green, 0; blue, 0 }  ][line width=0.75]    (10.93,-3.29) .. controls (6.95,-1.4) and (3.31,-0.3) .. (0,0) .. controls (3.31,0.3) and (6.95,1.4) .. (10.93,3.29)   ;
\draw    (297.25,431.27) -- (314.07,431.27) ;
\draw [shift={(316.07,431.27)}, rotate = 180] [color={rgb, 255:red, 0; green, 0; blue, 0 }  ][line width=0.75]    (10.93,-3.29) .. controls (6.95,-1.4) and (3.31,-0.3) .. (0,0) .. controls (3.31,0.3) and (6.95,1.4) .. (10.93,3.29)   ;
\draw [line width=1.5]    (71.45,326.27) -- (541.86,288.09) ;
\draw    (71.45,326.27) -- (304.66,307.34) ;
\draw [shift={(306.66,307.18)}, rotate = 175.36] [color={rgb, 255:red, 0; green, 0; blue, 0 }  ][line width=0.75]    (10.93,-3.29) .. controls (6.95,-1.4) and (3.31,-0.3) .. (0,0) .. controls (3.31,0.3) and (6.95,1.4) .. (10.93,3.29)   ;
\draw [line width=1.5]    (71.45,192.64) -- (541.86,154.45) ;
\draw    (71.45,192.64) -- (304.66,173.71) ;
\draw [shift={(306.66,173.55)}, rotate = 175.36] [color={rgb, 255:red, 0; green, 0; blue, 0 }  ][line width=0.75]    (10.93,-3.29) .. controls (6.95,-1.4) and (3.31,-0.3) .. (0,0) .. controls (3.31,0.3) and (6.95,1.4) .. (10.93,3.29)   ;
\draw [line width=1.5]    (71.45,154.45) -- (541.86,116.27) ;
\draw    (71.45,154.45) -- (304.66,135.53) ;
\draw [shift={(306.66,135.36)}, rotate = 175.36] [color={rgb, 255:red, 0; green, 0; blue, 0 }  ][line width=0.75]    (10.93,-3.29) .. controls (6.95,-1.4) and (3.31,-0.3) .. (0,0) .. controls (3.31,0.3) and (6.95,1.4) .. (10.93,3.29)   ;
\draw    (71.45,173.55) -- (541.86,135.36) ;
\draw [line width=0.75]    (71.45,173.55) -- (304.66,154.62) ;
\draw [shift={(306.66,154.45)}, rotate = 175.36] [color={rgb, 255:red, 0; green, 0; blue, 0 }  ][line width=0.75]    (10.93,-3.29) .. controls (6.95,-1.4) and (3.31,-0.3) .. (0,0) .. controls (3.31,0.3) and (6.95,1.4) .. (10.93,3.29)   ;
\draw [line width=1.5]    (71.45,383.55) -- (541.86,345.36) ;
\draw    (71.45,383.55) -- (304.66,364.62) ;
\draw [shift={(306.66,364.45)}, rotate = 175.36] [color={rgb, 255:red, 0; green, 0; blue, 0 }  ][line width=0.75]    (10.93,-3.29) .. controls (6.95,-1.4) and (3.31,-0.3) .. (0,0) .. controls (3.31,0.3) and (6.95,1.4) .. (10.93,3.29)   ;
\draw    (71.45,345.36) -- (541.86,307.18) ;
\draw [line width=0.75]    (71.45,345.36) -- (304.66,326.43) ;
\draw [shift={(306.66,326.27)}, rotate = 175.36] [color={rgb, 255:red, 0; green, 0; blue, 0 }  ][line width=0.75]    (10.93,-3.29) .. controls (6.95,-1.4) and (3.31,-0.3) .. (0,0) .. controls (3.31,0.3) and (6.95,1.4) .. (10.93,3.29)   ;
\draw [line width=1.5]    (71.45,249.91) -- (541.86,211.73) ;
\draw    (71.45,249.91) -- (304.66,230.98) ;
\draw [shift={(306.66,230.82)}, rotate = 175.36] [color={rgb, 255:red, 0; green, 0; blue, 0 }  ][line width=0.75]    (10.93,-3.29) .. controls (6.95,-1.4) and (3.31,-0.3) .. (0,0) .. controls (3.31,0.3) and (6.95,1.4) .. (10.93,3.29)   ;
\draw [line width=1.5]    (71.45,326.27) -- (541.86,288.09) ;
\draw    (71.45,326.27) -- (304.66,307.34) ;
\draw [shift={(306.66,307.18)}, rotate = 175.36] [color={rgb, 255:red, 0; green, 0; blue, 0 }  ][line width=0.75]    (10.93,-3.29) .. controls (6.95,-1.4) and (3.31,-0.3) .. (0,0) .. controls (3.31,0.3) and (6.95,1.4) .. (10.93,3.29)   ;
\draw    (71.45,230.82) -- (541.86,192.64) ;
\draw [line width=0.75]    (71.45,230.82) -- (81.89,229.97) -- (304.66,211.89) ;
\draw [shift={(306.66,211.73)}, rotate = 175.36] [color={rgb, 255:red, 0; green, 0; blue, 0 }  ][line width=0.75]    (10.93,-3.29) .. controls (6.95,-1.4) and (3.31,-0.3) .. (0,0) .. controls (3.31,0.3) and (6.95,1.4) .. (10.93,3.29)   ;
\draw    (71.45,211.73) -- (541.86,173.55) ;
\draw [line width=0.75]    (71.45,211.73) -- (304.66,192.8) ;
\draw [shift={(306.66,192.64)}, rotate = 175.36] [color={rgb, 255:red, 0; green, 0; blue, 0 }  ][line width=0.75]    (10.93,-3.29) .. controls (6.95,-1.4) and (3.31,-0.3) .. (0,0) .. controls (3.31,0.3) and (6.95,1.4) .. (10.93,3.29)   ;
\draw  [dash pattern={on 0.84pt off 2.51pt}]  (52.64,135.36) -- (52.64,364.45) ;
\draw  [dash pattern={on 0.84pt off 2.51pt}]  (52.64,249.91) -- (52.64,242.36) ;
\draw [shift={(52.64,240.36)}, rotate = 90] [color={rgb, 255:red, 0; green, 0; blue, 0 }  ][line width=0.75]    (10.93,-3.29) .. controls (6.95,-1.4) and (3.31,-0.3) .. (0,0) .. controls (3.31,0.3) and (6.95,1.4) .. (10.93,3.29)   ;
\draw    (41.18,155) -- (60,155) ;
\draw    (43.23,383.55) -- (62.05,383.55) ;
\draw  [dash pattern={on 0.84pt off 2.51pt}]  (598.31,97.18) -- (598.31,326.27) ;
\draw  [dash pattern={on 0.84pt off 2.51pt}]  (598.31,211.73) -- (598.31,204.18) ;
\draw [shift={(598.31,202.18)}, rotate = 90] [color={rgb, 255:red, 0; green, 0; blue, 0 }  ][line width=0.75]    (10.93,-3.29) .. controls (6.95,-1.4) and (3.31,-0.3) .. (0,0) .. controls (3.31,0.3) and (6.95,1.4) .. (10.93,3.29)   ;
\draw    (588.9,116.27) -- (607.72,116.27) ;
\draw    (588.9,345.36) -- (607.72,345.36) ;
\draw  [dash pattern={on 0.84pt off 2.51pt}]  (52.64,49.45) -- (52.64,135.36) ;
\draw  [dash pattern={on 0.84pt off 2.51pt}]  (52.64,97.18) -- (52.64,89.64) ;
\draw [shift={(52.64,87.64)}, rotate = 90] [color={rgb, 255:red, 0; green, 0; blue, 0 }  ][line width=0.75]    (10.93,-3.29) .. controls (6.95,-1.4) and (3.31,-0.3) .. (0,0) .. controls (3.31,0.3) and (6.95,1.4) .. (10.93,3.29)   ;
\draw    (43.23,49.45) -- (62.05,49.45) ;
\draw  [dash pattern={on 0.84pt off 2.51pt}]  (52.64,364.45) -- (52.64,402.64) ;
\draw  [dash pattern={on 0.84pt off 2.51pt}]  (52.64,431.27) -- (52.64,404.64) ;
\draw [shift={(52.64,402.64)}, rotate = 90] [color={rgb, 255:red, 0; green, 0; blue, 0 }  ][line width=0.75]    (10.93,-3.29) .. controls (6.95,-1.4) and (3.31,-0.3) .. (0,0) .. controls (3.31,0.3) and (6.95,1.4) .. (10.93,3.29)   ;
\draw    (43.23,431.27) -- (62.05,431.27) ;
\draw  [dash pattern={on 0.84pt off 2.51pt}]  (598.31,49.45) -- (598.31,97.18) ;
\draw  [dash pattern={on 0.84pt off 2.51pt}]  (598.31,78.09) -- (598.31,75.32) ;
\draw [shift={(598.31,73.32)}, rotate = 90] [color={rgb, 255:red, 0; green, 0; blue, 0 }  ][line width=0.75]    (10.93,-3.29) .. controls (6.95,-1.4) and (3.31,-0.3) .. (0,0) .. controls (3.31,0.3) and (6.95,1.4) .. (10.93,3.29)   ;
\draw    (588.9,49.45) -- (607.72,49.45) ;
\draw  [dash pattern={on 0.84pt off 2.51pt}]  (598.31,326.27) -- (598.31,431.27) ;
\draw  [dash pattern={on 0.84pt off 2.51pt}]  (598.31,393.09) -- (598.31,380.77) ;
\draw [shift={(598.31,378.77)}, rotate = 90] [color={rgb, 255:red, 0; green, 0; blue, 0 }  ][line width=0.75]    (10.93,-3.29) .. controls (6.95,-1.4) and (3.31,-0.3) .. (0,0) .. controls (3.31,0.3) and (6.95,1.4) .. (10.93,3.29)   ;
\draw    (588.9,431.27) -- (597.37,431.27) -- (607.72,431.27) ;
\draw    (71.45,364.45) -- (541.86,326.27) ;
\draw [line width=0.75]    (71.45,364.45) -- (304.66,345.53) ;
\draw [shift={(306.66,345.36)}, rotate = 175.36] [color={rgb, 255:red, 0; green, 0; blue, 0 }  ][line width=0.75]    (10.93,-3.29) .. controls (6.95,-1.4) and (3.31,-0.3) .. (0,0) .. controls (3.31,0.3) and (6.95,1.4) .. (10.93,3.29)   ;
\draw    (541.86,383.55) -- (560.68,383.55) ;
\draw  [dash pattern={on 0.84pt off 2.51pt}]  (551.27,383.55) -- (551.27,402.64) ;
\draw  [dash pattern={on 0.84pt off 2.51pt}]  (551.27,431.27) -- (551.27,404.64) ;
\draw [shift={(551.27,402.64)}, rotate = 90] [color={rgb, 255:red, 0; green, 0; blue, 0 }  ][line width=0.75]    (10.93,-3.29) .. controls (6.95,-1.4) and (3.31,-0.3) .. (0,0) .. controls (3.31,0.3) and (6.95,1.4) .. (10.93,3.29)   ;
\draw    (541.86,431.27) -- (560.68,431.27) ;
\draw  [dash pattern={on 0.84pt off 2.51pt}]  (9.41,135.91) -- (10,430) ;
\draw  [dash pattern={on 0.84pt off 2.51pt}]  (9.41,250.45) -- (9.41,242.91) ;
\draw [shift={(9.41,240.91)}, rotate = 90] [color={rgb, 255:red, 0; green, 0; blue, 0 }  ][line width=0.75]    (10.93,-3.29) .. controls (6.95,-1.4) and (3.31,-0.3) .. (0,0) .. controls (3.31,0.3) and (6.95,1.4) .. (10.93,3.29)   ;
\draw    (0,50) -- (18.82,50) ;
\draw  [dash pattern={on 0.84pt off 2.51pt}]  (9.41,50) -- (9.41,135.91) ;
\draw    (1.18,430) -- (20,430) ;
\draw  [dash pattern={on 0.84pt off 2.51pt}]  (639.41,135.91) -- (640,430) ;
\draw  [dash pattern={on 0.84pt off 2.51pt}]  (639.41,250.45) -- (639.41,242.91) ;
\draw [shift={(639.41,240.91)}, rotate = 90] [color={rgb, 255:red, 0; green, 0; blue, 0 }  ][line width=0.75]    (10.93,-3.29) .. controls (6.95,-1.4) and (3.31,-0.3) .. (0,0) .. controls (3.31,0.3) and (6.95,1.4) .. (10.93,3.29)   ;
\draw    (630,50) -- (648.82,50) ;
\draw  [dash pattern={on 0.84pt off 2.51pt}]  (639.41,50) -- (639.41,135.91) ;
\draw    (631.18,430) -- (650,430) ;

\draw (303.48,32.4) node [anchor=north west][inner sep=0.75pt]  [font=\footnotesize]  {$x$};
\draw (303.48,437.13) node [anchor=north west][inner sep=0.75pt]  [font=\footnotesize]  {$y$};
\draw (71.92,366.49) node [anchor=north west][inner sep=0.75pt]  [font=\scriptsize]  {$f_{m}$};
\draw (71.66,347.4) node [anchor=north west][inner sep=0.75pt]  [font=\scriptsize]  {$f_{m-1}$};
\draw (72.98,315.31) node [anchor=north west][inner sep=0.75pt]  [font=\scriptsize]  {$\vdots $};
\draw (72.92,155.45) node [anchor=north west][inner sep=0.75pt]  [font=\scriptsize]  {$f_{1}$};
\draw (71.98,174.58) node [anchor=north west][inner sep=0.75pt]  [font=\scriptsize]  {$f_{2}$};
\draw (72.92,192.63) node [anchor=north west][inner sep=0.75pt]  [font=\scriptsize]  {$f_{3}$};
\draw (72.92,212.72) node [anchor=north west][inner sep=0.75pt]  [font=\scriptsize]  {$f_{4}$};
\draw (72.98,223.85) node [anchor=north west][inner sep=0.75pt]  [font=\scriptsize]  {$\vdots $};
\draw (543.27,329.26) node [anchor=north west][inner sep=0.75pt]  [font=\scriptsize]  {$f_{m}$};
\draw (543,310.17) node [anchor=north west][inner sep=0.75pt]  [font=\scriptsize]  {$f_{m-1}$};
\draw (548.09,275.17) node [anchor=north west][inner sep=0.75pt]  [font=\scriptsize]  {$\vdots $};
\draw (543.33,176.54) node [anchor=north west][inner sep=0.75pt]  [font=\scriptsize]  {$f_{4}$};
\draw (543.33,157.45) node [anchor=north west][inner sep=0.75pt]  [font=\scriptsize]  {$f_{3}$};
\draw (543.39,185.67) node [anchor=north west][inner sep=0.75pt]  [font=\scriptsize]  {$\vdots $};
\draw (576.31,309.22) node [anchor=north west][inner sep=0.75pt]  [font=\footnotesize]  {$\tau $};
\draw (576.31,175.58) node [anchor=north west][inner sep=0.75pt]  [font=\footnotesize]  {$\tau $};
\draw (563.7,127.81) node [anchor=north west][inner sep=0.75pt]  [font=\footnotesize]  {$\ \ \ \tau _{0}$};
\draw (297.72,261.4) node [anchor=north west][inner sep=0.75pt]    {$\vdots $};
\draw (290.19,118.26) node [anchor=north west][inner sep=0.75pt]  [font=\footnotesize]  {$u_{0}$};
\draw (290.19,176.76) node [anchor=north west][inner sep=0.75pt]  [font=\footnotesize]  {$u_{|\tau_0|}$};
\draw (290.13,366.45) node [anchor=north west][inner sep=0.75pt]  [font=\footnotesize]  {$u_{m}$};
\draw (25,232.81) node [anchor=north west][inner sep=0.75pt]  [font=\small]  {$z$};
\draw (-10,237.81) node [anchor=north west][inner sep=0.75pt]  [font=\small]  {$w$};
\draw (643,237.81) node [anchor=north west][inner sep=0.75pt]  [font=\small]  {$w$};
\draw (604.54,194.63) node [anchor=north west][inner sep=0.75pt]  [font=\small]  {$z$};
\draw (21,87.4) node [anchor=north west][inner sep=0.75pt]  [font=\small]  {$w_{1}$};
\draw (22.88,395.04) node [anchor=north west][inner sep=0.75pt]  [font=\small]  {$w_{0}$};
\draw (604.3,366.26) node [anchor=north west][inner sep=0.75pt]  [font=\small]  {$\widetilde{w}_{0}$};
\draw (542.39,118.31) node [anchor=north west][inner sep=0.75pt]  [font=\scriptsize]  {$f_{1}$};
\draw (543.33,138.35) node [anchor=north west][inner sep=0.75pt]  [font=\scriptsize]  {$f_{2}$};
\draw (604.3,70.35) node [anchor=north west][inner sep=0.75pt]  [font=\small]  {$\widetilde{w}_{1}$};
\draw (559.14,395.04) node [anchor=north west][inner sep=0.75pt]  [font=\small]  {$w_{0}$};
\draw (580.39,350.28) node [anchor=north west][inner sep=0.75pt]  [font=\large,rotate=-180]  {$\begin{cases}
 & \\
 & 
\end{cases}$};
\draw (580.39,215.74) node [anchor=north west][inner sep=0.75pt]  [font=\large,rotate=-180]  {$\begin{cases}
 & \\
 & 
\end{cases}$};
\draw (575.18,150.96) node [anchor=north west][inner sep=0.75pt]  [font=\tiny,rotate=-180]  {$\begin{cases}
 & \\
 & 
\end{cases}$};
\draw (551.85,356.95) node [anchor=north west][inner sep=0.75pt]  [font=\footnotesize]  {$\tau $};

\end{tikzpicture}

 \caption{The van~Kampen diagram $\Delta$ for Lemma~\ref{quadratic bound from stack}.}
  \label{vK diagram with corridor block}
\end{figure}
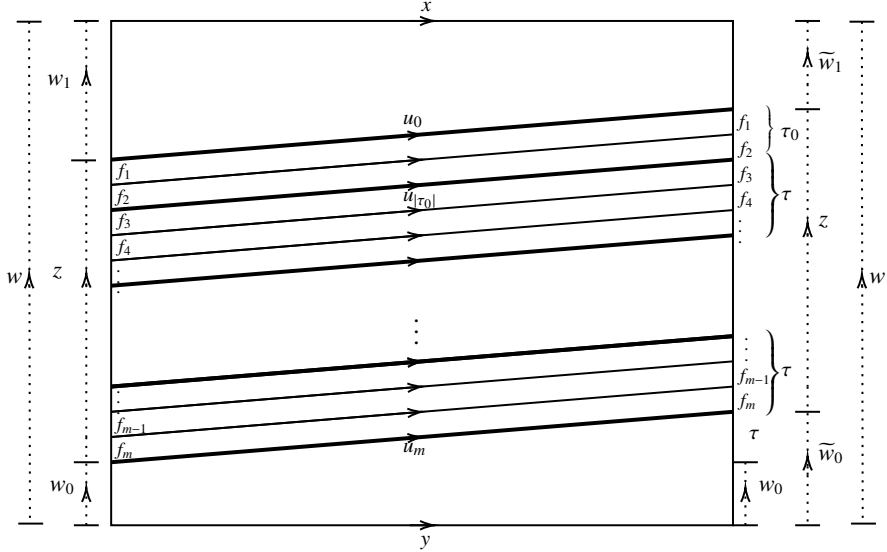

Let $\lambda = \exp(z)$.  Here is an estimate we will call on multiple times. It is contingent on the exponent $\mu_1$ of the first power of $s$ in the  normal form \eqref{normal form of u_0} for $u_0$:
\begin{equation} \label{lambda estimate}
 \mu_1 \neq 0 \ \implies \ |\lambda| \  \leq \ \max \set{|x_1|, |y_1|} \ \leq \ 2n,
 \end{equation}   
To establish this, let $\mu'_1$ be exponent of the first power of $s$ in the normal form of $u_m$.  Then  $\mu'_1 = 2^{\lambda}{\mu_1}$ and, by Lemma~\ref{length detail for L},  $2^{\lambda}{\mu_1} \leq 2^{|y_1|}$.  So   $\lambda \leq |y_1|$ because $\mu_1  \in \Z \smallsetminus \set{0}$.  Also  $2^{-\lambda} {\mu}' = {\mu}$ and ${\mu}' \in \Z \smallsetminus \set{0}$, so likewise,  $-\lambda \leq |x_1|$. So  \eqref{lambda estimate} follows, with   \eqref{u2} giving the final inequality.

We will argue bounds on $|w|$ in cases according to details of the normal form \eqref{L normal form} of $u_0$.

\begin{enumerate}[label=\arabic*.]    

\item \label{Case 1} \emph{Case: $u_0 = s^{\mu_1}$ for some $\mu_1 \in \Z$.}   

\begin{enumerate}[label=\alph*.]
\item \emph{Subcase: $w_0 = \tilde{w}_0$ (as words).}  
If $\mu_1=0$, then $u_0 =u_m=1$ and $z$ is the empty word, for otherwise we could remove it and  $wx = y  w$ in $G_4$ would remain true.  So,     $|w| = |w_0 w_1|  = |w_0| + |w_1|    \leq 2n$ by \eqref{u1}.  

Assume, then, that $\mu_1 \neq 0$. We have $|\lambda| \leq 2n$ by \eqref{lambda estimate}, and $$u_m \ = \ z u_0 z^{-1} \ = \ z s^{\mu_1} z^{-1} \ = \  s^{2^{\lambda} \mu_1} \ = \ a^{\lambda} s^{\mu_1} a^{-\lambda}.$$ So we may take $z = a^{\lambda}$ because doing so could only shorten $w$ and it will remain true that $wx=yw$ in $G_4$.  So,  using \eqref{u1}, we get $|w| = |w_0 a^{\lambda} w_1|  = |w_0| + |w_1| + |\lambda|   \leq 3n$.

\item \label{Case 1b} \emph{Subcase: there is some non-empty word $\tau$ such that $\tilde{w}_0  = w_0 \tau$ (as words).}   
This is the situation illustrated in Figure~\ref{vK diagram with corridor block}.  We have $z = \tau^{\ell} \tau_0$ (as words) for some $\ell \geq 0$ and some proper (perhaps empty) prefix $\tau_0$ of $\tau$.    
Let $e = \exp(\tau)$.  

\begin{enumerate}[label=\roman*.]
\item \label{Case 1bi} \emph{Assume $e =0$ or $\mu_1 =0$.} Define $w' := w_0 \tau_0 w_1$. We claim that $w'x=yw'$.  If $\mu_1 =0$, then $u_0 = u_1 = \cdots = u_m$.  If $e=0$, then $\tau^{-1} u_m  \tau = u_m$ and so $u_{|\tau_0|} = \tau^{-\ell} u_m     \tau^{\ell} = u_m$.  Either way, removing the region whose boundary is labelled by $zu_{|\tau_0|}z^{-1}u_m^{-1}$  from the van~Kampen diagram of Figure~\ref{vK diagram with corridor block} and identifying the path labelled by $u_m$  with that labelled by $u_{|\tau_0|}$ demonstrates that $w'x=yw'$.  So, because $|w|$ is minimal, $\ell=0$ and $w  \ = \  w_0  \tau_0 w_1$.  But  $| w_0  \tau_0| \leq | w_0  \tau| \leq |y|$ and $|w_1| \leq |x|$, so $|w| \leq n$.   

\item \emph{Assume, instead, that $e \neq 0$ and $\mu_1 = 0$.}   Then $|\lambda| \leq 2n$ because $\mu_1 \neq 0$ means \eqref{lambda estimate} applies.  But $\lambda = e \ell + \exp(\tau_0)$, so, because $e \neq 0$ and      $|\exp(\tau_0)| \leq |\tau| \leq n$, we deduce that $|\ell| \leq 3n$. So  $$|w|  \ = \  | w_0  \tau^{\ell} \tau_0 w_1|  \ \leq \  | w_0 | + |w_1|  + (\ell+1) |\tau|  \  \leq  \  n +  (3n+1) n   \ = \  3n^2 + 2n.$$
\end{enumerate}

\item \emph{Subcase: there is some non-empty word $\tau$ such that   $w_0  = \tilde{w}_0 \tau$ (as words).}   
We  argue that  $|w| \leq n$  or that  $|w| \leq 3n^2 + 2n$ like in the prior subcase, mutatis matantis.  

\end{enumerate}

\item \emph{Case, $u_0 =  c^{\gamma_0} s^{\mu_1}   c^{\gamma_1} \cdots s^{\mu_{k}}   c^{\gamma_{k}}$ with $\mu_1 \neq 0$ and  ($\gamma_0 \neq 0$ or $\gamma_1 \neq 0$).} 
 In this case all the normal form words $u_0, \ldots, u_m$ contain at least one $c$-letter and so there are no $d^{\pm 1}$ letters in $z$ and we are in the $z = a^{\lambda}$ case of Lemma~\ref{block}.  So, by \eqref{u1} and \eqref{lambda estimate},  $|w| =  | w_0 | + |z| + |w_1|    \leq  n + |\lambda| \leq 3n$.

\item \label{Case 3} \emph{Case, $u_0 = b^{\beta_0} c^{\gamma_0} \neq 1$.}   
Lemma~\ref{block} applies and we divide into subcases accordingly.  

\begin{enumerate}[label=\alph*.]
\item \label{Case 3a} \emph{Subcase, $z = a^{\lambda_2}d^{\xi} a^{\lambda_1}$ with $\xi \neq 0$.} 

\begin{enumerate}[label=\roman*.]
\item \emph{Assume $w_0 = \tilde{w}_0$ (as words).}   Due to the absence of $s$-letters in $u_0$, $u_{|\lambda_1|} = u_{|\lambda_1|+|\xi|}$ and   $\xi=0$, because otherwise we could shorten $w$.  So this case does not arise. 

\item  \label{Case 3aii} \emph{Assume, instead, that there is some non-empty word $\tau$ such that $\tilde{w}_0  = w_0 \tau$ (as words).}  
Then $z = \tau^{\ell} \tau_0$ (as words) for some $\ell \geq 0$ and some proper (perhaps empty) prefix $\tau_0$ of $\tau$.  There are now two possibilities.  The first is that  $a^{\lambda_2}d^{\xi}$ is a prefix of $\tau$, in which case $|z| \leq 2|\tau| \leq 2n$, and so    $|w| =  | w_0 | + |z| + |w_1|    \leq 4n$.  
The second is that $z = d^{\xi}$ (as words), but then $\gamma_0 =0$ and $u_0 = \cdots =u_m = b^{\beta_0}$,  and like in Case~\ref{Case 1}\ref{Case 1b}\ref{Case 1bi}, $\ell=0$,  $w  \ = \  w_0  \tau_0 w_1$, and $|w| \leq n$.

\item \emph{Assume, instead, that there is some non-empty word $\tau$ such that $w_0  = \tilde{w}_0 \tau$ (as words).}  
The same argument applies, mutatis mutantis.  
\end{enumerate}

\item \emph{Subcase, $z = a^{\lambda}$.}  In this event, $m = |\lambda|$  and   $u_m = b^{\beta_0} c^{\gamma_0 + \lambda \beta_0}$.

\begin{enumerate}[label=\roman*.]
\item \emph{Assume $\beta_0 = 0$.}  Then $m= \lambda=0$, and  $u_0 = \cdots =u_m = c^{\gamma_0}$,  and like in Cases~\ref{Case 1}\ref{Case 1b}\ref{Case 1bi} and \ref{Case 3}\ref{Case 3a}\ref{Case 3aii}, $\ell=0$,    $w  \ = \  w_0  \tau_0 w_1$, and $|w| \leq n$.     

\item \emph{Assume, instead, that  $\beta \neq 0$.}  We can use the bounds    $|\gamma| \leq |x_1|^2$ and $|\gamma_0 + \lambda \beta_0| \leq |y_1|^2$ from Lemma~\ref{length detail for L}, to get that $m= |\lambda| \leq |\lambda| |\beta_0| \leq |\gamma_0| +   |y_1|^2 \leq  |x_1|^2 +   |y_1|^2$.  Finally,   $|w|   \leq  n + (2 n)^2  \leq 5n^2$ by \eqref{u1} and \eqref{u2}.   
\end{enumerate}
\end{enumerate}

\item \emph{Case, $u_0 = b^{\beta_0} c^{\gamma_0} s^{\mu_1} b^{\beta_1} c^{\gamma_1} \cdots s^{\mu_{k}} b^{\beta_{k}} c^{\gamma_{k}}$ with $\mu_1 \neq 0$ and $\beta_i \neq 0$ for some $i$.}   

Lemma~\ref{block} applies and tells us that $z = a^{\lambda_2}d^{\xi} a^{\lambda_1}$ for some $\lambda_1, \xi, \lambda_2 \in \Z$.      

We have $\gamma_i = - \beta_i \lambda_1$ and $\gamma'_i =  \beta_i \lambda_2$.  Lemma~\ref{length detail for L}  applies to $x_1 =u_0$ and to $y_1 =u_m$ and gives us that $\gamma_i \leq |x_1|^2$ and $\gamma'_i \leq |y_1|^2$.  So, because $\beta_i \neq 0$, we learn that    $|\lambda_1| \leq  |x_1|^2$ and  $|\lambda_2| \leq |y_1|^2$. 

Now, the exponent   of the first power  of $s$ in the normal form of   $u_m$ is   $2^{\lambda_1 + \xi + \lambda_2} \mu_{1}$.  By  Lemma~\ref{length detail for L}, $|2^{\lambda_1 + \xi + \lambda_2} \mu_{1}| \leq 2^{|y_1|}$. And, because $\mu_1 \neq 0$, we deduce that  $|\lambda_1 + \xi + \lambda_2| \leq |y_1|$.  So $|\xi| \leq |y_1| +  |\lambda_1| + |\lambda_2| \leq |y_1| +  |x_1|^2 + |y_1|^2$ and $$|z| \  = \  |\lambda_1| + |\xi| + |\lambda_2| \ \leq \ |y_1| +  2|x_1|^2 + 2|y_1|^2 \ \leq \ n + 8n^2.$$    
Finally,   $|w|  \leq  n + |z| \leq  2n + 8n^2$ by \eqref{u1} and \eqref{u2}. 
\end{enumerate}
\end{proof}

 \begin{rem}
         The defining relations of $G_4$ that contain $c$ are $[a,b]=c,[a,c]=1$, and $[b,c]=1$. The last two of these form corridors in diagrams, except that they may begin  or end at a cell labelled $[a,b]=c$ instead of at the boundary of the diagram.  We will call these sequences of cells \emph{$c$-segments}, and the cells corresponding to $[a,b]=c$ the endpoints of a $c$-segment. A $c$-segment  can close up and  form an annulus (which could be contractible or non-contractible within an annular diagram). In general, if a $c$-segment has no endpoints, then they behave exactly like corridors, so we will refer to them as $c$-arches, radial $c$-segments, and so on. 
\end{rem}

\begin{prop}\label{G4_CL_Upper}   $\CL_{G_4}(n) \ \lesssim \ n^2$.
\end{prop}
\begin{proof}
    Suppose words $u$ and  $v$ represent conjugate elements of $G_4$.  Let   $n = |u|+|v|$  and let $\Omega$ be an annular diagram for $u$ and $v$. In light of Lemma~\ref{no b-rings}, we can assume $\Omega$ contains no  contractible $b$-corridor. 
    
Suppose there is a $b$-arch in $\Omega$ connecting a pair of edges labeled by some $b$ and some $b^{-1}$ in $u$.  Then there is a \emph{pinch} subword $\alpha =  b \hat{u} b^{-1}$ or $\alpha = b^{-1} \hat{u} b$ of some cyclic permutation of $u$ that equals in $G_4$ an element of the subgroup $\langle a,c,d \rangle$. By Lemma~\ref{length detail}, $\alpha$ equals in $G_4$ a word $\hat{\alpha}$ on $a^{\pm 1}, c^{\pm 1}, d^{\pm 1}$ with  
$|\hat{\alpha} |_a + | \hat{\alpha}|_d  \leq |\alpha|$ and  $| \hat{\alpha} |_c  \leq 3 |\alpha|^2$.  

Now, because $b$-arches in $\Omega$ that connect a $b^{\pm 1}$ to a $b^{\mp 1}$ in $u$ nest, there is a cyclic permutation of $u$  that contains disjoint such pinch subwords $\alpha_1, \ldots, \alpha_p$ for some $p$, and is such that \emph{every} letter $b^{\pm 1}$  in $u$ that labels an edge where a $b$-arch meets the $u$-boundary component of $\Omega$ is a letter in one of these  $\alpha_i$. Applying Lemma~\ref{length detail} to all of these subwords and to a similar family of subwords in some cyclic permutation of $v$, we     
obtain words $u'$ and $v'$ (equal, up to cyclic conjugacy, to $u$ and $v$ respectively) such that $|u'|+|v'|\leq 3n^2$ and $|u'|_a+|v'|_a+|u'|_d+|v'|_d\leq n$.
    

    We consider three cases.

   \begin{enumerate}[label=\arabic*.]     
\item    \emph{Case: $\Omega$  has no radial or non-contractible $b$-corridor.} In this event, $\Omega$ contains no $b$-edges whatsoever, so $u'$ and $v'$  represent conjugate elements of   $E   =   \langle a, c,d,s \mid  [a,c],  asa^{-1} s^{-2}, dsd^{-1}   s^{-2} \rangle$.   By Lemma~\ref{CL of E},  $\CL(u',v')$ is at most a constant times $|u'| + |v'|$, and so $\CL(u,v) \leq \lambda n^2$ for a suitable constant $\lambda >0$.

        \begin{figure}
            \centering

\tikzset{every picture/.style={line width=0.75pt}} 

\begin{tikzpicture}[x=0.5pt,y=0.5pt,yscale=-.75,xscale=.75]

\draw   (50,50) -- (600,50) -- (600,400) -- (50,400) -- cycle ;
\draw    (50,230) -- (50,222) ;
\draw [shift={(50,220)}, rotate = 90] [color={rgb, 255:red, 0; green, 0; blue, 0 }  ][line width=0.75]    (10.93,-4.9) .. controls (6.95,-2.3) and (3.31,-0.67) .. (0,0) .. controls (3.31,0.67) and (6.95,2.3) .. (10.93,4.9)   ;
\draw    (600,230) -- (600,222) ;
\draw [shift={(600,220)}, rotate = 90] [color={rgb, 255:red, 0; green, 0; blue, 0 }  ][line width=0.75]    (10.93,-4.9) .. controls (6.95,-2.3) and (3.31,-0.67) .. (0,0) .. controls (3.31,0.67) and (6.95,2.3) .. (10.93,4.9)   ;
\draw   (150,350) -- (200,350) -- (200,400) -- (150,400) -- cycle ;
\draw   (150,300) -- (200,300) -- (200,350) -- (150,350) -- cycle ;
\draw   (150,250) -- (200,250) -- (200,300) -- (150,300) -- cycle ;
\draw   (150,200) -- (200,200) -- (200,250) -- (150,250) -- cycle ;
\draw   (150,150) -- (200,150) -- (200,200) -- (150,200) -- cycle ;
\draw   (150,100) -- (200,100) -- (200,150) -- (150,150) -- cycle ;
\draw   (150,50) -- (200,50) -- (200,100) -- (150,100) -- cycle ;
\draw    (320,400) -- (328,400) ;
\draw [shift={(330,400)}, rotate = 180] [color={rgb, 255:red, 0; green, 0; blue, 0 }  ][line width=0.75]    (10.93,-4.9) .. controls (6.95,-2.3) and (3.31,-0.67) .. (0,0) .. controls (3.31,0.67) and (6.95,2.3) .. (10.93,4.9)   ;
\draw    (320,50) -- (328,50) ;
\draw [shift={(330,50)}, rotate = 180] [color={rgb, 255:red, 0; green, 0; blue, 0 }  ][line width=0.75]    (10.93,-4.9) .. controls (6.95,-2.3) and (3.31,-0.67) .. (0,0) .. controls (3.31,0.67) and (6.95,2.3) .. (10.93,4.9)   ;
\draw   (450,350) -- (500,350) -- (500,400) -- (450,400) -- cycle ;
\draw   (450,300) -- (500,300) -- (500,350) -- (450,350) -- cycle ;
\draw   (450,250) -- (500,250) -- (500,300) -- (450,300) -- cycle ;
\draw   (450,200) -- (500,200) -- (500,250) -- (450,250) -- cycle ;
\draw   (450,150) -- (500,150) -- (500,200) -- (450,200) -- cycle ;
\draw   (450,100) -- (500,100) -- (500,150) -- (450,150) -- cycle ;
\draw   (450,50) -- (500,50) -- (500,100) -- (450,100) -- cycle ;

\draw (167,80.4) node [anchor=north west][inner sep=0.75pt]    {$c$};
\draw (167,130.4) node [anchor=north west][inner sep=0.75pt]    {$c$};
\draw (167,180.4) node [anchor=north west][inner sep=0.75pt]    {$c$};
\draw (167,230.4) node [anchor=north west][inner sep=0.75pt]    {$c$};
\draw (167,280.4) node [anchor=north west][inner sep=0.75pt]    {$c$};
\draw (167,330.4) node [anchor=north west][inner sep=0.75pt]    {$c$};
\draw (167,380.4) node [anchor=north west][inner sep=0.75pt]    {$c$};
\draw (126,200) node [anchor=north west][inner sep=0.75pt]  [font=\large]  {$\begin{cases}
 & \\
 & 
\end{cases}$};
\draw (100,230) node [anchor=north west][inner sep=0.75pt]    {$a^{m}$};
\draw (524,293) node [anchor=north west][inner sep=0.75pt]  [font=\large,rotate=-180]  {$\begin{cases}
 & \\
 & 
\end{cases}$};
\draw (525,230) node [anchor=north west][inner sep=0.75pt]    {$a^{m'}$};
\draw (305,60.4) node [anchor=north west][inner sep=0.75pt]    {$w_1$};
\draw (305,365) node [anchor=north west][inner sep=0.75pt]    {$w_2$};
\draw (467,80.4) node [anchor=north west][inner sep=0.75pt]    {$c$};
\draw (467,130.4) node [anchor=north west][inner sep=0.75pt]    {$c$};
\draw (467,180.4) node [anchor=north west][inner sep=0.75pt]    {$c$};
\draw (467,230.4) node [anchor=north west][inner sep=0.75pt]    {$c$};
\draw (467,280.4) node [anchor=north west][inner sep=0.75pt]    {$c$};
\draw (467,330.4) node [anchor=north west][inner sep=0.75pt]    {$c$};
\draw (467,380.4) node [anchor=north west][inner sep=0.75pt]    {$c$};
\draw (310,20) node [anchor=north west][inner sep=0.75pt]    {$u$};
\draw (310,410) node [anchor=north west][inner sep=0.75pt]    {$v$};

\end{tikzpicture}

            \caption{Two radial $c$-segments}
            \label{fig:radial-c-segments}
        \end{figure}
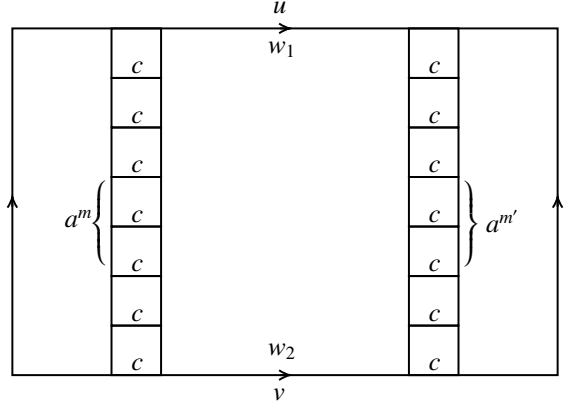

\medskip

\item  \emph{Case: $\Omega$ has at least one non-contractible $b$-annulus $\beta$.}

Then $u',v'$ contain no $b$-letters, since by their definition none of their $b$-letters can be part of a $b$-arch and any radial $b$-corridor must cross $\beta$ and so cannot exist. We claim that there exists a constant  $\rho >0$ (independent of $u'$ and $v'$) and words $\tilde{u}$ and $\tilde{v}$ on $\{ a,b,c,d \}$ such that $u' \sim \tilde{u}$ and $v' \sim \tilde{v}$ in $G_4$, and \begin{align}
|\tilde{u}|_A, |\tilde{v}|_A & \leq \rho n + \rho, \label{tildeu bound} \\ 
\CL(u', \tilde{u}), \CL(v', \tilde{v}) & \leq  \rho n. \label{utilde u bound} 
\end{align}
We will argue this for $\tilde{v}$ in  two subcases. The argument for $\tilde{u}$ is the same mutatis mutantis.  

Without loss of generality, take $\beta$ to be the outermost $b$-corridor.

\begin{enumerate}[label=\alph*.]  
    \item  \label{case 1}
Suppose there is a $c$-segment running from the outer side of $\beta$ to the outer boundary, as in Figure \ref{fig:c_seg_meets_b_corr}. Then, up to cyclic permutation, $v'$ is equal in $G_4$ to the word $\tilde{v}$ along the dotted line, which is an element of $A$. Every $s$-letter in $v'$ corresponds lexically to an $s$-letter in $v$, so there is a path in $v'$ from the $c$-edge of the given $c$-segment to some vertex belonging to our original word $v$ which does not contain any $s$-edges. Thus we may assume that $\tilde{v}$ is conjugate to $v'$ (and thus to $v$) in $A$. Since $A$ is undistorted in $G_4$ (Corollary~\ref{G4Cor}), there therefore exists a constant $\rho >0$ such that $|\tilde{v}|_A\leq  \rho |v|+ \rho\leq \rho n+ \rho$, and since $\tilde{v}$ equals a cyclic permutation ov $v'$, we also have $\CL(v',\tilde{v})\leq |v'|<n$.

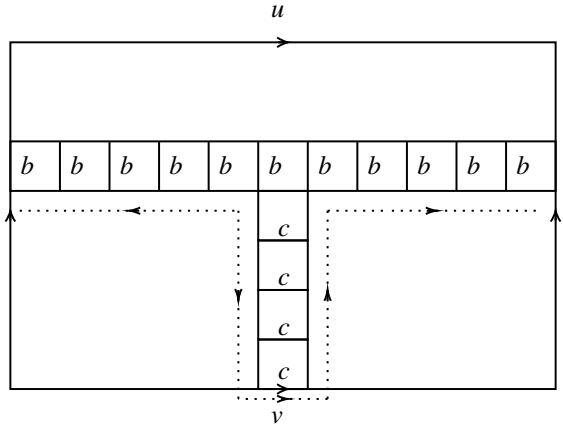
\begin{figure}
    \centering

\tikzset{every picture/.style={line width=0.75pt}} 

\begin{tikzpicture}[x=0.5pt,y=0.5pt,yscale=-.75,xscale=.75]

\draw   (50,50) -- (600,50) -- (600,400) -- (50,400) -- cycle ;
\draw    (50,230) -- (50,222) ;
\draw [shift={(50,220)}, rotate = 90] [color={rgb, 255:red, 0; green, 0; blue, 0 }  ][line width=0.75]    (10.93,-4.9) .. controls (6.95,-2.3) and (3.31,-0.67) .. (0,0) .. controls (3.31,0.67) and (6.95,2.3) .. (10.93,4.9)   ;
\draw    (600,230) -- (600,222) ;
\draw [shift={(600,220)}, rotate = 90] [color={rgb, 255:red, 0; green, 0; blue, 0 }  ][line width=0.75]    (10.93,-4.9) .. controls (6.95,-2.3) and (3.31,-0.67) .. (0,0) .. controls (3.31,0.67) and (6.95,2.3) .. (10.93,4.9)   ;
\draw   (300,350) -- (350,350) -- (350,400) -- (300,400) -- cycle ;
\draw   (300,300) -- (350,300) -- (350,350) -- (300,350) -- cycle ;
\draw   (300,250) -- (350,250) -- (350,300) -- (300,300) -- cycle ;
\draw   (300,200) -- (350,200) -- (350,250) -- (300,250) -- cycle ;
\draw    (320,400) -- (328,400) ;
\draw [shift={(330,400)}, rotate = 180] [color={rgb, 255:red, 0; green, 0; blue, 0 }  ][line width=0.75]    (10.93,-4.9) .. controls (6.95,-2.3) and (3.31,-0.67) .. (0,0) .. controls (3.31,0.67) and (6.95,2.3) .. (10.93,4.9)   ;
\draw    (320,50) -- (328,50) ;
\draw [shift={(330,50)}, rotate = 180] [color={rgb, 255:red, 0; green, 0; blue, 0 }  ][line width=0.75]    (10.93,-4.9) .. controls (6.95,-2.3) and (3.31,-0.67) .. (0,0) .. controls (3.31,0.67) and (6.95,2.3) .. (10.93,4.9)   ;
\draw   (50,200) -- (50,150) -- (600,150) -- (600,200) -- (50,200) -- cycle ;
\draw    (100,150) -- (100,200) ;
\draw    (150,150) -- (150,200) ;
\draw    (200,150) -- (200,200) ;
\draw    (250,150) -- (250,200) ;
\draw    (300,150) -- (300,200) ;
\draw    (350,150) -- (350,200) ;
\draw    (400,150) -- (400,200) ;
\draw    (450,150) -- (450,200) ;
\draw    (500,150) -- (500,200) ;
\draw    (550,150) -- (550,200) ;
\draw  [dash pattern={on 0.84pt off 2.51pt}]  (280,220) -- (280,410) ;
\draw  [dash pattern={on 0.84pt off 2.51pt}]  (370,410) -- (280,410) ;
\draw  [dash pattern={on 0.84pt off 2.51pt}]  (370,220) -- (370,410) ;
\draw  [dash pattern={on 0.84pt off 2.51pt}]  (370,220) -- (430,220) ;
\draw    (325,410) -- (328,410) ;
\draw [shift={(330,410)}, rotate = 180] [color={rgb, 255:red, 0; green, 0; blue, 0 }  ][line width=0.75]    (10.93,-3.29) .. controls (6.95,-1.4) and (3.31,-0.3) .. (0,0) .. controls (3.31,0.3) and (6.95,1.4) .. (10.93,3.29)   ;
\draw    (180,220) -- (172,220) ;
\draw [shift={(170,220)}, rotate = 360] [color={rgb, 255:red, 0; green, 0; blue, 0 }  ][line width=0.75]    (10.93,-3.29) .. controls (6.95,-1.4) and (3.31,-0.3) .. (0,0) .. controls (3.31,0.3) and (6.95,1.4) .. (10.93,3.29)   ;
\draw    (370,310) -- (370,302) ;
\draw [shift={(370,300)}, rotate = 90] [color={rgb, 255:red, 0; green, 0; blue, 0 }  ][line width=0.75]    (10.93,-3.29) .. controls (6.95,-1.4) and (3.31,-0.3) .. (0,0) .. controls (3.31,0.3) and (6.95,1.4) .. (10.93,3.29)   ;
\draw    (280,300) -- (280,308) ;
\draw [shift={(280,310)}, rotate = 270] [color={rgb, 255:red, 0; green, 0; blue, 0 }  ][line width=0.75]    (10.93,-3.29) .. controls (6.95,-1.4) and (3.31,-0.3) .. (0,0) .. controls (3.31,0.3) and (6.95,1.4) .. (10.93,3.29)   ;
\draw  [dash pattern={on 0.84pt off 2.51pt}]  (580,220) -- (430,220) ;
\draw    (470,220) -- (478,220) ;
\draw [shift={(480,220)}, rotate = 180] [color={rgb, 255:red, 0; green, 0; blue, 0 }  ][line width=0.75]    (10.93,-3.29) .. controls (6.95,-1.4) and (3.31,-0.3) .. (0,0) .. controls (3.31,0.3) and (6.95,1.4) .. (10.93,3.29)   ;
\draw  [dash pattern={on 0.84pt off 2.51pt}]  (276,220) -- (270,220) -- (56,220) ;

\draw (317,230.4) node [anchor=north west][inner sep=0.75pt]    {$c$};
\draw (317,280.4) node [anchor=north west][inner sep=0.75pt]    {$c$};
\draw (317,330.4) node [anchor=north west][inner sep=0.75pt]    {$c$};
\draw (317,375.4) node [anchor=north west][inner sep=0.75pt]    {$c$};
\draw (51,162.4) node [anchor=north west][inner sep=0.75pt]    {$\ b$};
\draw (101,162.4) node [anchor=north west][inner sep=0.75pt]    {$\ b$};
\draw (151,162.4) node [anchor=north west][inner sep=0.75pt]    {$\ b$};
\draw (201,162.4) node [anchor=north west][inner sep=0.75pt]    {$\ b$};
\draw (251,162.4) node [anchor=north west][inner sep=0.75pt]    {$\ b$};
\draw (301,162.4) node [anchor=north west][inner sep=0.75pt]    {$\ b$};
\draw (351,162.4) node [anchor=north west][inner sep=0.75pt]    {$\ b$};
\draw (401,162.4) node [anchor=north west][inner sep=0.75pt]    {$\ b$};
\draw (451,162.4) node [anchor=north west][inner sep=0.75pt]    {$\ b$};
\draw (501,162.4) node [anchor=north west][inner sep=0.75pt]    {$\ b$};
\draw (551,162.4) node [anchor=north west][inner sep=0.75pt]    {$\ b$};
\draw (310,10) node [anchor=north west][inner sep=0.75pt]    {$u$};
\draw (310,420) node [anchor=north west][inner sep=0.75pt]    {$v$};

\end{tikzpicture}

    \caption{$c$-segment meeting a non-contractable $b$-corridor}
    \label{fig:c_seg_meets_b_corr}
\end{figure}

\item 
Suppose there are no such $c$-segments. Then every $c$-segment leaving $\beta$ must return to it, so we have the situation illustrated in Figure \ref{fig:no c segment between B and boundary}. Let $\tilde{v}\in C = F( a, d)$ (recall, Lemma~\ref{structureG4}\eqref{C subgroup property}) be the word along the dotted line. Every $a$- or $d$-letter of the reduced form of $\tilde{v}$ is part of a corridor running from $\beta$ to the outer boundary,  so $|\tilde{v}|_C\leq |v'|_a+|v'|_d\leq n$.

Now, every $c^{\pm 1}$ in $v'$ labels the initial edge of a $c$-arch in $\Omega$, and so excising these, we learn that that a cyclic conjugate of $v'$  equals in $G_4$ a word  $v'' \in \langle a,d,s\rangle = B$. So $\CL(v',v'') \leq |v'| + |v''|$ and,  since the words along the sides of a $c$-arch are the same, $|v''| \leq |v'|$. Since $v''$ and $\tilde{v}$ represent elements of $B$ that are conjugate in $G_4$, they are conjugate in $B$ by Corollary~\ref{G4Cor}(\ref{conjInB}). By Lemma~\ref{Blinear},  $\CL(v'',\tilde{v})$  is at most a constant times $|v''| + |\tilde{v}|$  (in $B$ and therefore in $G_4$). Combining these estimates give that $\CL(v',\tilde{v}) \leq   \CL(v',v'')  +  \CL(v'',\tilde{v})$  is at most $\rho|v'|<\rho n$.

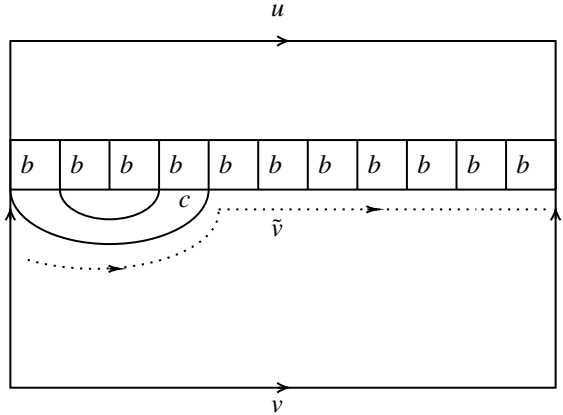
\begin{figure}
    \centering

\tikzset{every picture/.style={line width=0.75pt}} 

\begin{tikzpicture}[x=0.5pt,y=0.5pt,yscale=-.75,xscale=.75]

\draw   (50,50) -- (600,50) -- (600,400) -- (50,400) -- cycle ;
\draw    (50,230) -- (50,222) ;
\draw [shift={(50,220)}, rotate = 90] [color={rgb, 255:red, 0; green, 0; blue, 0 }  ][line width=0.75]    (10.93,-4.9) .. controls (6.95,-2.3) and (3.31,-0.67) .. (0,0) .. controls (3.31,0.67) and (6.95,2.3) .. (10.93,4.9)   ;
\draw    (600,230) -- (600,222) ;
\draw [shift={(600,220)}, rotate = 90] [color={rgb, 255:red, 0; green, 0; blue, 0 }  ][line width=0.75]    (10.93,-4.9) .. controls (6.95,-2.3) and (3.31,-0.67) .. (0,0) .. controls (3.31,0.67) and (6.95,2.3) .. (10.93,4.9)   ;
\draw    (320,400) -- (328,400) ;
\draw [shift={(330,400)}, rotate = 180] [color={rgb, 255:red, 0; green, 0; blue, 0 }  ][line width=0.75]    (10.93,-4.9) .. controls (6.95,-2.3) and (3.31,-0.67) .. (0,0) .. controls (3.31,0.67) and (6.95,2.3) .. (10.93,4.9)   ;
\draw    (320,50) -- (328,50) ;
\draw [shift={(330,50)}, rotate = 180] [color={rgb, 255:red, 0; green, 0; blue, 0 }  ][line width=0.75]    (10.93,-4.9) .. controls (6.95,-2.3) and (3.31,-0.67) .. (0,0) .. controls (3.31,0.67) and (6.95,2.3) .. (10.93,4.9)   ;
\draw   (50,200) -- (50,150) -- (600,150) -- (600,200) -- (50,200) -- cycle ;
\draw    (100,150) -- (100,200) ;
\draw    (150,150) -- (150,200) ;
\draw    (200,150) -- (200,200) ;
\draw    (250,150) -- (250,200) ;
\draw    (300,150) -- (300,200) ;
\draw    (350,150) -- (350,200) ;
\draw    (400,150) -- (400,200) ;
\draw    (450,150) -- (450,200) ;
\draw    (500,150) -- (500,200) ;
\draw    (550,150) -- (550,200) ;
\draw  [dash pattern={on 0.84pt off 2.51pt}]  (260,220) -- (520,220) ;
\draw  [draw opacity=0] (200,200) .. controls (200,216.57) and (177.61,230) .. (150,230) .. controls (122.39,230) and (100,216.57) .. (100,200) -- (150,200) -- cycle ; \draw   (200,200) .. controls (200,216.57) and (177.61,230) .. (150,230) .. controls (122.39,230) and (100,216.57) .. (100,200) ;  
\draw  [draw opacity=0] (250,200) .. controls (250,200) and (250,200) .. (250,200) .. controls (250,230.38) and (205.23,255) .. (150,255) .. controls (94.77,255) and (50,230.38) .. (50,200) -- (150,200) -- cycle ; \draw   (250,200) .. controls (250,200) and (250,200) .. (250,200) .. controls (250,230.38) and (205.23,255) .. (150,255) .. controls (94.77,255) and (50,230.38) .. (50,200) ;  
\draw  [dash pattern={on 0.84pt off 2.51pt}]  (590,220) -- (520,220) ;
\draw    (410,220) -- (418,220) ;
\draw [shift={(420,220)}, rotate = 180] [color={rgb, 255:red, 0; green, 0; blue, 0 }  ][line width=0.75]    (10.93,-3.29) .. controls (6.95,-1.4) and (3.31,-0.3) .. (0,0) .. controls (3.31,0.3) and (6.95,1.4) .. (10.93,3.29)   ;
\draw  [draw opacity=0][dash pattern={on 0.84pt off 2.51pt}] (260,222.5) .. controls (260,254.26) and (204.04,280) .. (135,280) .. controls (110,280) and (86.71,276.62) .. (67.18,270.81) -- (135,222.5) -- cycle ; \draw  [dash pattern={on 0.84pt off 2.51pt}] (260,222.5) .. controls (260,254.26) and (204.04,280) .. (135,280) .. controls (110,280) and (86.71,276.62) .. (67.18,270.81) ;  
\draw    (150,280) -- (158,280) ;
\draw [shift={(160,280)}, rotate = 180] [color={rgb, 255:red, 0; green, 0; blue, 0 }  ][line width=0.75]    (10.93,-3.29) .. controls (6.95,-1.4) and (3.31,-0.3) .. (0,0) .. controls (3.31,0.3) and (6.95,1.4) .. (10.93,3.29)   ;

\draw (51,162.4) node [anchor=north west][inner sep=0.75pt]    {$\ b$};
\draw (101,162.4) node [anchor=north west][inner sep=0.75pt]    {$\ b$};
\draw (151,162.4) node [anchor=north west][inner sep=0.75pt]    {$\ b$};
\draw (201,162.4) node [anchor=north west][inner sep=0.75pt]    {$\ b$};
\draw (251,162.4) node [anchor=north west][inner sep=0.75pt]    {$\ b$};
\draw (301,162.4) node [anchor=north west][inner sep=0.75pt]    {$\ b$};
\draw (351,162.4) node [anchor=north west][inner sep=0.75pt]    {$\ b$};
\draw (401,162.4) node [anchor=north west][inner sep=0.75pt]    {$\ b$};
\draw (451,162.4) node [anchor=north west][inner sep=0.75pt]    {$\ b$};
\draw (501,162.4) node [anchor=north west][inner sep=0.75pt]    {$\ b$};
\draw (551,162.4) node [anchor=north west][inner sep=0.75pt]    {$\ b$};
\draw (217,202.4) node [anchor=north west][inner sep=0.75pt]    {$c$};
\draw (310,10) node [anchor=north west][inner sep=0.75pt]    {$u$};
\draw (310,410) node [anchor=north west][inner sep=0.75pt]    {$v$};
\draw (310,227) node [anchor=north west][inner sep=0.75pt]    {$\tilde{v}$};

\end{tikzpicture}

    \caption{no $c$-segment running between a non-contractable $b$-corridor and the boundary}
    \label{fig:no c segment between B and boundary}
\end{figure}
\end{enumerate}

With these two cases complete, note that Lemma~\ref{Aquadratic} applies to  $\tilde{u}$ and $\tilde{v}$, since they represent conjugate elements of $A$. Moreover, this lemma combines with  \eqref{tildeu bound} and \eqref{utilde u bound} to give
 $$\CL(u,v) \ \leq \  \CL(u,\tilde{u}) + \CL(\tilde{u},\tilde{v}) + \CL(\tilde{v},v)  \ \leq \mu n^2,$$ for a suitable constant $\mu>0$, as required.

\medskip

\item \emph{Case:  $\Omega$ contains a radial $b$-corridor.} 
If there is a $c$-segment running between each pair of consecutive radial $b$-corridors, then we have the situation illustrated in Figure~\ref{fig:connected radial b-corridors}. Up to cyclic conjugation, $u$ and $v$ are equal in $G_4$ to the words along the two dotted lines, which are words on $\{a,b,c,d\}$.  So there are cyclic conjugates $u_0$ and $v_0$ of $u$ and $v$, respectively, that represent elements of $\langle a,b,c,d \rangle$ and  the minimal  length words  $\tilde{u}$ and $\tilde{v}$ on $\{a,b,c,d\}$ equaling $u_0$ and $v_0$ in $G_4$, respectively, have   $|\tilde{u}| + |\tilde{v}| \lesssim n$ by  Corollary~\ref{G4Cor}\eqref{ABUndistorted}.  And because $u \sim v$ in $G_4$ and $A$ is a retract of $G_4$, we have  $\tilde{u} \sim \tilde{v}$ in $A$.  And then, by Lemma~\ref{Aquadratic}, there exists a word $w$ such that  $\tilde{u} w  = w \tilde{v}$ in $A$ and $|w| \leq  \mu n^2$, for a suitable constant $\mu >0$.  But then,  $u_0 w  = w  v_0$ in $G_4$, and the claim follows.   

Alternatively, suppose some pair of radial $b$-corridors has no $c$-segment running between them, as shown in Figure \ref{fig:disconnected radial b-corridors}. Then, up to cyclic conjugation, $u$ and $v$ are \textit{conjugate}  via a word $w$ on $a$ and $d$ (the word read along the dotted line in the figure). The claim then follows from Lemma~\ref{quadratic bound from stack} (applied to the equality $u w  = w  v$). 
\end{enumerate}
 \end{proof}

\begin{figure}
    \centering

\tikzset{every picture/.style={line width=0.75pt}} 

\begin{tikzpicture}[x=0.5pt,y=0.5pt,yscale=-.75,xscale=.75]

\draw   (50,50) -- (600,50) -- (600,400) -- (50,400) -- cycle ;
\draw    (50,230) -- (50,222) ;
\draw [shift={(50,220)}, rotate = 90] [color={rgb, 255:red, 0; green, 0; blue, 0 }  ][line width=0.75]    (10.93,-4.9) .. controls (6.95,-2.3) and (3.31,-0.67) .. (0,0) .. controls (3.31,0.67) and (6.95,2.3) .. (10.93,4.9)   ;
\draw    (600,230) -- (600,222) ;
\draw [shift={(600,220)}, rotate = 90] [color={rgb, 255:red, 0; green, 0; blue, 0 }  ][line width=0.75]    (10.93,-4.9) .. controls (6.95,-2.3) and (3.31,-0.67) .. (0,0) .. controls (3.31,0.67) and (6.95,2.3) .. (10.93,4.9)   ;
\draw   (450,350) -- (500,350) -- (500,400) -- (450,400) -- cycle ;
\draw   (450,300) -- (500,300) -- (500,350) -- (450,350) -- cycle ;
\draw   (450,250) -- (500,250) -- (500,300) -- (450,300) -- cycle ;
\draw   (450,200) -- (500,200) -- (500,250) -- (450,250) -- cycle ;
\draw    (320,400) -- (328,400) ;
\draw [shift={(330,400)}, rotate = 180] [color={rgb, 255:red, 0; green, 0; blue, 0 }  ][line width=0.75]    (10.93,-4.9) .. controls (6.95,-2.3) and (3.31,-0.67) .. (0,0) .. controls (3.31,0.67) and (6.95,2.3) .. (10.93,4.9)   ;
\draw    (320,50) -- (328,50) ;
\draw [shift={(330,50)}, rotate = 180] [color={rgb, 255:red, 0; green, 0; blue, 0 }  ][line width=0.75]    (10.93,-4.9) .. controls (6.95,-2.3) and (3.31,-0.67) .. (0,0) .. controls (3.31,0.67) and (6.95,2.3) .. (10.93,4.9)   ;
\draw   (450,150) -- (500,150) -- (500,200) -- (450,200) -- cycle ;
\draw   (450,100) -- (500,100) -- (500,150) -- (450,150) -- cycle ;
\draw   (450,50) -- (500,50) -- (500,100) -- (450,100) -- cycle ;
\draw   (150,350) -- (200,350) -- (200,400) -- (150,400) -- cycle ;
\draw   (150,300) -- (200,300) -- (200,350) -- (150,350) -- cycle ;
\draw   (150,250) -- (200,250) -- (200,300) -- (150,300) -- cycle ;
\draw   (150,200) -- (200,200) -- (200,250) -- (150,250) -- cycle ;
\draw   (150,150) -- (200,150) -- (200,200) -- (150,200) -- cycle ;
\draw   (150,100) -- (200,100) -- (200,150) -- (150,150) -- cycle ;
\draw   (150,50) -- (200,50) -- (200,100) -- (150,100) -- cycle ;
\draw   (200,200) -- (250,200) -- (250,250) -- (200,250) -- cycle ;
\draw   (250,200) -- (300,200) -- (300,250) -- (250,250) -- cycle ;
\draw   (300,200) -- (350,200) -- (350,250) -- (300,250) -- cycle ;
\draw   (350,200) -- (400,200) -- (400,250) -- (350,250) -- cycle ;
\draw   (400,200) -- (450,200) -- (450,250) -- (400,250) -- cycle ;
\draw   (50,100) -- (100,100) -- (100,150) -- (50,150) -- cycle ;
\draw   (100,100) -- (150,100) -- (150,150) -- (100,150) -- cycle ;
\draw   (500,100) -- (550,100) -- (550,150) -- (500,150) -- cycle ;
\draw   (550,100) -- (600,100) -- (600,150) -- (550,150) -- cycle ;
\draw   (50,300) -- (100,300) -- (100,350) -- (50,350) -- cycle ;
\draw   (100,300) -- (150,300) -- (150,350) -- (100,350) -- cycle ;
\draw   (500,300) -- (550,300) -- (550,350) -- (500,350) -- cycle ;
\draw   (550,300) -- (600,300) -- (600,350) -- (550,350) -- cycle ;
\draw  [dash pattern={on 0.84pt off 2.51pt}]  (60,90) -- (108,90) ;
\draw [shift={(110,90)}, rotate = 180] [color={rgb, 255:red, 0; green, 0; blue, 0 }  ][line width=0.75]    (10.93,-3.29) .. controls (6.95,-1.4) and (3.31,-0.3) .. (0,0) .. controls (3.31,0.3) and (6.95,1.4) .. (10.93,3.29)   ;
\draw  [dash pattern={on 0.84pt off 2.51pt}]  (110,90) -- (140,90) ;
\draw  [dash pattern={on 0.84pt off 2.51pt}]  (140,90) -- (140,40) ;
\draw  [dash pattern={on 0.84pt off 2.51pt}]  (140,40) -- (178,40) ;
\draw [shift={(180,40)}, rotate = 180] [color={rgb, 255:red, 0; green, 0; blue, 0 }  ][line width=0.75]    (10.93,-3.29) .. controls (6.95,-1.4) and (3.31,-0.3) .. (0,0) .. controls (3.31,0.3) and (6.95,1.4) .. (10.93,3.29)   ;
\draw  [dash pattern={on 0.84pt off 2.51pt}]  (210,40) -- (180,40) ;
\draw  [dash pattern={on 0.84pt off 2.51pt}]  (210,40) -- (210,128) ;
\draw [shift={(210,130)}, rotate = 270] [color={rgb, 255:red, 0; green, 0; blue, 0 }  ][line width=0.75]    (10.93,-3.29) .. controls (6.95,-1.4) and (3.31,-0.3) .. (0,0) .. controls (3.31,0.3) and (6.95,1.4) .. (10.93,3.29)   ;
\draw  [dash pattern={on 0.84pt off 2.51pt}]  (210,190) -- (210,130) ;
\draw  [dash pattern={on 0.84pt off 2.51pt}]  (210,190) -- (318,190) ;
\draw [shift={(320,190)}, rotate = 180] [color={rgb, 255:red, 0; green, 0; blue, 0 }  ][line width=0.75]    (10.93,-3.29) .. controls (6.95,-1.4) and (3.31,-0.3) .. (0,0) .. controls (3.31,0.3) and (6.95,1.4) .. (10.93,3.29)   ;
\draw  [dash pattern={on 0.84pt off 2.51pt}]  (320,190) -- (440,190) ;
\draw  [dash pattern={on 0.84pt off 2.51pt}]  (440,190) -- (440,102) ;
\draw [shift={(440,100)}, rotate = 90] [color={rgb, 255:red, 0; green, 0; blue, 0 }  ][line width=0.75]    (10.93,-3.29) .. controls (6.95,-1.4) and (3.31,-0.3) .. (0,0) .. controls (3.31,0.3) and (6.95,1.4) .. (10.93,3.29)   ;
\draw  [dash pattern={on 0.84pt off 2.51pt}]  (440,40) -- (440,100) ;

\draw  [dash pattern={on 0.84pt off 2.51pt}]  (440,40) -- (478,40) ;
\draw [shift={(480,40)}, rotate = 180] [color={rgb, 255:red, 0; green, 0; blue, 0 }  ][line width=0.75]    (10.93,-3.29) .. controls (6.95,-1.4) and (3.31,-0.3) .. (0,0) .. controls (3.31,0.3) and (6.95,1.4) .. (10.93,3.29)   ;
\draw  [dash pattern={on 0.84pt off 2.51pt}]  (510,40) -- (480,40) ;
\draw  [dash pattern={on 0.84pt off 2.51pt}]  (510,90) -- (558,90) ;
\draw [shift={(560,90)}, rotate = 180] [color={rgb, 255:red, 0; green, 0; blue, 0 }  ][line width=0.75]    (10.93,-3.29) .. controls (6.95,-1.4) and (3.31,-0.3) .. (0,0) .. controls (3.31,0.3) and (6.95,1.4) .. (10.93,3.29)   ;
\draw  [dash pattern={on 0.84pt off 2.51pt}]  (560,90) -- (590,90) ;
\draw  [dash pattern={on 0.84pt off 2.51pt}]  (510,90) -- (510,40) ;
\draw  [dash pattern={on 0.84pt off 2.51pt}]  (210,260) -- (318,260) ;
\draw [shift={(320,260)}, rotate = 180] [color={rgb, 255:red, 0; green, 0; blue, 0 }  ][line width=0.75]    (10.93,-3.29) .. controls (6.95,-1.4) and (3.31,-0.3) .. (0,0) .. controls (3.31,0.3) and (6.95,1.4) .. (10.93,3.29)   ;
\draw  [dash pattern={on 0.84pt off 2.51pt}]  (320,260) -- (440,260) ;
\draw  [dash pattern={on 0.84pt off 2.51pt}]  (210,410) -- (210,322) ;
\draw [shift={(210,320)}, rotate = 90] [color={rgb, 255:red, 0; green, 0; blue, 0 }  ][line width=0.75]    (10.93,-3.29) .. controls (6.95,-1.4) and (3.31,-0.3) .. (0,0) .. controls (3.31,0.3) and (6.95,1.4) .. (10.93,3.29)   ;
\draw  [dash pattern={on 0.84pt off 2.51pt}]  (210,260) -- (210,320) ;

\draw  [dash pattern={on 0.84pt off 2.51pt}]  (440,260) -- (440,348) ;
\draw [shift={(440,350)}, rotate = 270] [color={rgb, 255:red, 0; green, 0; blue, 0 }  ][line width=0.75]    (10.93,-3.29) .. controls (6.95,-1.4) and (3.31,-0.3) .. (0,0) .. controls (3.31,0.3) and (6.95,1.4) .. (10.93,3.29)   ;
\draw  [dash pattern={on 0.84pt off 2.51pt}]  (440,410) -- (440,350) ;
\draw  [dash pattern={on 0.84pt off 2.51pt}]  (140,410) -- (178,410) ;
\draw [shift={(180,410)}, rotate = 180] [color={rgb, 255:red, 0; green, 0; blue, 0 }  ][line width=0.75]    (10.93,-3.29) .. controls (6.95,-1.4) and (3.31,-0.3) .. (0,0) .. controls (3.31,0.3) and (6.95,1.4) .. (10.93,3.29)   ;
\draw  [dash pattern={on 0.84pt off 2.51pt}]  (210,410) -- (180,410) ;
\draw  [dash pattern={on 0.84pt off 2.51pt}]  (140,410) -- (140,360) ;
\draw  [dash pattern={on 0.84pt off 2.51pt}]  (60,360) -- (108,360) ;
\draw [shift={(110,360)}, rotate = 180] [color={rgb, 255:red, 0; green, 0; blue, 0 }  ][line width=0.75]    (10.93,-3.29) .. controls (6.95,-1.4) and (3.31,-0.3) .. (0,0) .. controls (3.31,0.3) and (6.95,1.4) .. (10.93,3.29)   ;
\draw  [dash pattern={on 0.84pt off 2.51pt}]  (110,360) -- (140,360) ;
\draw  [dash pattern={on 0.84pt off 2.51pt}]  (440,410) -- (478,410) ;
\draw [shift={(480,410)}, rotate = 180] [color={rgb, 255:red, 0; green, 0; blue, 0 }  ][line width=0.75]    (10.93,-3.29) .. controls (6.95,-1.4) and (3.31,-0.3) .. (0,0) .. controls (3.31,0.3) and (6.95,1.4) .. (10.93,3.29)   ;
\draw  [dash pattern={on 0.84pt off 2.51pt}]  (510,410) -- (480,410) ;
\draw  [dash pattern={on 0.84pt off 2.51pt}]  (510,410) -- (510,360) ;
\draw  [dash pattern={on 0.84pt off 2.51pt}]  (510,360) -- (558,360) ;
\draw [shift={(560,360)}, rotate = 180] [color={rgb, 255:red, 0; green, 0; blue, 0 }  ][line width=0.75]    (10.93,-3.29) .. controls (6.95,-1.4) and (3.31,-0.3) .. (0,0) .. controls (3.31,0.3) and (6.95,1.4) .. (10.93,3.29)   ;
\draw  [dash pattern={on 0.84pt off 2.51pt}]  (560,360) -- (590,360) ;

\draw (467,222.4) node [anchor=north west][inner sep=0.75pt]    {$b$};
\draw (467,272.4) node [anchor=north west][inner sep=0.75pt]    {$b$};
\draw (467,322.4) node [anchor=north west][inner sep=0.75pt]    {$b$};
\draw (467,372.4) node [anchor=north west][inner sep=0.75pt]    {$b$};
\draw (467,72.4) node [anchor=north west][inner sep=0.75pt]    {$b$};
\draw (467,122.4) node [anchor=north west][inner sep=0.75pt]    {$b$};
\draw (467,172.4) node [anchor=north west][inner sep=0.75pt]    {$b$};
\draw (167,222.4) node [anchor=north west][inner sep=0.75pt]    {$b$};
\draw (167,272.4) node [anchor=north west][inner sep=0.75pt]    {$b$};
\draw (167,322.4) node [anchor=north west][inner sep=0.75pt]    {$b$};
\draw (167,372.4) node [anchor=north west][inner sep=0.75pt]    {$b$};
\draw (167,72.4) node [anchor=north west][inner sep=0.75pt]    {$b$};
\draw (167,122.4) node [anchor=north west][inner sep=0.75pt]    {$b$};
\draw (167,172.4) node [anchor=north west][inner sep=0.75pt]    {$b$};
\draw (207,212.4) node [anchor=north west][inner sep=0.75pt]    {$c$};
\draw (257,212.4) node [anchor=north west][inner sep=0.75pt]    {$c$};
\draw (307,212.4) node [anchor=north west][inner sep=0.75pt]    {$c$};
\draw (357,212.4) node [anchor=north west][inner sep=0.75pt]    {$c$};
\draw (407,212.4) node [anchor=north west][inner sep=0.75pt]    {$c$};
\draw (57,112.4) node [anchor=north west][inner sep=0.75pt]    {$c$};
\draw (107,112.4) node [anchor=north west][inner sep=0.75pt]    {$c$};
\draw (507,112.4) node [anchor=north west][inner sep=0.75pt]    {$c$};
\draw (557,112.4) node [anchor=north west][inner sep=0.75pt]    {$c$};
\draw (57,312.4) node [anchor=north west][inner sep=0.75pt]    {$c$};
\draw (107,312.4) node [anchor=north west][inner sep=0.75pt]    {$c$};
\draw (507,312.4) node [anchor=north west][inner sep=0.75pt]    {$c$};
\draw (557,312.4) node [anchor=north west][inner sep=0.75pt]    {$c$};
\draw (310,10) node [anchor=north west][inner sep=0.75pt]    {$u$};
\draw (310,410) node [anchor=north west][inner sep=0.75pt]    {$v$};

\end{tikzpicture}
    \caption{$c$-segments running between consecutive radial $b$-corridors}
    \label{fig:connected radial b-corridors}
\end{figure}
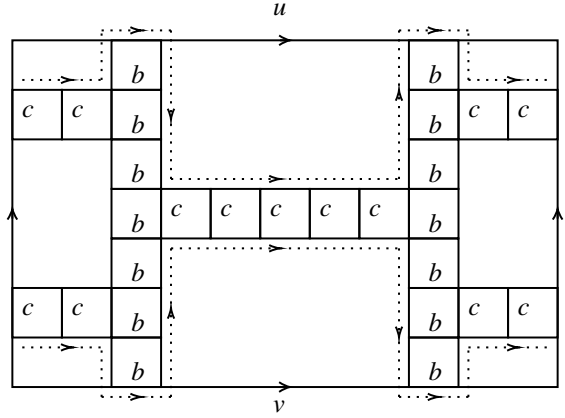

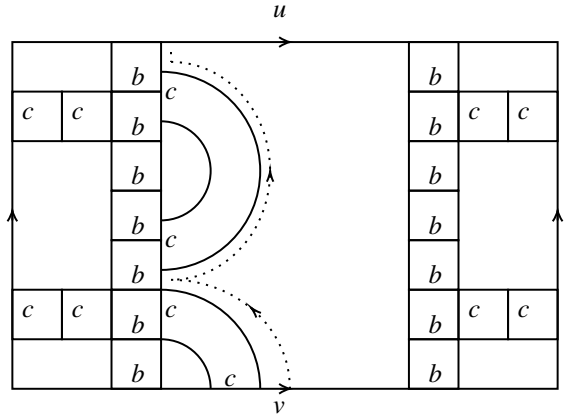
\begin{figure}
    \centering

\tikzset{every picture/.style={line width=0.75pt}} 

\begin{tikzpicture}[x=0.5pt,y=0.5pt,yscale=-.75,xscale=.75]

\draw   (50,50) -- (600,50) -- (600,400) -- (50,400) -- cycle ;
\draw    (50,230) -- (50,222) ;
\draw [shift={(50,220)}, rotate = 90] [color={rgb, 255:red, 0; green, 0; blue, 0 }  ][line width=0.75]    (10.93,-4.9) .. controls (6.95,-2.3) and (3.31,-0.67) .. (0,0) .. controls (3.31,0.67) and (6.95,2.3) .. (10.93,4.9)   ;
\draw    (600,230) -- (600,222) ;
\draw [shift={(600,220)}, rotate = 90] [color={rgb, 255:red, 0; green, 0; blue, 0 }  ][line width=0.75]    (10.93,-4.9) .. controls (6.95,-2.3) and (3.31,-0.67) .. (0,0) .. controls (3.31,0.67) and (6.95,2.3) .. (10.93,4.9)   ;
\draw   (450,350) -- (500,350) -- (500,400) -- (450,400) -- cycle ;
\draw   (450,300) -- (500,300) -- (500,350) -- (450,350) -- cycle ;
\draw   (450,250) -- (500,250) -- (500,300) -- (450,300) -- cycle ;
\draw   (450,200) -- (500,200) -- (500,250) -- (450,250) -- cycle ;
\draw    (320,400) -- (328,400) ;
\draw [shift={(330,400)}, rotate = 180] [color={rgb, 255:red, 0; green, 0; blue, 0 }  ][line width=0.75]    (10.93,-4.9) .. controls (6.95,-2.3) and (3.31,-0.67) .. (0,0) .. controls (3.31,0.67) and (6.95,2.3) .. (10.93,4.9)   ;
\draw    (320,50) -- (328,50) ;
\draw [shift={(330,50)}, rotate = 180] [color={rgb, 255:red, 0; green, 0; blue, 0 }  ][line width=0.75]    (10.93,-4.9) .. controls (6.95,-2.3) and (3.31,-0.67) .. (0,0) .. controls (3.31,0.67) and (6.95,2.3) .. (10.93,4.9)   ;
\draw   (450,150) -- (500,150) -- (500,200) -- (450,200) -- cycle ;
\draw   (450,100) -- (500,100) -- (500,150) -- (450,150) -- cycle ;
\draw   (450,50) -- (500,50) -- (500,100) -- (450,100) -- cycle ;
\draw   (150,350) -- (200,350) -- (200,400) -- (150,400) -- cycle ;
\draw   (150,300) -- (200,300) -- (200,350) -- (150,350) -- cycle ;
\draw   (150,250) -- (200,250) -- (200,300) -- (150,300) -- cycle ;
\draw   (150,200) -- (200,200) -- (200,250) -- (150,250) -- cycle ;
\draw   (150,150) -- (200,150) -- (200,200) -- (150,200) -- cycle ;
\draw   (150,100) -- (200,100) -- (200,150) -- (150,150) -- cycle ;
\draw   (150,50) -- (200,50) -- (200,100) -- (150,100) -- cycle ;
\draw   (50,100) -- (100,100) -- (100,150) -- (50,150) -- cycle ;
\draw   (100,100) -- (150,100) -- (150,150) -- (100,150) -- cycle ;
\draw   (50,300) -- (100,300) -- (100,350) -- (50,350) -- cycle ;
\draw   (100,300) -- (150,300) -- (150,350) -- (100,350) -- cycle ;
\draw   (500,100) -- (550,100) -- (550,150) -- (500,150) -- cycle ;
\draw   (550,100) -- (600,100) -- (600,150) -- (550,150) -- cycle ;
\draw   (500,300) -- (550,300) -- (550,350) -- (500,350) -- cycle ;
\draw   (550,300) -- (600,300) -- (600,350) -- (550,350) -- cycle ;
\draw  [draw opacity=0] (200,350) .. controls (200,350) and (200,350) .. (200,350) .. controls (227.61,350) and (250,372.39) .. (250,400) -- (200,400) -- cycle ; \draw   (200,350) .. controls (200,350) and (200,350) .. (200,350) .. controls (227.61,350) and (250,372.39) .. (250,400) ;  
\draw  [draw opacity=0] (200,300) .. controls (200,300) and (200,300) .. (200,300) .. controls (255.23,300) and (300,344.77) .. (300,400) -- (200,400) -- cycle ; \draw   (200,300) .. controls (200,300) and (200,300) .. (200,300) .. controls (255.23,300) and (300,344.77) .. (300,400) ;  
\draw  [draw opacity=0] (200,130) .. controls (200,130) and (200,130) .. (200,130) .. controls (227.61,130) and (250,152.39) .. (250,180) .. controls (250,207.61) and (227.61,230) .. (200,230) -- (200,180) -- cycle ; \draw   (200,130) .. controls (200,130) and (200,130) .. (200,130) .. controls (227.61,130) and (250,152.39) .. (250,180) .. controls (250,207.61) and (227.61,230) .. (200,230) ;  
\draw  [draw opacity=0] (200,80) .. controls (255.23,80) and (300,124.77) .. (300,180) .. controls (300,235.23) and (255.23,280) .. (200,280) -- (200,180) -- cycle ; \draw   (200,80) .. controls (255.23,80) and (300,124.77) .. (300,180) .. controls (300,235.23) and (255.23,280) .. (200,280) ;  
\draw  [draw opacity=0][dash pattern={on 0.84pt off 2.51pt}] (210,290) .. controls (276.27,290) and (330,339.25) .. (330,400) -- (210,400) -- cycle ; \draw  [dash pattern={on 0.84pt off 2.51pt}] (210,290) .. controls (276.27,290) and (330,339.25) .. (330,400) ;  
\draw    (295,322.33) -- (292.08,319.41) ;
\draw [shift={(290.67,318)}, rotate = 45] [color={rgb, 255:red, 0; green, 0; blue, 0 }  ][line width=0.75]    (10.93,-3.29) .. controls (6.95,-1.4) and (3.31,-0.3) .. (0,0) .. controls (3.31,0.3) and (6.95,1.4) .. (10.93,3.29)   ;
\draw  [draw opacity=0][dash pattern={on 0.84pt off 2.51pt}] (210,70) .. controls (210,70) and (210,70) .. (210,70) .. controls (265.23,70) and (310,119.25) .. (310,180) .. controls (310,240.75) and (265.23,290) .. (210,290) -- (210,180) -- cycle ; \draw  [dash pattern={on 0.84pt off 2.51pt}] (210,70) .. controls (210,70) and (210,70) .. (210,70) .. controls (265.23,70) and (310,119.25) .. (310,180) .. controls (310,240.75) and (265.23,290) .. (210,290) ;  
\draw    (310,185) -- (310,182) ;
\draw [shift={(310,180)}, rotate = 90] [color={rgb, 255:red, 0; green, 0; blue, 0 }  ][line width=0.75]    (10.93,-3.29) .. controls (6.95,-1.4) and (3.31,-0.3) .. (0,0) .. controls (3.31,0.3) and (6.95,1.4) .. (10.93,3.29)   ;
\draw  [dash pattern={on 0.84pt off 2.51pt}]  (210,60) -- (210,70) ;

\draw (467,222.4) node [anchor=north west][inner sep=0.75pt]    {$b$};
\draw (467,272.4) node [anchor=north west][inner sep=0.75pt]    {$b$};
\draw (467,322.4) node [anchor=north west][inner sep=0.75pt]    {$b$};
\draw (467,372.4) node [anchor=north west][inner sep=0.75pt]    {$b$};
\draw (467,72.4) node [anchor=north west][inner sep=0.75pt]    {$b$};
\draw (467,122.4) node [anchor=north west][inner sep=0.75pt]    {$b$};
\draw (467,172.4) node [anchor=north west][inner sep=0.75pt]    {$b$};
\draw (167,222.4) node [anchor=north west][inner sep=0.75pt]    {$b$};
\draw (167,272.4) node [anchor=north west][inner sep=0.75pt]    {$b$};
\draw (167,322.4) node [anchor=north west][inner sep=0.75pt]    {$b$};
\draw (167,372.4) node [anchor=north west][inner sep=0.75pt]    {$b$};
\draw (167,72.4) node [anchor=north west][inner sep=0.75pt]    {$b$};
\draw (167,122.4) node [anchor=north west][inner sep=0.75pt]    {$b$};
\draw (167,172.4) node [anchor=north west][inner sep=0.75pt]    {$b$};
\draw (57,112.4) node [anchor=north west][inner sep=0.75pt]    {$c$};
\draw (107,112.4) node [anchor=north west][inner sep=0.75pt]    {$c$};
\draw (57,312.4) node [anchor=north west][inner sep=0.75pt]    {$c$};
\draw (107,312.4) node [anchor=north west][inner sep=0.75pt]    {$c$};
\draw (507,112.4) node [anchor=north west][inner sep=0.75pt]    {$c$};
\draw (557,112.4) node [anchor=north west][inner sep=0.75pt]    {$c$};
\draw (507,312.4) node [anchor=north west][inner sep=0.75pt]    {$c$};
\draw (557,312.4) node [anchor=north west][inner sep=0.75pt]    {$c$};
\draw (201,92.4) node [anchor=north west][inner sep=0.75pt]    {$c$};
\draw (201,312.4) node [anchor=north west][inner sep=0.75pt]    {$c$};
\draw (201,242.4) node [anchor=north west][inner sep=0.75pt]    {$c$};
\draw (261,382.4) node [anchor=north west][inner sep=0.75pt]    {$c$};
\draw (310,10) node [anchor=north west][inner sep=0.75pt]    {$u$};
\draw (310,410) node [anchor=north west][inner sep=0.75pt]    {$v$};

\end{tikzpicture}

    \caption{radial $b$-corridors not connected by a $c$-segment}
    \label{fig:disconnected radial b-corridors}
\end{figure}

The final bound we need for Theorem~\ref{MainTheorem} is:

\begin{prop} \label{lower bound on anng4 prop} $2^{n^2}\lesssim \Ann_{G_4}(n)$.
\end{prop}

\begin{proof}
    Suppose $n \in \N$. Let $u=b^{-1}c^{n^2}s$,  $v=b^{-1}s$, and $w = a^{-n^2}d^{n^2}$. 
  We claim that in $G_4$, 
     $$uw \ = \ b^{-1} c^{n^2} s \, a^{-n^2}   d^{n^2}  \ = \ b^{-1} c^{n^2} a^{-n^2}   d^{n^2} s  \ = \ a^{-n^2}  b^{-1} d^{n^2} s \ = \ a^{-n^2}d^{n^2} \, b^{-1}s \ = \ wv,$$ and so $u$ and $v$ are conjugate in $G_4$. 
        The second of these  equalities holds because $a s a^{-1} =s^2$ and $d s d^{-1} =s^2$ in $G_4$ imply that for all $k \in \N$ we have $sa^{-k} =  a^{-k} s^{2^k}$ and $d^{k} s =   s^{2^k} d^{k}$, and therefore $s$ commutes with  $a^{-k} d^k$.   The third uses  $[a,b] = a b a^{-1} b^{-1} = c$ and $[a,c]=[b,c]=1$, from which it follows that $a^{-1} b^{-1}  = b^{-1} c a^{-1}$ and then that $a^{-k} b^{-1}  = b^{-1} c^k a^{-k}$ for all $k \in \Z$.  The fourth uses $[b,d]=1$.    
 
 Also, $|u|_{G_4}\lesssim n$  and $$\Area(u  \, (b^{-1}[a^n,b^n]s)^{-1}) \ \leq \  \Area(b^{-1} c^{n^2}  \, [a^n,b^n]^{-1} b) \ \leq  \ n^3,$$ because $c^{n^2}=[a^n,b^n]$ in $G_4$ as a consequence of apply defining relations $[a,b] = c$ and $[a,c]=1$ and $[b,c]=1$ at most $n^3$ times. So, by Lemma \ref{Gluing van Kampen to annular diagram}, it   suffices to show that $2^{n^2-1} \leq \Ann(u,v)$. 

Well, suppose $\Omega$ is any annular diagram  for $u\sim v$.  A $b$-corridor $\beta$ connects the edge labelled by the $b^{-1}$ in the $u$-boundary component to the   $b^{-1}$ in the $v$-boundary component. Let $w$ be the word on $a$, $c$, and $d$ read along one side of that $b$-corridor, so that cutting the diagram along that side gives a van~Kampen diagram $\Delta$ for $uwv^{-1}w^{-1}$---see Figure~\ref{vK diagram concentric d-rings}.

\begin{figure}[ht]

\tikzset{every picture/.style={line width=0.75pt}} 

\begin{tikzpicture}[x=0.5pt,y=0.5pt,yscale=-1,xscale=1]

\draw   (50,50) -- (600,50) -- (600,400) -- (50,400) -- cycle ;
\draw    (50,230) -- (50,222) ;
\draw [shift={(50,220)}, rotate = 90] [color={rgb, 255:red, 0; green, 0; blue, 0 }  ][line width=0.75]    (10.93,-3.29) .. controls (6.95,-1.4) and (3.31,-0.3) .. (0,0) .. controls (3.31,0.3) and (6.95,1.4) .. (10.93,3.29)   ;
\draw    (600,230) -- (600,222) ;
\draw [shift={(600,220)}, rotate = 90] [color={rgb, 255:red, 0; green, 0; blue, 0 }  ][line width=0.75]    (10.93,-3.29) .. controls (6.95,-1.4) and (3.31,-0.3) .. (0,0) .. controls (3.31,0.3) and (6.95,1.4) .. (10.93,3.29)   ;
\draw   (100,350) -- (150,350) -- (150,400) -- (100,400) -- cycle ;
\draw   (100,300) -- (150,300) -- (150,350) -- (100,350) -- cycle ;
\draw   (100,250) -- (150,250) -- (150,300) -- (100,300) -- cycle ;
\draw   (100,200) -- (150,200) -- (150,250) -- (100,250) -- cycle ;
\draw   (100,150) -- (150,150) -- (150,200) -- (100,200) -- cycle ;
\draw   (100,100) -- (150,100) -- (150,150) -- (100,150) -- cycle ;
\draw   (100,50) -- (150,50) -- (150,100) -- (100,100) -- cycle ;
\draw   (250,350) -- (300,350) -- (300,400) -- (250,400) -- cycle ;
\draw   (250,300) -- (300,300) -- (300,350) -- (250,350) -- cycle ;
\draw   (250,250) -- (300,250) -- (300,300) -- (250,300) -- cycle ;
\draw   (250,200) -- (300,200) -- (300,250) -- (250,250) -- cycle ;
\draw   (250,150) -- (300,150) -- (300,200) -- (250,200) -- cycle ;
\draw   (250,100) -- (300,100) -- (300,150) -- (250,150) -- cycle ;
\draw   (250,50) -- (300,50) -- (300,100) -- (250,100) -- cycle ;

\draw    (570,50) -- (578,50) ;
\draw [shift={(580,50)}, rotate = 180] [color={rgb, 255:red, 0; green, 0; blue, 0 }  ][line width=0.75]    (10.93,-3.29) .. controls (6.95,-1.4) and (3.31,-0.3) .. (0,0) .. controls (3.31,0.3) and (6.95,1.4) .. (10.93,3.29)   ;
\draw    (50,350) -- (600,350) ;
\draw    (50,370) -- (50,378) ;
\draw [shift={(50,380)}, rotate = 270] [color={rgb, 255:red, 0; green, 0; blue, 0 }  ][line width=0.75]    (10.93,-3.29) .. controls (6.95,-1.4) and (3.31,-0.3) .. (0,0) .. controls (3.31,0.3) and (6.95,1.4) .. (10.93,3.29)   ;
\draw    (100,370) -- (100,378) ;
\draw [shift={(100,380)}, rotate = 270] [color={rgb, 255:red, 0; green, 0; blue, 0 }  ][line width=0.75]    (10.93,-3.29) .. controls (6.95,-1.4) and (3.31,-0.3) .. (0,0) .. controls (3.31,0.3) and (6.95,1.4) .. (10.93,3.29)   ;
\draw    (250,370) -- (250,378) ;
\draw [shift={(250,380)}, rotate = 270] [color={rgb, 255:red, 0; green, 0; blue, 0 }  ][line width=0.75]    (10.93,-3.29) .. controls (6.95,-1.4) and (3.31,-0.3) .. (0,0) .. controls (3.31,0.3) and (6.95,1.4) .. (10.93,3.29)   ;
\draw    (300,370) -- (300,378) ;
\draw [shift={(300,380)}, rotate = 270] [color={rgb, 255:red, 0; green, 0; blue, 0 }  ][line width=0.75]    (10.93,-3.29) .. controls (6.95,-1.4) and (3.31,-0.3) .. (0,0) .. controls (3.31,0.3) and (6.95,1.4) .. (10.93,3.29)   ;
\draw  [dash pattern={on 0.84pt off 2.51pt}]  (20,50) -- (20,400) ;
\draw  [dash pattern={on 0.84pt off 2.51pt}]  (630,50) -- (630,400) ;
\draw    (150,370) -- (150,378) ;
\draw [shift={(150,380)}, rotate = 270] [color={rgb, 255:red, 0; green, 0; blue, 0 }  ][line width=0.75]    (10.93,-3.29) .. controls (6.95,-1.4) and (3.31,-0.3) .. (0,0) .. controls (3.31,0.3) and (6.95,1.4) .. (10.93,3.29)   ;
\draw   (500,350) -- (550,350) -- (550,400) -- (500,400) -- cycle ;
\draw   (500,300) -- (550,300) -- (550,350) -- (500,350) -- cycle ;
\draw   (500,250) -- (550,250) -- (550,300) -- (500,300) -- cycle ;
\draw   (500,200) -- (550,200) -- (550,250) -- (500,250) -- cycle ;
\draw   (500,150) -- (550,150) -- (550,200) -- (500,200) -- cycle ;
\draw   (500,100) -- (550,100) -- (550,150) -- (500,150) -- cycle ;
\draw   (500,50) -- (550,50) -- (550,100) -- (500,100) -- cycle ;
\draw    (500,370) -- (500,378) ;
\draw [shift={(500,380)}, rotate = 270] [color={rgb, 255:red, 0; green, 0; blue, 0 }  ][line width=0.75]    (10.93,-3.29) .. controls (6.95,-1.4) and (3.31,-0.3) .. (0,0) .. controls (3.31,0.3) and (6.95,1.4) .. (10.93,3.29)   ;
\draw    (550,370) -- (550,378) ;
\draw [shift={(550,380)}, rotate = 270] [color={rgb, 255:red, 0; green, 0; blue, 0 }  ][line width=0.75]    (10.93,-3.29) .. controls (6.95,-1.4) and (3.31,-0.3) .. (0,0) .. controls (3.31,0.3) and (6.95,1.4) .. (10.93,3.29)   ;
\draw    (600,370) -- (600,378) ;
\draw [shift={(600,380)}, rotate = 270] [color={rgb, 255:red, 0; green, 0; blue, 0 }  ][line width=0.75]    (10.93,-3.29) .. controls (6.95,-1.4) and (3.31,-0.3) .. (0,0) .. controls (3.31,0.3) and (6.95,1.4) .. (10.93,3.29)   ;
\draw    (20,230) -- (20,222) ;
\draw [shift={(20,220)}, rotate = 90] [color={rgb, 255:red, 0; green, 0; blue, 0 }  ][line width=0.75]    (10.93,-3.29) .. controls (6.95,-1.4) and (3.31,-0.3) .. (0,0) .. controls (3.31,0.3) and (6.95,1.4) .. (10.93,3.29)   ;
\draw    (630,235) -- (630,227) ;
\draw [shift={(630,225)}, rotate = 90] [color={rgb, 255:red, 0; green, 0; blue, 0 }  ][line width=0.75]    (10.93,-3.29) .. controls (6.95,-1.4) and (3.31,-0.3) .. (0,0) .. controls (3.31,0.3) and (6.95,1.4) .. (10.93,3.29)   ;
\draw    (100,230) -- (100,222) ;
\draw [shift={(100,220)}, rotate = 90] [color={rgb, 255:red, 0; green, 0; blue, 0 }  ][line width=0.75]    (10.93,-3.29) .. controls (6.95,-1.4) and (3.31,-0.3) .. (0,0) .. controls (3.31,0.3) and (6.95,1.4) .. (10.93,3.29)   ;
\draw    (150,230) -- (150,222) ;
\draw [shift={(150,220)}, rotate = 90] [color={rgb, 255:red, 0; green, 0; blue, 0 }  ][line width=0.75]    (10.93,-3.29) .. controls (6.95,-1.4) and (3.31,-0.3) .. (0,0) .. controls (3.31,0.3) and (6.95,1.4) .. (10.93,3.29)   ;
\draw    (250,230) -- (250,222) ;
\draw [shift={(250,220)}, rotate = 90] [color={rgb, 255:red, 0; green, 0; blue, 0 }  ][line width=0.75]    (10.93,-3.29) .. controls (6.95,-1.4) and (3.31,-0.3) .. (0,0) .. controls (3.31,0.3) and (6.95,1.4) .. (10.93,3.29)   ;
\draw    (300,230) -- (300,222) ;
\draw [shift={(300,220)}, rotate = 90] [color={rgb, 255:red, 0; green, 0; blue, 0 }  ][line width=0.75]    (10.93,-3.29) .. controls (6.95,-1.4) and (3.31,-0.3) .. (0,0) .. controls (3.31,0.3) and (6.95,1.4) .. (10.93,3.29)   ;
\draw    (500,230) -- (500,222) ;
\draw [shift={(500,220)}, rotate = 90] [color={rgb, 255:red, 0; green, 0; blue, 0 }  ][line width=0.75]    (10.93,-3.29) .. controls (6.95,-1.4) and (3.31,-0.3) .. (0,0) .. controls (3.31,0.3) and (6.95,1.4) .. (10.93,3.29)   ;
\draw    (550,230) -- (550,222) ;
\draw [shift={(550,220)}, rotate = 90] [color={rgb, 255:red, 0; green, 0; blue, 0 }  ][line width=0.75]    (10.93,-3.29) .. controls (6.95,-1.4) and (3.31,-0.3) .. (0,0) .. controls (3.31,0.3) and (6.95,1.4) .. (10.93,3.29)   ;
\draw    (70,50) -- (78,50) ;
\draw [shift={(80,50)}, rotate = 180] [color={rgb, 255:red, 0; green, 0; blue, 0 }  ][line width=0.75]    (10.93,-3.29) .. controls (6.95,-1.4) and (3.31,-0.3) .. (0,0) .. controls (3.31,0.3) and (6.95,1.4) .. (10.93,3.29)   ;
\draw    (200,50) -- (208,50) ;
\draw [shift={(210,50)}, rotate = 180] [color={rgb, 255:red, 0; green, 0; blue, 0 }  ][line width=0.75]    (10.93,-3.29) .. controls (6.95,-1.4) and (3.31,-0.3) .. (0,0) .. controls (3.31,0.3) and (6.95,1.4) .. (10.93,3.29)   ;
\draw  [dash pattern={on 0.84pt off 2.51pt}]  (50,10) -- (600,10) ;
\draw    (320,10) -- (323,10) ;
\draw [shift={(325,10)}, rotate = 180] [color={rgb, 255:red, 0; green, 0; blue, 0 }  ][line width=0.75]    (10.93,-3.29) .. controls (6.95,-1.4) and (3.31,-0.3) .. (0,0) .. controls (3.31,0.3) and (6.95,1.4) .. (10.93,3.29)   ;
\draw  [dash pattern={on 0.84pt off 2.51pt}]  (50,440) -- (600,440) ;
\draw    (320,440) -- (323,440) ;
\draw [shift={(325,440)}, rotate = 180] [color={rgb, 255:red, 0; green, 0; blue, 0 }  ][line width=0.75]    (10.93,-3.29) .. controls (6.95,-1.4) and (3.31,-0.3) .. (0,0) .. controls (3.31,0.3) and (6.95,1.4) .. (10.93,3.29)   ;
\draw    (570,400) -- (578,400) ;
\draw [shift={(580,400)}, rotate = 180] [color={rgb, 255:red, 0; green, 0; blue, 0 }  ][line width=0.75]    (10.93,-3.29) .. controls (6.95,-1.4) and (3.31,-0.3) .. (0,0) .. controls (3.31,0.3) and (6.95,1.4) .. (10.93,3.29)   ;
\draw    (70,400) -- (78,400) ;
\draw [shift={(80,400)}, rotate = 180] [color={rgb, 255:red, 0; green, 0; blue, 0 }  ][line width=0.75]    (10.93,-3.29) .. controls (6.95,-1.4) and (3.31,-0.3) .. (0,0) .. controls (3.31,0.3) and (6.95,1.4) .. (10.93,3.29)   ;
\draw    (200,400) -- (208,400) ;
\draw [shift={(210,400)}, rotate = 180] [color={rgb, 255:red, 0; green, 0; blue, 0 }  ][line width=0.75]    (10.93,-3.29) .. controls (6.95,-1.4) and (3.31,-0.3) .. (0,0) .. controls (3.31,0.3) and (6.95,1.4) .. (10.93,3.29)   ;
\draw    (120,400) -- (128,400) ;
\draw [shift={(130,400)}, rotate = 180] [color={rgb, 255:red, 0; green, 0; blue, 0 }  ][line width=0.75]    (10.93,-3.29) .. controls (6.95,-1.4) and (3.31,-0.3) .. (0,0) .. controls (3.31,0.3) and (6.95,1.4) .. (10.93,3.29)   ;
\draw    (120,50) -- (128,50) ;
\draw [shift={(130,50)}, rotate = 180] [color={rgb, 255:red, 0; green, 0; blue, 0 }  ][line width=0.75]    (10.93,-3.29) .. controls (6.95,-1.4) and (3.31,-0.3) .. (0,0) .. controls (3.31,0.3) and (6.95,1.4) .. (10.93,3.29)   ;
\draw    (270,50) -- (278,50) ;
\draw [shift={(280,50)}, rotate = 180] [color={rgb, 255:red, 0; green, 0; blue, 0 }  ][line width=0.75]    (10.93,-3.29) .. controls (6.95,-1.4) and (3.31,-0.3) .. (0,0) .. controls (3.31,0.3) and (6.95,1.4) .. (10.93,3.29)   ;
\draw    (520,50) -- (528,50) ;
\draw [shift={(530,50)}, rotate = 180] [color={rgb, 255:red, 0; green, 0; blue, 0 }  ][line width=0.75]    (10.93,-3.29) .. controls (6.95,-1.4) and (3.31,-0.3) .. (0,0) .. controls (3.31,0.3) and (6.95,1.4) .. (10.93,3.29)   ;
\draw    (270,400) -- (278,400) ;
\draw [shift={(280,400)}, rotate = 180] [color={rgb, 255:red, 0; green, 0; blue, 0 }  ][line width=0.75]    (10.93,-3.29) .. controls (6.95,-1.4) and (3.31,-0.3) .. (0,0) .. controls (3.31,0.3) and (6.95,1.4) .. (10.93,3.29)   ;
\draw    (520,400) -- (528,400) ;
\draw [shift={(530,400)}, rotate = 180] [color={rgb, 255:red, 0; green, 0; blue, 0 }  ][line width=0.75]    (10.93,-3.29) .. controls (6.95,-1.4) and (3.31,-0.3) .. (0,0) .. controls (3.31,0.3) and (6.95,1.4) .. (10.93,3.29)   ;

\draw (115,79.4) node [anchor=north west][inner sep=0.75pt]    {$d^{\varepsilon _{1}}$};
\draw (115,129.4) node [anchor=north west][inner sep=0.75pt]    {$d^{\varepsilon _{1}}$};
\draw (115,179.4) node [anchor=north west][inner sep=0.75pt]    {$d^{\varepsilon _{1}}$};
\draw (115,229.4) node [anchor=north west][inner sep=0.75pt]    {$d^{\varepsilon _{1}}$};
\draw (115,279.4) node [anchor=north west][inner sep=0.75pt]    {$d^{\varepsilon _{1}}$};
\draw (115,329.4) node [anchor=north west][inner sep=0.75pt]    {$d^{\varepsilon _{1}}$};
\draw (115,372.4) node [anchor=north west][inner sep=0.75pt]    {$d^{\varepsilon _{1}}$};
\draw (-10,172.4) node [anchor=north west][inner sep=0.75pt]    {\rotatebox{90}{$u=b^{-1} c^{n^{2}} s$}};
\draw (633,172.4) node [anchor=north west][inner sep=0.75pt]    {\rotatebox{90}{$v=b^{-1} s$}};
\draw (31,362.4) node [anchor=north west][inner sep=0.75pt]    {$b$};
\draw (607,362.4) node [anchor=north west][inner sep=0.75pt]    {$b$};
\draw (265,79.4) node [anchor=north west][inner sep=0.75pt]    {$d^{\varepsilon _{2}}$};
\draw (265,129.4) node [anchor=north west][inner sep=0.75pt]    {$d^{\varepsilon _{2}}$};
\draw (265,179.4) node [anchor=north west][inner sep=0.75pt]    {$d^{\varepsilon _{2}}$};
\draw (265,229.4) node [anchor=north west][inner sep=0.75pt]    {$d^{\varepsilon _{2}}$};
\draw (265,279.4) node [anchor=north west][inner sep=0.75pt]    {$d^{\varepsilon _{2}}$};
\draw (265,329.4) node [anchor=north west][inner sep=0.75pt]    {$d^{\varepsilon _{2}}$};
\draw (265,372.4) node [anchor=north west][inner sep=0.75pt]    {$d^{\varepsilon _{2}}$};
\draw (505,76.4) node [anchor=north west][inner sep=0.75pt]    {$d^{\varepsilon _{m-1}}$};
\draw (505,126.4) node [anchor=north west][inner sep=0.75pt]    {$d^{\varepsilon _{m-1}}$};
\draw (505,176.4) node [anchor=north west][inner sep=0.75pt]    {$d^{\varepsilon _{m-1}}$};
\draw (505,226.4) node [anchor=north west][inner sep=0.75pt]    {$d^{\varepsilon _{m-1}}$};
\draw (505,276.4) node [anchor=north west][inner sep=0.75pt]    {$d^{\varepsilon _{m-1}}$};
\draw (505,326.4) node [anchor=north west][inner sep=0.75pt]    {$d^{\varepsilon _{m-1}}$};
\draw (501,369.4) node [anchor=north west][inner sep=0.75pt]    {$d^{\varepsilon _{m-1}}$};
\draw (70,212.4) node [anchor=north west][inner sep=0.75pt]    {$s^{\gamma _{1}}$};
\draw (156,212.4) node [anchor=north west][inner sep=0.75pt]    {$s^{\xi _{1}}$};
\draw (221,212.4) node [anchor=north west][inner sep=0.75pt]    {$s^{\gamma _{2}}$};
\draw (306,212.4) node [anchor=north west][inner sep=0.75pt]    {$s^{\xi _{2}}$};
\draw (455,212.4) node [anchor=north west][inner sep=0.75pt]    {$s^{\gamma _{m-1}}$};
\draw (556,212.4) node [anchor=north west][inner sep=0.75pt]    {$s^{\xi _{m-1}}$};
\draw (61,22.4) node [anchor=north west][inner sep=0.75pt]    {$w_{1}$};
\draw (187,22.4) node [anchor=north west][inner sep=0.75pt]    {$w_{2}$};
\draw (564,22.4) node [anchor=north west][inner sep=0.75pt]    {$w_{m}$};
\draw (61,412.4) node [anchor=north west][inner sep=0.75pt]    {$w_{1}$};
\draw (187,412.4) node [anchor=north west][inner sep=0.75pt]    {$w_{2}$};
\draw (564,412.4) node [anchor=north west][inner sep=0.75pt]    {$w_{m}$};
\draw (68,162.4) node [anchor=north west][inner sep=0.75pt]    {$\Delta _{1}$};
\draw (191,162.4) node [anchor=north west][inner sep=0.75pt]    {$\Delta _{2}$};
\draw (385,172.4) node [anchor=north west][inner sep=0.75pt]    {$\cdots$};
\draw (568,162.4) node [anchor=north west][inner sep=0.75pt]    {$\Delta _{m}$};
\draw (311,-10.6) node [anchor=north west][inner sep=0.75pt]    {$w$};
\draw (311,450.4) node [anchor=north west][inner sep=0.75pt]    {$w$};

\end{tikzpicture}

 \caption{A van~Kampen diagram $\Delta$  for $uwv^{-1}w^{-1}$ per Proposition~\ref{lower bound on anng4 prop}.}
  \label{vK diagram concentric d-rings}
\end{figure}
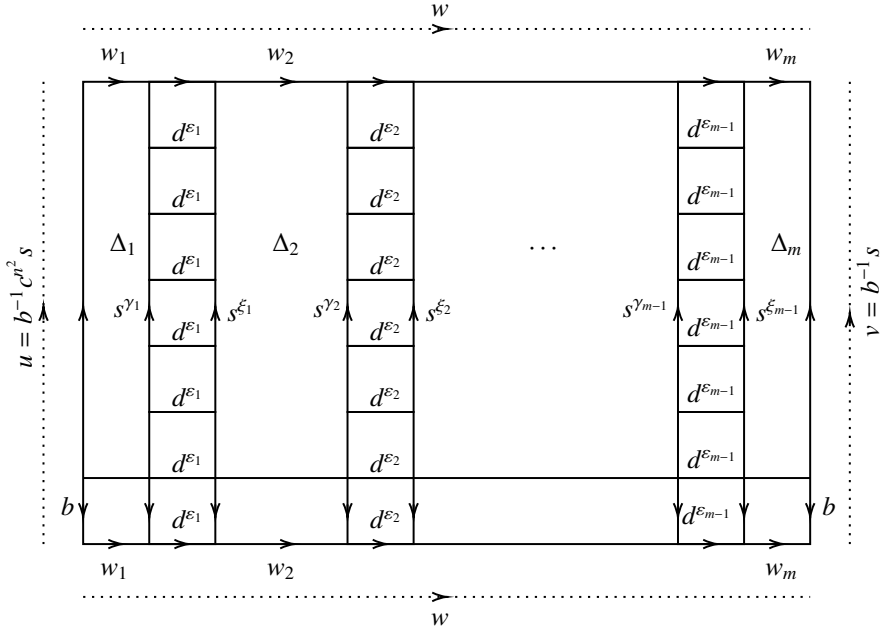

Killing $d$ and $s$ maps $G_4$ onto the Heisenberg group $H = \langle a, b, c \mid [a,c], [b,c], [a,b]c^{-1} \rangle$ and sends $w \mapsto a^{\alpha} c^{\gamma}$, where $\alpha = \exp_a(w)$ and $\gamma = \exp_c(w)$.  Thereby, $uw = wv$ in $G_4$ implies that $b^{-1}c^{n^2} a^{\alpha} c^{\gamma}  =  a^{\alpha} c^{\gamma} b^{-1}$ in $H$, and so $\alpha = -n^2$, because $\langle c \rangle \cong \Z$ is central and $a^{\alpha} b^{-1} = b^{-1} a^{\alpha}  c^{-\alpha}$ in $H$.  

Killing $b$ and $c$ maps $G_4$ on $B = \langle a, d, s\mid  asa^{-1} s^{-2}, dsd^{-1}   s^{-2} \rangle$ and deleting these letters takes $w$ to the word $\overline{w}$, which represents an element of the free subgroup $F(a,d)$ of $B$.      The equality $uw = wv$ in $G_4$ implies that  $s \overline{w} = \overline{w} s$ in $B$.  Then on mapping $B$ to $\langle a,  s\mid  asa^{-1} s^{-2} \rangle$ by sending $d \mapsto a$, we learn that $\exp(\overline{w}) =0$, and so $\exp_d(w) = - \exp_a(w)$, which we previously evaluated to be $n^2$.

Lemma~\ref{no b-rings} allows us to assume that $\Omega$ is reduced and contains no contractible $b$- or $d$-rings.  And then because there are no $d$-edges in $\partial \Omega$, any $d$-corridor $\delta$ in $\Omega$ must close up as a non-contractible $d$-ring.  We claim that $\delta$ can have at most one 2-cell in common with $\beta$.  Otherwise there would (by an \emph{innermost} argument) be a $d$-corridor that crosses $\beta$ twice so as to enclose a disc-subdiagram whose boundary is labelled by a word $\sigma s^{\nu}$ (where $\nu \in \Z$) such that  $\sigma= \sigma(a,c)$ follows part of one side of $\beta$ and  $s^{\nu}$ follows part of one side of a $d$-corridor.  But there can be no $b$-edges or $d$-edges in this subdiagram, so  $\sigma s^{\nu} =1$ in $\langle a,c,s \mid [a,c], asa^{-1}s^{-2} \rangle$.  Mapping to $\langle a, s \mid  asa^{-1}s^{-2} \rangle$ by killing $c$, we learn that $\nu=0$, and there are adjacent 2-cells, both labelled by $[b,d]$, where the $b$- and $d$-corridors cross, contrary to $\Omega$ being reduced.   

So the $d$-corridors  in $\Omega$ form a family of nested non-contractible $d$-rings, one for each $d^{\pm 1}$ in $w$.   Therefore $\Delta$ is as shown in   Figure~\ref{vK diagram concentric d-rings}: $w = w_1 d^{\epsilon_1} w_2 d^{\epsilon_2} \cdots d^{\epsilon_{m-1}} w_m$, where each $\epsilon_i = \pm 1$, $\nu_m =1$, and, for $i=1, \ldots, m-1$,  $w_i = w_i(a,c)$, the words along the sides of the $d$-corridors are  $b^{-1}s^{\nu_i}$ and $b^{-1}s^{\xi_i}$, and $\nu_i = 2^{\epsilon_i} \xi_i$, and $\Delta$ consists of one $b$-corridor and the $d$-corridors, and, in between, subdiagrams $\Delta_1$, \ldots, $\Delta_m$ over $\langle a,c,s \mid [a,c], asa^{-1}s^{-2} \rangle$.        

For $i =2, \ldots, m$, the boundary of the sub-diagram $\Delta_i$ is labelled by  $$\kappa_i  \ = \  s^{\xi_{i-1}} \, w_i(a,c) \, s^{-\nu_i}  \,   w_i(c^{-1} a,c)^{-1}.$$    Killing $s$ maps $\langle a,c,s \mid [a,c], asa^{-1}s^{-2} \rangle \onto \Z^2 = \langle a,c \rangle$, so $\exp_c(\kappa_i) =0$, and therefore $\exp_a(w_i) =0$.  
Killing $c$ maps $\langle a,c,s \mid [a,c], asa^{-1}s^{-2} \rangle \onto   \langle a,s  \mid  asa^{-1}s^{-2} \rangle$ and $w_i \mapsto 1$, and so $s^{\xi_{i-1}} =   s^{\nu_i}$ in $\langle a,s  \mid  asa^{-1}s^{-2} \rangle$, which implies that $\xi_{i-1} = \nu_i$.

So $\nu_i    = \  2^{\epsilon_i} \nu_{i+1}$  for $i=1, \ldots, m-1$, and because $ \exp_d(w)  =   \epsilon_1 + \cdots + \epsilon_{m-1}   =    n^2$ and $\nu_m  =   1$, we deduce that  $\nu_j = 2^{n^2}$  for some $j$.  A count of the 2-cells comprising the $d$-corridor along one side of which we read   $s^{\nu_j}$ gives $\Area(\Delta) \geq 2^{n^2-1}$, and therefore $\Ann(u,v) \geq 2^{n^2-1}$, as required.      
\end{proof}


\bibliographystyle{alpha}
\bibliography{bib}

\ni  {Conan Gillis and Timothy R.\ Riley} \rule{0mm}{6mm} \\
Department of Mathematics, 310 Malott Hall,  Cornell University, Ithaca, NY 14853, USA 
\\ {cg527@cornell.edu}, \
\href{https://math.cornell.edu/conan-gillis}{http://www.math.cornell.edu/conan-gillis/}
 \\ {tim.riley@math.cornell.edu},  \
\href{http://www.math.cornell.edu/~riley/}{http://www.math.cornell.edu/$\sim$riley/}

\end{document}